\documentclass[11pt]{amsart}

\usepackage{macros}

\def\sAd{\sA{\rm d}}

\usepackage[final]{pdfpages}

\numberwithin{equation}{section}

\begin{document}
\title{Higher Kac-Moody algebras and symmetries of holomorphic field theories}

\author{Owen Gwilliam}
\address{Department of Mathematics and Statistics \\
Lederle Graduate Research Tower, 1623D \\
University of Massachusetts Amherst \\
710 N. Pleasant Street}
\email{gwilliam@math.umass.edu}

\author{Brian Williams}
\address{Department of Mathematics, 
Northeastern University \\ 
567 Lake Hall \\ 
Boston, MA 02115 \\ U.S.A.}
\email{br.williams@northeastern.edu}

\begin{abstract}
We introduce a higher dimensional generalization of the affine Kac-Moody algebra using the language of factorization algebras. 
In particular, on any complex manifold there is a factorization algebra of "currents" associated to any Lie algebra. 
We classify local cocycles of these current algebras, and compare them to central extensions of higher affine algebras recently proposed by Faonte-Hennion-Kapranov. 
A central goal of this paper is to witness higher Kac-Moody algebras as symmetries of a class of holomorphic quantum field theories. 
In particular, we prove a generalization of the free field realization of an affine Kac-Moody algebra and also develop the theory of $q$-characters for this class of algebras in terms of factorization homology.
Finally, we exhibit the ``large N" behavior of higher Kac-Moody algebras and their relationship to symmetries of non-commutative field theories. 
\end{abstract}

\maketitle
\thispagestyle{empty}

\newpage
\tableofcontents

\newpage

\section*{}

The loop algebra $L\fg = \fg[z,z^{-1}]$, consisting of Laurent polynomials valued in a Lie algebra $\fg$,
admits a non-trivial central extension $\Hat{\fg}$ for each choice of invariant pairing on $\fg$.
This affine Lie algebra and its cousin, the Kac-Moody vertex algebra, are foundational objects in representation theory and conformal field theory. 
A natural question then arises: do there exists multivariable, or higher dimensional, generalizations of the affine Lie algebra and Kac-Moody vertex algebra? 

In this work, we pursue two independent yet related goals:
 
\begin{enumerate}
\item Use factorization algebras to study the (co)sheaf of Lie algebra-valued currents on complex manifolds, and their relationship to higher affine algebras;
\item Develop tools for understanding symmetries of {\em holomorphic field theory} in any dimension, that provide a systematic generalization of methods used in chiral conformal field theory on Riemann surfaces.
\end{enumerate}

Concretely, for every complex dimension $d$ and to every Lie algebra, we define a factorization algebra defined on all $d$-dimensional complex manifolds. 
There is also a version that works for an arbitrary principal bundle. 
When $d=1$, it is shown in \cite{CG1}, that this factorization algebra recovers the ordinary affine algebra by restricting the factorization algebra to the punctured complex line $C^*$. 
When $d > 1$, part of our main result is to show how the factorization algebra on $\CC^d \setminus \{0\}$ recovers a higher dimensional central extensions of $\fg$-valued functions on the punctured plane. 
A model for these ``higher affine algebras" has recently appeared in work of Faonte-Hennion-Kapranov \cite{FHK}, and we will give a systematic relationship between our approaches. 

By a standard procedure, there is a way of enhancing the affine algebra to a vertex algebra. 
The so-called Kac-Moody vertex algebra, as developed in \cite{IgorKM, KacVertex, BorcherdsVertex}, is important in its own right to representation theory and conformal field theory. 
In \cite{CG1} it is also shown how the holomorphic factorization algebra associated to a Lie algebra recovers this vertex algebra. 
The key point is that the OPE is encoded by the factorization product between disks embedded in $\CC$. 
Our proposed factorization algebra, then, provides a higher dimensional enhancement of this vertex algebra through the factorization product of balls or polydisks in $\CC^d$. 
This structure can be thought of as a holomorphic analog of an algebra over the operad of little $d$-disks.

It is the general philosophy of \cite{CG1,CG2} that every quantum field theory defines a factorization algebra of observables.
This perspective allows us to realize the higher Kac-Moody algebra inside of familiar higher dimensional field theories. 
In particular, this philosophy leads to higher dimensional analogs of free field realization via a quantum field theory called the $\beta\gamma$ system, which is defined on any complex manifold. 

In complex dimension one, a vertex algebra is a gadget associated to any conformal field theory that completely determines the algebra of local operators.  
More recently, vertex algebras have been extracted from higher dimensional field theories, such as $4$-dimensional gauge theories~\cite{Beem1,Beem2}. 

A future direction, which we do not undertake here, would be to use these higher dimensional vertex algebras as a more refined invariant of the quantum field theory. 

Before embarking on our main results, we take some time to motivate higher dimensional current algebras from two different perspectives. 

\subsection*{A view from physics}

In conformal field theory, the Kac-Moody algebra appears as the symmetry of a system with an action by a Lie algebra. 
A generic example is a flavor symmetry of a field theory where the matter takes values in some representation.
In ordinary 2d chiral CFT, the central extension appears as the failure of the classical Lie bracket on $\fg$-valued currents to be compatible with the OPE. 
This is measured by the charge anomaly, which occurs as a $2$-point function in the CFT. 

This paper is concerned with symmetries for holomorphic theories in any complex dimension. 
Classically, the story is completely analogous to the ordinary picture in chiral CFT: for holomorphic theories, the action by a Lie algebra is enhanced to a symmetry by an infinite dimensional Lie algebra of currents on the $S^{2d-1}$-modes of the holomorphic theory. 
This current algebra is an algebraic version of the sphere mapping space $\Map(S^{2d-1}, \fg)$. 

In any dimension, there is a chiral charge anomaly for the class of holomorphic field theories that we study, which measures the failure of quantizing the classical symmetry. 
In complex dimension $2$ (real dimension $4$), for instance, the anomaly is a holomorphic version of the Adler-Bardeen-Jackiw anomaly \cite{Adler, BJ}.
In terms of supersymmetric field theory, the anomaly is the holomorphic twist of the Konishi anomaly \cite{Konishi}. 
For a general form of the anomaly in our situation, we refer to Section \ref{sec: qft}, where we consider a general class of theories with ``holomorphic matter". 

Throughout this paper, we use ideas and techniques from the Batalin-Vilkovisky formalism, as articulated by Costello, and from the theory of factorization algebras, following \cite{CG1,CG2}.
In this introduction, however, we will try to explain the key objects and constructions with a light touch,
in a way that does not require familiarity with that formalism,
merely comfort with basic complex geometry and ideas of quantum field theory.

A running example is the following version of the $\beta\gamma$ system.

Let $X$ be a complex $d$-dimensional manifold.
Let $G$ be a complex algebraic group, such as $GL_n(\CC)$, 
and let $P \to X$ be a holomorphic principal $G$-bundle.
Fix a finite-dimensional $G$-representation $V$ and let $V^\vee$ denote the dual vector space with the natural induced $G$-action.
Let $\cV \to X$ denote the holomorphic associated bundle $P \times^G V$, 
and let $\cV^! \to X$ denote the holomorphic bundle $K_X \otimes \cV^\vee$,
where $\cV^* \to X$ is the holomorphic associated bundle $P \times^G V^*$.
Note that there is a natural fiberwise pairing
\[
\langle-,-\rangle: \cV \otimes \cV^! \to K_X \footnote{The shriek denotes the Serre dual, $\sV^! = K_X \tensor V^\vee$.}
\]
arising from the evaluation pairing between $V$ and~$V^\vee$.

The field theory involves fields $\gamma$, for a smooth section of $\cV$, and $\beta$, for a smooth section of $\Omega^{0,d-1} \tensor \sV^\vee$.
Here, $\sV^\vee$ denotes the dual bundle. 
The action functional is
\[
S(\beta,\gamma) = \int_X \langle \beta, \dbar \gamma \rangle,
\]
so that the equations of motion are
\[
\dbar \gamma = 0 = \dbar \beta.
\]
Thus, the classical theory is manifestly holomorphic: it picks out holomorphic sections of $\cV$ and $\cV^!$ as solutions.

The theory also enjoys a natural symmetry with respect to $G$,
arising from the $G$-action on $\cV$ and $\cV^!$.
For instance, if $\dbar \gamma = 0$ and $g \in G$, then the section $g \gamma$ is also holomorphic.
In fact, there is a local symmetry as well.
Let ${\rm ad}(P) \to X$ denote the Lie algebra-valued bundle $P \times^G \fg \to X$ arising from the adjoint representation $\ad(G)$.
Then a holomorphic section $f$ of $\ad(P)$ acts on a holomorphic section $\gamma$ of $\cV$,
and 
\[
\dbar(f \gamma) =  (\dbar f) \gamma + f \dbar \gamma = 0,
\]
so that the sheaf of holomorphic sections of $\ad(P)$ encodes a class of local symmetries of this classical theory.

If one takes a BV/BRST approach to field theory, as we will in this paper,
then one works with a cohomological version of fields and symmetries.
For instance, it is natural to view the classical fields as consisting of the graded vector space of Dolbeault forms
\[
\gamma \in \Omega^{0,*}(X,\cV) \quad \text{and} \quad \beta \in \Omega^{0,*}(X, \cV^!) \cong \Omega^{d,*}(X, \cV^*),
\]
but using the same action functional, extended in the natural way.
As we are working with a free theory and hence have only a quadratic action,
the equations of motion are linear and can be viewed as equipping the fields with the differential $\dbar$.
In this sense, the sheaf $\cE$ of solutions to the equations of motion can be identified with the elliptic complex that assigns to an open set $U \subset X$, the complex
\[
\cE(U) = \Omega^{0,*}(U,\cV) \oplus \Omega^{0,*}(U, \cV^!),
\]
with $\dbar$ as the differential.
This dg approach is certainly appealing from the perspective of complex geometry,
where one routinely works with the Dolbeault complex of a holomorphic bundle.

It is natural then to encode the local symmetries in the same way.
Let $\sAd(P)$ denote the Dolbeault complex of ${\rm ad}(P)$ viewed as a sheaf.
That is, it assigns to the open set $U \subset X$, the dg Lie algebra 
\[
\sAd(P)(U) = \Omega^{0,*}(U,\ad(P))
\]
with differential $\dbar$ for this bundle.
By construction, $\sAd(P)$ acts on $\cE$.
In words, $\cE$ is a sheaf of dg modules for the sheaf of dg Lie algebra~$\sAd(P)$.

So far, we have simply lifted the usual discussion of symmetries to a dg setting,
using standard tools of complex geometry.
We now introduce a novel maneuver that is characteristic of the BV/factorization package of~\cite{CG1,CG2}.

The idea is to work with compactly supported sections of $\sAd(P)$, 
i.e., to work with the precosheaf $\sAd(P)_c$ of dg Lie algebras that assigns to an open $U$,
the dg Lie algebra
\[
\sAd(P)_c(U) = \Omega^{0,*}_c(U,\ad(P)).
\]
The terminology {\em precosheaf} encodes the fact that there is natural way to extend a section supported in $U$ to a larger open $V \supset U$ (namely, extend by zero),
and so one has a functor $\sAd(P) \colon {\rm Opens}(X) \to {\rm Alg}_{\rm Lie}$.

There are several related reasons to consider compact support.\footnote{In Section \ref{sec: fact} we extract factorization algebras from $\sAd(P)_c$,
and then extract associative and vertex algebras of well-known interest.
We postpone discussions within that framework till that section.}
First, it is common in physics to consider compactly-supported modifications of a field.
Recall the variational calculus, where one extracts the equations of motion by working with precisely such first-order perturbations.
Hence, it is natural to focus on such symmetries as well.
Second, one could ask how such compactly supported actions of $\sAd(P)$ affect observables.
More specifically, one can ask about the charges of the theory with respect to this local symmetry.\footnote{We remark that it is precisely this relationship with traditional physical terminology of currents and charges that led de Rham to use {\em current} to mean a distributional section of the de Rham complex.}
Third---and this reason will become clearer in a moment---the anomaly that appears when trying to quantize this symmetry are naturally local in $X$, and hence it is encoded by a kind of Lagrangian density $L$ on sections of $\sAd(P)$.
Such a density only defines a functional on compactly supported sections,
since when evaluated a noncompactly supported section $f$, the density $L(f)$ may be non-integrable.
Thus $L$ determines a central extension of $\sAd(P)_c$ as a precosheaf of dg Lie algebras,
but not as a sheaf.\footnote{We remark that to stick with sheaves, one must turn to quite sophisticated tools \cite{WittenGr, GetzlerGM, ManBeilSch} that can be tricky to interpret, much less generalize to higher dimension, whereas the cosheaf-theoretic version is quite mundane and easy to generalize, as we'll see.}

Let us sketch how to make these reasons explicit.
The first step is to understand how $\sAd(P)_c$ acts on the observables of this theory.

Modulo functional analytic issues,
we say that the observables of this classical theory are the commutative dg algebra
\[
(\Sym(\Omega^{0,*}(X,\cV)^* \oplus \Omega^{0,*}(X, \cV^!)^*), \dbar),
\]
i.e., the polynomial functions on $\cE(X)$.
More accurately, we work with a commutative dg algebra essentially generated by the continuous linear functionals on $\cE(X)$, 
which are compactly supported distributional sections of certain Dolbeault complexes ({\it aka} Dolbeault currents).
We could replace $X$ by any open set $U \subset X$, 
in which case the observables with support in $U$ arise from such distributions supported in $U$.
We denote this commutative dg algebra by $\Obs^{cl}(U)$.
Since observables on an open $U$ extend to observables on a larger open $V \supset U$,
we recognize that $\Obs^{cl}$ forms a precosheaf.

Manifestly, $\sAd(P)_c(U)$ acts on $\Obs^{cl}(U)$,
by precomposing with its action on fields.
Moreover, these actions are compatible with the extension maps of the precosheaves,
so that $\Obs^{cl}$ is a module for $\sAd(P)_c$ in precosheaves of cochain complexes.
This relationship already exhibits why one might choose to focus on $\sAd(P)_c$,
as it naturally intertwines with the structure of the observables.

But Noether's theorem provides a further reason,
when understood in the context of the BV formalism.
The idea is that $\Obs^{cl}$ has a Poisson bracket $\{-,-\}$ of degree 1
(although there are some issues with distributions here that we suppress for the moment).
Hence one can ask to realize the action of $\sAd(P)_c$ via the Poisson bracket.
In other words, we ask to find a map of (precosheaves of) dg Lie algebras
\[
J \colon \sAd(P)_c \to \Obs^{cl}[-1]
\]
such that for any $f \in \sAd(P)_c(U)$ and $F \in \Obs^{cl}(U)$,
we have
\[
f \cdot F = \{J(f),F\}.
\]
Such a map would realize every symmetry as given by an observable,
much as in Hamiltonian mechanics.

In this case, there is such a map:
\[
J(f)(\gamma,\beta) = \int_U \langle\beta, f \gamma\rangle.
\]
This functional is local, and it is natural to view it as describing the ``minimal coupling'' between our free $\beta\gamma$ system and a kind of gauge field implicit in $\sAd(P)$.
This construction thus shows again that it is natural to work with compactly supported sections of $\sAd(P)$,
since it allows one to encode the Noether map in a natural way.
We call $\sAd(P)_c$ the Lie algebra of {\em classical currents} as we have explained how, via $J$, we realize these symmetries as classical observables.

\begin{rmk}
We remark that it is not always possible to produce such a Noether map,
but the obstruction always determines a central extension of $\sAd(P)_c$ as a precosheaf of dg Lie algebras,
and one can then produce such a map to the classical observables.
\end{rmk}

In the BV formalism, quantization amounts to a deformation of the differential on $\Obs^{cl}$,
where the deformation is required to satisfy certain properties.
Two conditions are preeminent:
\begin{itemize}
\item the differential satisfies a {\em quantum master equation}, which ensures that $\Obs^q(U)[-1]$ is still a dg Lie algebra via the bracket,\footnote{Again, we are suppressing---for the moment important---issues about renormalization, which will play a key role when we get to the real work.} and
\item it respects support of observables so that $\Obs^q$ is still a precosheaf.
\end{itemize}
The first condition is more or less what  BV quantization means, 
whereas the second is a version of the locality of field theory.

We can now ask whether the Noether map $J$ determines a map of precosheaves of dg Lie algebras from $\sAd(P)_c$ to $\Obs^q[-1]$.
Since the Lie bracket has not changed on the observables, 
the only question is where $J$ is a cochain map for the new differential $\d^q$
If we write $\d^q = \d^{cl} + \hbar \Delta$,\footnote{By working with smeared observables, one really can work with the naive BV Laplacian $\Delta$. Otherwise, one must take a little more care.} then 
\[
[\d,J] = \hbar \Delta \circ J.
\]
Naively---i.e., ignoring renormalization issues---this term is the functional $ob$ on $\sAd(P)_c$ given 
\[
ob(f) = \int \langle f K_\Delta \rangle,
\]
where $K_\Delta$ is the integral kernel for the identity with respect to the pairing $\langle-,-\rangle$.
(It encodes a version of the trace of $f$ over $\cE$.)
This obstruction, when examined with care, provides a holomorphic analogue of the ABJ and Konishi anomalies.

This functional $ob$ is a cocycle in Lie algebra cohomology for $\sAd(P)$ and hence determines a central extension $\widehat{\sAd(P)}_c$ as precosheaves of dg Lie algebras.
It is the Lie algebra of {\em quantum} currents, as there is a lift of $J$ to a map $J^q$ out of this extension to the quantum observables.

\subsection*{A view from geometry}

There is also a strong motivation for the algebras we consider from the perspective of the geometry of mapping spaces. 
There is an embedding $\fg[z,z^{-1}] \hookrightarrow C^\infty(S^1) \tensor \fg = {\rm Map}(S^1, \fg)$, induced by the embedding of algebraic functions on punctured affine line inside of smooth functions on $S^1$. 
Thus, a natural starting point for $d$-dimensional affine algebras is the ``sphere algebra" 
\beqn\label{mapping space}
{\rm Map}(S^{2d-1}, \fg) ,
\eeqn
where we view $S^{2d-1}$ sitting inside punctured affine space~$\pAA^d = \CC^d \setminus \{0\}$. 

When $d=1$, affine algebras are given by extensions $L\fg$ prescribed by a $2$-cocycle involving the algebraic residue pairing. 
Note that this cocycle is {\em not} pulled back from any cocycle on $\sO_{\rm alg}(\AA^1) \tensor \fg = \fg[z]$. 
%Now, consider algebraic functions on the punctured $d$-dimensional affine space $\AA^{d \times}$.

When $d > 1$, Hartog's theorem implies that the space of holomorphic functions on punctured affine space is the same as the space of holomorphic functions on affine space.
The same holds for algebraic functions, so that $\sO_{\rm alg}(\pAA^{d}) \tensor \fg = \sO_{\rm alg}(\AA^d) \tensor \fg$. 
In particular, the naive generalization $\sO_{\rm alg}(\pAA^{d}) \tensor \fg$ of (\ref{mapping space}) has no interesting central extensions. 
However, in contrast with the punctured line, the punctured affine space $\pAA^{d}$ has interesting higher cohomology. 

The key idea is to replace the commutative algebra $\cO_{\rm alg}(\pAA^{d})$ by the derived space of functions $\RR \Gamma(\pAA^{d}, \sO_{\rm alg})$. 
This complex has interesting cohomology and leads to nontrivial extensions of the Lie algebra object $\RR \Gamma(\pAA^{d}, \sO) \tensor \fg$, as well as its Dolbeault model $\Omega^{0,*}(\pAA^d) \tensor \fg$.
Faonte-Hennion-Kapranov \cite{FHK} have provided a systematic exploration of this situation.

Our starting point is to work in the style of complex differential geometry and use the sheaf of $\fg$-valued Dolbeault forms $\Omega^{0,*}(X, \fg)$, defined on any complex manifold $X$. 
We deem this sheaf of dg Lie algebras---or rather its cosheaf version $\sG_X = \Omega^{0,*}_c(X, \fg)$---the {\em holomorphic $\fg$-valued currents} on~$X$. 
We will see that there exists cocycls on this sheaf of dg Lie algebras that give rise to interesting extensions of the factorization algebra~$\clieu_*\sG_X$,
which serve as our model for a higher dimensional Kac-Moody algebra. 
Section~\ref{sec:FHK} is devoted to relating our construction to that in~\cite{FHK}.

A novel facet of this paper is that we enhance this Lie algebraic object to a {\em factorization algebra} on the manifold~$X$
by working with whe Lie algebra chains $\clieu_*\sG_X$ of this cosheaf.
It serves as a higher dimensional analog of the chiral enveloping algebra of $\fg$ introduced by Beilinson and Drinfeld \cite{BD}, 
and it yields a higher dimensional generalization of the vertex algebra of a Kac-Moody algebra. 

Analogs of important objects over Riemann surfaces arise from this new construction.
For instance, we obtain a version of bundles of conformal blocks from our higher Kac-Moody algebras:
factorization algebras are local-to-global objects, and one can take the global sections
(sometimes called the factorization or chiral homology).
In this paper we explicitly examine the factorization homology on Hopf manifolds,
which provide a systematic generalization of elliptic curves 
in the sense that their underlying manifolds are diffeomorphic to $S^1 \times S^{2d-1}$.
Due to the appearance $S^1$, one finds connections with traces.
As one might hope, these Hopf manifolds form moduli and so one can obtain, in principle, generalizations of $q$-character formulas.
(Giving explicit formulas is deferred to a future work.)

Another key generalization is given by natural determinant lines on moduli of bundles.
Any finite-dimensional representation $V$ of the Lie algebra $\fg$ determines a line bundle over the moduli of bundles on a complex manifold~$X$: 
take the determinant of the Dolbeault cohomology of the associated holomorphic vector bundle $\cV$ over~$X$.
In \cite{FHK} they use derived algebraic geometry to provide a higher Kac-Moody uniformization for complex $d$-folds and discuss these determinant lines.
We offer a complementary perspective: such a determinant line appears as the global sections of a certain factorization algebra on $X$ determined by the vector bundle~$\cV$.
That is, there is another factorization algebra whose bundle of conformal blocks encodes this determinant.
We construct this factorization algebra as observables of a quantum field theory,
as generalizations of the $bc$ and $\beta\gamma$ systems.\footnote{To be more precise, our construction uses formal derived geometry and works on the formal neighborhood of any point on the moduli of bundles. 
Properly taking into account the global geometry would require more discussion.}
In short, by combining \cite{FHK} with our results, 
there seems to emerge a systematic, higher-dimensional extension of the beautiful, rich dialogue between representation theory of infinite-dimensional Lie algebras, complex geometry, and conformal field theory.

\subsection*{Acknowledgements}

We have intermittently worked on this project for several years,
and our collaboration benefited from shared time at the Max Planck Institute for Mathematics in Bonn, Germany and the Perimeter Institute for Physics in Waterloo, Canada.
We thank both institutions for their support and for their convivial atmosphere.
Most of these results appeared first in Chapter 4 of the second author's Ph.D. thesis \cite{BWthesis},
of which this paper is a revised and enhanced version. 
The second author would therefore like to thank Northwestern University, where he received support as a graduate student whilst most of this work took place.
In addition, the second author enjoyed support as a graduate student research fellow under NSF Award DGE-1324585. 

In addition to institutional support, we received guidance and feedback from Kevin Costello, Giovanni Faonte, Benjamin Hennion, and Mikhail Kapranov. 
We are also grateful for Eugene Rabinovich's careful reading of the first arXiv submitted version of this paper, his helpful suggestions and feedback improved the overall quality of the exposition. 
Thank you!

\section{Current algebras on complex manifolds}
\label{sec: fact}

This paper takes general definitions and constructions from \cite{CG1} and specializes them to the context of complex manifolds.
In this subsection we will review some of the key ideas but refer to \cite{CG1} for foundational results.

\begin{rmk}
It might help to bear in mind the one-dimensional case that we wish to extend. 
Let $\Sigma$ be a Riemann surface, and let $\fg$ be a simple Lie algebra with Killing form $\kappa$.
Consider the local Lie algebra $\sG_\Sigma = \Omega^{0,*}_c(\Sigma) \tensor \fg$ on $\Sigma$.
There is a natural cocycle depending precisely on two inputs:
\[
\theta( \alpha \otimes M, \beta \otimes N) = \kappa(M,N) \, \int_\Sigma \alpha \wedge \partial \beta  ,
\]
where $\alpha, \beta \in \Omega^{0,*}_c(\Sigma)$ and $M,N \in \fg$.
In Chapter 5 of \cite{CG1} it is shown how the twisted enveloping factorization algebra of $\sG_\Sigma$ for this cocycle recovers the Kac-Moody vertex algebra associated to the affine algebra extending~$L\fg = \fg[z,z^{-1}]$.
\end{rmk}

\subsection{Local Lie algebras}

A key notion for us is a sheaf of Lie algebras on a smooth manifold.
These often appear as infinitesimal automorphisms of geometric objects,
and hence as symmetries in classical field theories.

\begin{dfn} 
A {\em local Lie algebra} on a smooth manifold $X$ is 
\begin{itemize}
\item[(i)] a $\ZZ$-graded vector bundle $L$ on $X$ of finite total rank;
\item[(ii)] a degree 1 operator $\ell_1:\sL^{sh} \to \sL^{sh}$ on the sheaf $\sL^{sh}$ of smooth sections of~$L$, and
\item[(iii)] a degree 0 bilinear operator
\[
\ell_2 : \sL^{sh} \times \sL^{sh} \to \sL^{sh}
\]
\end{itemize}
such that $\ell_1^2 = 0$, $\ell_1$ is a differential operator, $\ell_2$ is a bidifferential operator,
\[
\ell_1(\ell_2(x,y)) = \ell_2(\ell_1(x), y) + (-1)^{|x|} \ell_2(x, \ell_1(y))
\]
and the graded Jacobi identity holds
\[
(-1)^{|x||z|} \ell_2(x,\ell_2(y,z)) + (-1)^{|x||y|} \ell_2(y, \ell_2(z,x)) + (-1)^{|y||z|} \ell_2(z,(x,y)) = 0,
\]
for any sections $x,y,z$ of $\sL^{sh}$ of degree $|x|, |y|, |z|$, respectively. 
We call $\ell_1$ the {\em differential} and $\ell_2$ the {\em bracket}.
\end{dfn}

In other words, a local Lie algebra is a sheaf of dg Lie algebras 
where the underlying sections are smooth sections of a vector bundle and 
where the operations are local in the sense of not enlarging support of sections. 
(As we will see, such Lie algebras often appear by acting naturally on the local functionals from physics, namely functionals determined by Lagrangian densities.)

\begin{rmk}
For a local Lie algebra, we reserve the more succinct notation $\sL$ to denote the precosheaf of {\em compactly supported} sections of $L$,
which assigns a dg Lie algebra to each open set $U \subset X$, 
since the differential and bracket respect support.
At times we will abusively refer to $\sL$ to mean the data determining the local Lie algebra,
when the support of the sections is not relevant to the discussion at hand.
\end{rmk}

The key examples for this paper all arise from studying the symmetries of holomorphic principal bundles.
We begin with the specific and then examine a modest generalization.

Let $\pi : P \to X$ be a holomorphic principal $G$-bundle over a complex manifold.
We use $\ad(P) \to X$ to denote the associated {\em adjoint bundle} $P \times^{G} \fg \to X$, 
where the Borel construction uses adjoint action of $G$ on $\fg$ from the left. 
The complex structure defines a $(0,1)$-connection $\dbar_P : \Omega^{0,q}(X ; \ad(P)) \to \Omega^{0,q+1}(X ; \ad(P))$
on the Dolbeault forms with values in the adjoint bundle,
and this connection satisfies $\dbar_P^2 = 0$.
Note that the Lie bracket on $\fg$ induces a pointwise bracket on smooth sections of $\ad(P)$~by
\[
[s,t](x) = [s(x),t(x)]
\]
where $s, t$ are sections and $x$ is a point in $X$.
This bracket naturally extends to Dolbeault forms with values in the adjoint bundle,
as the Dolbeault forms are a graded-commutative algebra.

\begin{dfn}\label{dfn: adjoint local}
For $\pi : P \to X$ a holomorphic principal $G$-bundle,
let $\sAd(P)^{sh}$ denote the local Lie algebra whose sections are $\Omega^{0,*}(X,\ad(P))$,
whose differential is $\dbar_P$, and whose bracket is the pointwise operation just defined above.
\end{dfn}

The dg Lie algebra $\sAd(P)^{sh}(X)$ controls formal deformations of the holomorphic principal $G$-bundle $P$. 
Indeed, given a Maurer-Cartan element $\alpha \in \sAd(P)^{sh}(X)^1 = \Omega^{0,1}(X, {\rm ad}(P))$ one considers the new complex structure defined by the connection $\dbar_P + \alpha$. 
The Maurer-Cartan condition is equivalent to $(\dbar_P + \alpha)^2 = 0$. 

This construction admits important variations.
For example, we can move from working over a fixed manifold $X$ to working over a site.
Let ${\rm Hol}_d$ denote the category whose objects are complex $d$-folds and whose morphisms are local biholomorphisms,\footnote{A biholomorphism is a bijective map $\phi: X \to Y$ such that both $\phi$ and $\phi^{-1}$ are holomorphic. A {\em local} biholomorphism means a map $\phi: X \to Y$ such that every point $x \in X$ has a neighborhood on which $\phi$ is a biholomorphism.}
This category admits a natural Grothendieck topology where a cover $\{\phi_i: U_i \to X\}$ means a collection of morphisms into $X$ such that union of the images is all of $X$.
It then makes sense to talk about a local Lie algebra on the site ${\rm Hol}_d$.
Here is a particularly simple example that appears throughout the paper.

\begin{dfn}
Let $G$ be a complex Lie group and let $\fg$ denote its ordinary Lie algebra.
There is a natural functor 
\[
\begin{array}{cccc}
\sG^{sh} :&  {\rm Hol}_d^\opp & \to & {\dgLie}\\
& X & \mapsto &\Omega^{0,*}(X) \otimes \fg,
\end{array}
\]
which defines a sheaf of dg Lie algebras.
Restricted to each slice ${{\rm Hol}_d}_{/X}$, it determines the local Lie algebra for the trivial principal bundle $G \times X \to X$, in the sense described above.
We use $\sG$ to denote the cosheaf of compactly supported sections $\Omega^{0,*}_c \otimes \fg$ on this site.
\end{dfn}

\begin{rmk}
It is not necessary to start with a complex Lie group: 
the construction makes sense for a dg Lie algebra over $\CC$ of finite total dimension.
We lose, however, the interpretation in terms of infinitesimal symmetries of the principal bundle.
\end{rmk}

\begin{rmk}
For any complex manifold $X$ we can restrict the functor $\sG^{sh}$ to the overcategory of opens in $X$, that we denote by $\sG^{sh}_X$. 
In this case, $\sG^{sh}_X$, or its compactly supported version $\sG_X$, comes from the local Lie algebra of Definition \ref{dfn: adjoint local} in the case of the trivial $G$-bundle on $X$. 
In the case that $X = \CC^d$ we will denote the sheaves and cosheaves of the local Lie algebra by $\sG_d^{sh}, \sG_d$ respectively.
\end{rmk}

\subsection{Current algebras as enveloping factorization algebras of local Lie algebras}
\label{sec: envelopes}

Local Lie algebras often appear as symmetries of classical field theories.
For instance, as we will show in Section \ref{sec: qft}, 
each finite-dimensional complex representation $V$ of a Lie algebra $\fg$
determines a charged $\beta\gamma$-type system on a complex $d$-fold $X$ with choice of holomorphic principal bundle $\pi: P \to X$.
Namely, the on-shell $\gamma$ fields are holomorphic sections for the associated bundle $P \times^G V \to X$, 
and the on-shell $\beta$ fields are holomorphic $d$-forms with values in the associated bundle $P \times^G V^* \to X$.
It should be plausible that $\sAd(P)^{sh}$ acts as symmetries of this classical field theory,
since holomorphic sections of the adjoint bundle manifestly send on-shell fields to on-shell fields.

Such a symmetry determines currents, which we interpret as observables of the classical theory.
Note, however, a mismatch: 
while fields are contravariant in space(time) because fields pull back along inclusions of open sets, 
observables are covariant because an observable on a smaller region extends to any larger region containing it.
The currents, as observables, thus do not form a sheaf but a precosheaf.
We introduce the following terminology.

\begin{dfn}
For a local Lie algebra $(L\to X, \ell_1,\ell_2)$, its precosheaf $\sL[1]$ of {\em linear currents} is given by taking compactly supported sections of~$L$.
\end{dfn}

There are a number of features of this definition that may seem peculiar on first acquaintance.
First, we work with $\sL[1]$ rather than $\sL$.
This shift is due to the Batalin-Vilkovisky formalism. 
In that formalism the observables in the classical field theory possesses a 1-shifted Poisson bracket $\{-,-\}$ (also known as the antibracket), and so if the current $J(s)$ associated to a section $s \in \sL$ encodes the action of $s$ on the observables, i.e.,
\[
\{J(s), F\} = s \cdot F,
\]
then we need the cohomological degree of $J(s)$ to be 1 less than the degree of $s$.
In short, we want a map of dg Lie algebras $J: \sL \to \Obs^\cl[-1]$,
or equivalently a map of 1-shifted dg Lie algebras $J: \sL[1] \to \Obs^\cl$,
where $\Obs^\cl$ denotes the algebra of classical observables.

Second, we use the term ``linear'' here because the product of two such currents is not in $\sL[1]$ itself, 
although such a product will exist in the larger precosheaf $\Obs^\cl$ of observables.
In other words, if we have a Noether map of dg Lie algebras $J: \sL \to \Obs^\cl[-1]$,
it extends to a map of 1-shifted Poisson algebras
\[
J: \Sym(\sL[1]) \to \Obs^\cl
\]
as $\Sym(\sL[1])$ is the 1-shifted Poisson algebra freely generated by the 1-shifted dg Lie algebra $\sL[1]$.
We hence call $\Sym(\fg[1])$ the {\em enveloping 1-shifted Poisson algebra} of a dg Lie algebra~$\fg$.\footnote{See \cite{BashVor, BrLaz, GwHaug} for discussions of these constructions and ideas.}

For any particular field theory, the currents generated by the symmetry for {\em that} theory are given by the image of this map $J$ of 1-shifted Poisson algebras.
To study the general structure of such currents, without respect to a particular theory,
it is natural to study this enveloping algebra by itself.

\begin{dfn}\label{dfn: classical currents}
For a local Lie algebra $(L\to X, \ell_1,\ell_2)$, its {\em classical currents} $\Cur^\cl(\sL)$ is the precosheaf $\Sym(\sL[1])$ given by taking the enveloping 1-shifted Poisson algebra of the compactly supported sections of~$L$.
It assigns
\[
\Cur^\cl(\sL)(U) = \Sym(\sL(U)[1])
\]
to an open subset $U \subset X$. 
\end{dfn}

We emphasize here that by $\Sym(\sL(U)[1])$ we do {\em not} mean the symmetric algebra in the purely algebraic sense, but rather a construction that takes into account the extra structures on sections of vector bundles (e.g., the topological vector space structure).
Explicitly, the $n$th symmetric power  $\Sym^n(\sL(U)[1])$ means the smooth, compactly supported, and $S_n$-invariant sections of the graded vector bundle 
\[
L[1]^{\boxtimes n} \to U^n.
\]
For further discussion of functional analytic aspects (which play no tricky role in our work here),
see \cite{CG1}, notably the appendices.

A key result of \cite{CG1}, namely Theorem 5.6.0.1, is that this precosheaf of currents forms a factorization algebra. 
From hereon we refer to  $\Cur^\cl(\sL)$ as the {\em factorization algebra of classical currents}.
If the local Lie algebra acts as symmetries on some classical field theory,
we obtain a map of factorization algebras $J: \Cur^\cl(\sL) \to \Obs^\cl$ that encodes each current as a classical observable.

There is a quantum counterpart to this construction, in the Batalin-Vilkovisky formalism.
The idea is that for a dg Lie algebra $\fg$, 
the enveloping 1-shifted Poisson algebra $\Sym(\fg[1])$ admits a natural BV quantization via the Chevalley-Eilenberg chains $C_*(\fg)$.  
This assertion is transparent by examining the Chevalley-Eilenberg differential:
\[
\d_{CE}(xy) = \d_\fg(x)y \pm x\, \d_\fg(y) + [x,y]
\]
for $x,y$ elements of $\fg[1]$.
The first two terms behave like a derivation of $\Sym(\fg[1])$, 
and the last term agrees with the shifted Poisson bracket.
More accurately, to keep track of the $\hbar$-dependency in quantization,
we introduce a kind of Rees construction.

\begin{dfn}
\label{def: BD envelope}
The {\em enveloping $BD$ algebra} $U^{BD}(\fg)$ of a dg Lie algebra $\fg$ is given by the graded-commutative algebra in $\CC[\hbar]$-modules
\[
\Sym(\fg[1])[\hbar] \cong \Sym_{\CC[\hbar]}(\fg[\hbar][1]),
\]
but the differential is defined as a coderivation with respect to the natural graded-cocommutative coalgebra structure,
by the condition
\[
\d(xy) = \d_\fg(x)y \pm x\, \d_\fg(y) + \hbar [x,y].
\]
\end{dfn}

This construction determines a BV quantization of the enveloping 1-shifted Poisson algebra,
as can be verified directly from the definitions.
(For further discussion see \cite{GwHaug} and \cite{CG2}.)
It is also straightforward to extend this construction to ``quantize'' the factorization algebra of classical currents.

\begin{dfn}
\label{dfn: quantum currents}
For a local Lie algebra $(L\to X, \ell_1,\ell_2)$, 
its {\em factorization algebra of quantum currents} $\Cur^\q(\sL)$ is given by taking the enveloping $BD$   algebra of the compactly supported sections of~$L$.
It assigns
\[
\Cur^\q(\sL)(U) = U^{BD}(\sL(U))
\]
to an open subset $U \subset X$.
\end{dfn}

As mentioned just after the definition of the classical currents, 
the symmetric powers here mean the construction involving sections of the external tensor product.
Specializing $\hbar = 1$, we recover the following construction.

\begin{dfn}
For a local Lie algebra $(L\to X, \ell_1,\ell_2)$, 
its {\em enveloping factorization algebra} $\UU(\sL)$ is given by taking the Chevalley-Eilenberg chains $\cliels(\sL)$ of the compactly supported sections of~$L$.
\end{dfn}

Here the symmetric powers use sections of the external tensor powers, just as with the classical or quantum currents.

When a local Lie algebra acts as symmetries of a classical field theory,
it sometimes also lifts to symmetries of a BV quantization.
In that case the map $J: \Sym(\sL[1]) \to \Obs^\cl$ of 1-shifted Poisson algebras lifts to a cochain map $J^\q: \Cur^\q(\sL) \to \Obs^\q$ realizing quantum currents as quantum observables.
Sometimes, however, the classical symmetries do not lift directly to quantum symmetries.
We turn to discussing the natural home for the obstructions to such lifts after a brief detour to offer a structural perspective on the enveloping construction.

\subsubsection{A digression on the enveloping $E_n$ algebras}
\label{sec:knudsen}

This construction $\UU(\sL)$ has a special feature when the local Lie algebra is obtained by taking the de Rham forms with values in a dg Lie algebra $\fg$, i.e., when $\sL = \Omega^*_c \otimes \fg$.
In that case the enveloping factorization algebra is locally constant and, on the $d$-dimensional real manifold $\RR^d$, determines an $E_d$ algebra, also known as an algebra over the little $d$-disks operad, by a result of Lurie (see Theorem 5.5.4.10 of \cite{LurieHA}).
This construction satisfies a universal property: 
it is the $d$-dimensional generalization of the universal enveloping algebra of a Lie algebra.

To state this result of Knudsen precisely, we need to be in the context of $\infty$-categories.

\begin{thm}[\cite{Knudsen}] 
\label{thm:knudsen}
Let $\sC$ be a stable, $\CC$-linear, presentable, symmetric monoidal $\infty$-category.
There is an adjunction
\[
U^{E_d} : {\rm LieAlg}(\sC) \leftrightarrows E_d{\rm Alg} (\sC): F
\]
between Lie algebra objects in $\sC$ and $E_d$ algebra objects in $\sC$.
This adjunction intertwines with the free-forget adjunctions from Lie and $E_d$ algebras in $\sC$ to $\sC$ 
so that 
\[
{\rm Free}_{E_d}(X) \simeq U^{E_d} {\rm Free}_{Lie}(\Sigma^{d-1} X)
\]
for any object $X \in \sC$. 

When $\sC$ is the $\infty$-category of chain complexes over a field of characteristic zero,
the $E_d$ algebra $U^{E_d} \fg$ is modeled by the locally constant factorization algebra $\UU(\Omega^*_c \otimes \fg)$ on~$\RR^d$.
\end{thm}

This theorem is highly suggestive for us:
our main class of examples is $\sG_d$ and $\UU \sG_d$,
which replaces the de Rham complex with the Dolbeault complex.
In other words, we anticipate that $\UU \sG_d$ should behave like a holomorphic version of an $E_d$ algebra
and that it should be the canonical such algebra determined by a dg Lie algebra.
We do not pursue this structural result in this paper,
but it provides some intuition behind our constructions.

\subsection{Local cocycles and shifted extensions}
\label{sec: localcocycle}

Some basic questions about a dg Lie algebra $\fg$, such as the classification of extensions and derivations, are encoded cohomologically, typically as cocycles in the Chevalley-Eilenberg cochains $\clies(\fg,V)$ with coefficients in some $\fg$-representation~$V$.
When working with local Lie algebras, it is natural to focus on cocycles that are also local in the appropriate sense.
(Explicitly, we want to restrict to cocycles that are built out of polydifferential operators.)
After introducing the relevant construction, we turn to studying how such cocycles determine modified current algebras.

\subsubsection{Local cochains of a local Lie algebra}
\label{sec:cloc}

In Section 4.4 of \cite{CG2} the local cochains of a local Lie algebra are defined in detail, 
but we briefly recall it here.
The basic idea is that a local cochain is a Lagrangian density: 
it takes in a section of the local Lie algebra and produces a smooth density on the manifold. 
Such a cocycle determines a functional by integrating the density.
As usual with Lagrangian densities, we wish to work with them up to total derivatives,
i.e., we identify Lagrangian densities related using integration by parts and hence ignore boundary terms.

In a bit more detail, for $L$ is a graded vector bundle, let $JL$ denote the corresponding $\infty$-jet bundle,
which has a canonical flat connection.
In other words, it is a left $D_X$-module, where $D_X$ denotes the sheaf of smooth differential operators on $X$.
For a local Lie algebra, this $JL$ obtains the structure of a dg Lie algebra in left $D_X$-modules.
Thus, we may consider its reduced Chevalley-Eilenberg cochain complex $\clies(JL)$ in the category of left $D_X$-modules. 
By taking the de Rham complex of this left $D_X$-module, we obtain the local cochains.
For a variety of reasons, it is useful to ignore the ``constants'' term and work with the reduced cochains.
Hence we have the following definition.

\begin{dfn}
Let $\sL$ be a local Lie algebra on $X$.
The {\em local Chevalley-Eilenberg cochains}  of $\sL$~is 
\[
\cloc^*(\sL) = \Omega^{*}_X[2d] \tensor_{D_X} \cred^*(J L) .
\]
This sheaf of cochain complexes on $X$ has global sections that we denote by~$\cloc^*(\sL(X))$.
\end{dfn}

Note that we use the smooth de Rham forms, not the holomorphic de Rham forms.

\begin{rmk}
This construction $\cloc^*(\sL)$ is just a version of diagonal Gelfand-Fuks cohomology \cite{Fuks, LosikDiag},
where the adjective ``diagonal'' indicates that we are interested in continuous cochains whose integral kernels are supported on the small diagonals.
\end{rmk}

\subsubsection{Shifted extensions}

For an ordinary Lie algebra $\fg$, central extensions are parametrized by 2-cocycles on $\fg$ valued in the trivial module~$\CC$. 
It is possible to interpret arbitrary cocycles as determining as determining {\em shifted} central extensions as {\em $L_\infty$ algebras}.
Explicitly, a $k$-cocycle $\Theta$ of degree $n$ on a dg Lie algebra $\fg$ determines an $L_\infty$ algebra structure on the direct sum $\fg \oplus \CC[n-k]$ with the following brackets $\{\Hat{\ell}_m\}_{m \geq 1}$: $\Hat{\ell}_1$ is simply the differential on $\fg$, $\Hat{\ell}_2$ is the bracket on $\fg$, $\Hat{\ell}_m = 0$ for $m >2$ except
\[
\Hat{\ell}_k(x_1 + a_1, \ldots, x_k + a_k) = 0+ \Theta(x_1,x_2,\ldots, x_k).
\]
(See \cite{KonSoi,LodVal} for further discussion. Note that $n=2$ for $k=2$ with ordinary Lie algebras.)
Similarly, local cocycles provide shifted central extensions of local Lie algebras.

\begin{dfn}
For a local Lie algebra $(L, \ell_1,\ell_2)$, a cocycle $\Theta$ of degree $2+k$ in $\cloc^*(\sL)$ determines a {\em $k$-shifted central extension}
\beqn\label{kext}
0 \to \CC[k] \to \Hat{\sL}_\Theta \to \sL \to 0
\eeqn
of precosheaves of $L_\infty$ algebras, where the $L_\infty$ structure maps are defined by
\[
\Hat{\ell}_n(x_1,\ldots,x_n) = (\ell_n(x_1,\ldots,x_n), \int \Theta(x_1,\ldots,x_n)).
\]
Here we set $\ell_n = 0$ for $n > 2$.
\end{dfn}

As usual, cohomologous cocycles determine quasi-isomorphic extensions. 
Much of the rest of the section is devoted to constructing and analyzing various cocycles and the resulting extensions.

\subsubsection{Twists of the current algebras}

Local cocycles give a direct way of deforming the various current algebras a local Lie algebra.
For example, we have the following construction.

\begin{dfn} 
Let $\Theta$ be a degree 1 local cocycle for a local Lie algebra $(L \to X, \ell_1,\ell_2)$. 
Let $K$ denote a degree zero parameter so that $\CC[K]$ is a polynomial algebra concentrated in degree zero.
The {\em twisted enveloping factorization algebra} $\UU_\Theta(\sL)$ assigns to an open $U \subset X$, the cochain complex
\begin{align*}
\UU_\Theta(\sL)(U) & = \left(\Sym(\sL(U)[1] \oplus \CC \cdot K), \d_{\sL} + K \cdot \Theta\right) \\
& = \left(\Sym(\sL(U)[1])[K] , \d_{\sL} + K \cdot \Theta\right),
\end{align*}
where $\d_{\sL}$ denotes the differential on the untwisted enveloping factorization algebra and $\Theta$ is the operator extending the cocycle $\Theta : \Sym(\sL(U)[1]) \to \CC \cdot K$ to the symmetric coalgebra as a graded coderivation.
This twisted enveloping factorization algebra is module for the commutative ring~$\CC[K]$,
and so specializing the value of $K$ determines nontrivial modifications of~$\UU(\sL)$. 
\end{dfn}

An analogous construction applies to the quantum currents, which we will denote~$\Cur^\q_\Theta(\sL)$.

\subsubsection{A special class of cocycles: the $\fj$ functional} 
\label{sec: g j functional}

There is a particular family of local cocycles that has special importance in studying symmetries of higher dimensional holomorphic field theories. 

Consider 
\[
\theta \in \Sym^{d+1}(\fg^*)^\fg,
\]
so that $\theta$ is a $\fg$-invariant polynomial on $\fg$ of homogenous degree $d+1$. 
This data determines a local functional for $\sG = \Omega^{0,*} \otimes \fg$ on any complex $d$-fold as follows.

\begin{dfn}
For any complex $d$-fold $X$, extend $\theta$ to a functional $\fJ_X(\theta)$ on $\sG_X = \Omega^{0,*}_c(X) \tensor \fg$ by the formula
\beqn\label{j g formula}
\fJ_X(\theta)(\alpha_0 ,\ldots,\alpha_d) = \int_X \theta(\alpha_0,\partial \alpha_1,\ldots,\partial \alpha_d),
\eeqn
where $\partial$ denotes the holomorphic de Rham differential.
In this formula, we define the integral to be zero whenever the integrand is not a $(d,d)$-form.
\end{dfn}

To make this formula as clear as possible, suppose the $\alpha_i$ are pure tensors of the form $\omega_i \otimes y_i$ with $\omega_i \in \Omega^{0,*}_c(X)$ and $y_i \in \fg$.
Then
\beqn\label{jthetafactored}
\fJ_X(\theta) (\omega_0 \tensor y_0,\ldots,\omega_{d} \tensor y_{d}) = \theta(y_0,\ldots,y_{d}) \int_X \omega_0\wedge \partial \omega_1 \cdots \wedge \partial \omega_{d}.
\eeqn
Note that we use $d$ copies of the holomorphic derivative $\partial: \Omega^{0,*} \to \Omega^{1,*}$ to obtain an element of $\Omega^{d,*}_c$ in the integrand and hence something that can be integrated.

This formula manifestly makes sense for any complex $d$-fold $X$, 
and since integration is local on $X$, 
it intertwines nicely with the structure maps of~$\sG_X$.

\begin{dfn}\label{dfn: j}
For any complex $d$-fold $X$ and any $\theta \in \Sym^{d+1}(\fg^*)^\fg$, 
let $\fj_X(\theta)$ denote the local cochain in $\cloc^*(\sG_X)$ defined~by
\[
\fj_X(\theta)(\alpha_0 ,\ldots,\alpha_d) = \theta(\alpha_0,\partial \alpha_1,\ldots,\partial \alpha_d).
\]
Hence $\fJ_X(\theta) = \int_X \fj_X(\theta)$.
\end{dfn}

This integrand $\fj_X(\theta)$ is in fact a local cocycle, and 
in a moment we will use it to produce an important shifted central extension of~$\sG_X$.

\begin{prop}\label{prop j map} 
The assignment 
\[
\begin{array}{cccc}
\fj_X : & \Sym^{d+1} (\fg^*)^\fg [-1]  & \to & \cloc^*(\sG_X) \\ 
& \theta &\mapsto & \fj_X(\theta)
\end{array}
\]
is an cochain map.
\end{prop}

\begin{proof} 
The element $\fj_X(\theta)$ is local as it is expressed as a density produced by polydifferential operators.
We need to show that $\fj_X(\theta)$ is closed for the differential on $\cloc^*(\sG_X)$. 
Note that $\sG_X$ is the tensor product of the dg commutative algebra $\Omega^{0,*}_X$ and the Lie algebra $\fg$.
Hence the differential on the local cochains of $\sG_X$ splits as a sum $\dbar + \d_{\fg}$ where $\dbar$ denotes the differential on local cochains induced from the $\dbar$ differential on the Dolbeault forms and $\d_{\fg}$ denotes the differential induced from the Lie bracket on~$\fg$. 
We now analyze each term separately.

Observe that for any collection of $\alpha_i \in \sG$, we have
\begin{align*}
\dbar(\fj_X(\theta)(\alpha_0,\partial \alpha_1,\ldots,\partial \alpha_d)) 
&= \fj_X(\theta)(\dbar\alpha_0,\partial \alpha_1,\ldots,\partial \alpha_d) \pm \fj_X(\theta)(\alpha_0,\dbar \partial \alpha_1,\ldots,\partial \alpha_d) \pm \cdots \\ & \;\;\;\; \cdots \pm \fj_X(\theta)(\alpha_0,\partial \alpha_1,\ldots,\dbar\partial \alpha_d)\\
&= (\dbar \fj_X(\theta))(\alpha_0,\partial \alpha_1,\ldots,\partial \alpha_d)
\end{align*}
because $\dbar$ is a derivation and $\theta$ wedges the form components.
(It is easy to see this assertion when one works with inputs like in equation \eqref{jthetafactored}.)
Hence viewing $\fj_X(\theta)$ as a map from $\sG$ to the Dolbeault complex, 
it commutes with the differential $\dbar$.
This fact is equivalent to $\dbar \fj_X(\theta) = 0$ in local cochains.

Similarly, observe that for any collection of $\alpha_i \in \sG$, we have
\begin{align*}
(\d_\fg \fj_X(\theta))(\alpha_0, \alpha_1,\ldots, \alpha_d)
&= (\d_\fg\fj_X(\theta))(\alpha_0,\partial \alpha_1,\ldots,\partial \alpha_d)) \\
&= 0
\end{align*}
since $\theta$ is closed for the Chevalley-Eilenberg differential for $\fg$. 
\end{proof}

As should be clear from the construction, everything here works over the site ${\rm Hol}_d$ of complex $d$-folds, and hence we use $\fj(\theta)$ to denote the local cocycle for the local Lie algebra $\sG$ on~${\rm Hol}_d$.

This construction works nicely for an arbitrary holomorphic $G$-bundle $P$ on $X$,
because the cocycle is expressed in a coordinate-free fashion.
To be explicit, on a coordinate patch $U_i \subset X$ with a choice of trivialization of the adjoint bundle $\ad(P)$,
the formula for $\fj_X(\theta)$ makes sense.
On an overlap $U_i \cap U_j$, the cocycles patch because $\fj_X(\theta)$ is independent of the choice of coordinates.
Hence we can glue over any sufficiently refined cover to obtain a global cocycle. 
Thus, we have a cochain map
\[
\fj_X^P : \Sym^{d+1} (\fg^*)^\fg [-1] \to \cloc^*(\sAd(P)(X))
\]
given by the same formula as in~\eqref{j g formula}.

\subsubsection{Another special class: the LMNS extensions}
\label{sec: nekext}

Much of this paper focuses on local cocycles of type $\fj_X(\theta)$, where $\theta \in \Sym^{d+1}(\fg^*)^\fg$.
But there is another class of local cocycles that appear naturally when studying symmetries of holomorphic theories. 
Unlike the cocycle $\fj_X(\theta)$, which only depend on the manifold $X$ through its dimension, 
this class of cocycles depends on the geometry.

In complex dimension two, this class of cocycles has appeared in the work of Losev-Moore-Nekrasov-Shatashvili (LMNS) \cite{LMNS1,LMNS2,LMNS3} in their construction of a higher analog of the ``chiral WZW theory". 
Though our approaches differ, we share their ambition to formulate a higher analogs of constructions and ideas in chiral CFT. 

Let $X$ be a complex manifold of dimension $d$ with a choice of $(k,k)$-form~$\eta$. 
Choose a form $\theta_{d+1-k} \in \Sym(\fg^*)^\fg$.
This data determines a local cochain on~$\sG_X$ whose local functional~is:
\[
\begin{array}{cccc}
\displaystyle \phi_{\theta, \eta} : & \sG(X)^{\tensor d + 1 - k} & \to & \CC \\
\displaystyle & \alpha_0 \tensor\cdots \tensor \alpha_{d-k} & \mapsto & \displaystyle \int_X \eta \wedge \theta_{d+1-k}(\alpha_0, \partial\alpha_1,\ldots,\partial \alpha_{d-k})
\end{array}.
\]
Such a cochain is a cocycle only if $\dbar \eta = 0$, because $\eta$ does not interact with the Lie structure.

Note that a K\"{a}hler manifold always produces natural choices of $\eta$ by taking $\eta = \omega^{k}$, where $\omega$ is the symplectic form.
In this way, K\"{a}hler geometry determines an important class of extensions.
It would be interesting to explore what aspects of the geometry are reflected by these associated current algebras.
The following is a direct calculation.

\begin{lem}\label{lem: cocycle KM}
Fix $\theta \in \Sym^{d+1-k}(\fg^*)^\fg$.
If a form $\eta \in \Omega^{k,k}(X)$ satisfies $\dbar \eta = 0$ and $\partial \eta = 0$,
then the local cohomology class $[\phi_{\theta,\omega}] \in H^1_{\rm loc}(\sG_X)$  depends only on the cohomology class $[\omega] \in H^{k}(X , \Omega^k_{cl})$.
\end{lem}

When $\eta = 1$, it trivially satisfies the conditions of the lemma. 
In this case $\phi_{\theta, 1} = \fj_X(\theta)$ in the notation of the last section. 

\subsection{The higher Kac-Moody factorization algebra}

Finally, we can introduce the central object of this paper.

\begin{dfn}
Let $X$ be a complex manifold of complex dimension $d$ equipped with a holomorphic principal $G$-bundle $P$.
Let $\Theta$ be a degree 1 cocycle in $\cloc^*(\sAd(P))$, 
which determines a 1-shifted central extension $\sAd(P)_\Theta$.
The {\em Kac-Moody factorization algebra on $X$ of type $\Theta$} is the twisted enveloping factorization algebra $\UU_\Theta (\sAd(P))$ that assigns
\[
\left(\Sym\left(\Omega^{0,*}_c(U, \ad(P))[1]\right) [K] , \dbar + \d_{CE} + \Theta\right) 
\]
to an open set $U \subset X$.
\end{dfn}

\begin{rmk} 
As in the definition of twisted enveloping factorization algebras, the factorization algebras $\UU_\Theta(\sAd(P))$ are modules for the ring $\CC[K]$. 
In keeping with conventions above, when $P$ is the trivial bundle on $X$, 
we will denote the Kac-Moody factorization algebra by $\UU_\Theta(\sG_X)$. 
\end{rmk}

The most important class of such higher Kac-Moody algebras makes sense over the site ${\rm Hol}_d$ of all complex $d$-folds.

\begin{dfn}
Let $\fg$ be an ordinary Lie algebra and let $\theta \in \Sym^{d+1}(\fg^*)^\fg$.  
Let $\sG_{d,\theta}$ denote the 1-shifted central extension of $\sG_d$ determined by the local cocycle $\fj(\theta)$.
Let $\UU_\theta (\sG)$ denote the {\em $\theta$-twisted enveloping factorization algebra} $\UU_{\fj(\theta)} (\sG)$ for the local Lie algebra $\sG = \Omega^{0,*}_c \otimes \fg$ on the site ${\rm Hol}_d$ of complex $d$-folds.
\end{dfn}

In the case $d = 1$ the definition above agrees with the Kac-Moody factorization algebra on Riemann surfaces given in \cite{CG1}.
There, it is shown that this factorization algebra, restricted to the complex manifold $\CC$, recovers a vertex algebra isomorphic to that of the ordinary Kac-Moody vertex algebra.
(See Section 5 of Chapter 5.)
Thus, we think of the object $\UU_\Theta(\sAd(P))$ as a higher dimensional version of the Kac-Moody vertex algebra.

\subsubsection{Holomorphic translation invariance and higher dimensional vertex algebras} \label{sec: hol trans main}

To put some teeth into the previous paragraph,
we note that \cite{CG1} introduces a family of colored operads ${\rm PDiscs_d}$, the little $d$-dimensional polydiscs operads,
that provide a holomorphic analog of the little $d$-disks operads~$E_d$.
Concretely, this operad ${\rm PDiscs_d}$ encodes the idea of the operator product expansion, 
where one now understands observables supported in small disks mapping into observables in large disks, rather than point-like observables.

In the case $d=1$, Theorem 5.3.3 of \cite{CG1} shows that a ${\rm PDiscs_1}$-algebra $\cA$ determines a vertex algebra $\VV(\cA)$ so long as $\cA$ is suitably equivariant under rotation .
This construction $\VV$ is functorial.
As shown in \cite{CG1}, many vertex algebras appear this way, and any vertex algebras that arise from physics should, in light of the main results of~\cite{CG1,CG2}.

For this reason, one can interpret ${\rm PDiscs_d}$-algebras, particularly when suitably equivariant under rotation, as providing a systematic and operadic generalization of vertex algebras to higher dimensions. 
Proposition 5.2.2 of \cite{CG1} provides a useful mechanism for producing ${\rm PDiscs_d}$-algebra: 
it says that if a factorization algebra is equivariant under translation in a holomorphic manner, then it determines such an algebra.

Hence it is interesting to identify when the higher Kac-Moody factorization algebras are invariant in the sense needed to produce ${\rm PDiscs_d}$-algebras.
We now address this question.

First, note that on the complex $d$-fold $X = \CC^d$, 
the local Lie algebra $\sG_{d}$ is manifestly equivariant under translation.

It is important to recognize that this translation action is holomorphic in the sense that the infinitesimal action of the (complexified) vector fields $\partial/\partial \Bar{z}_i$ is homotopically trivial.
Explicitly, consider the operator $\eta_i = \iota_{\partial/\partial \Bar{z}_i}$ on Dolbeault forms
(and which hence extends to $\sG_{\CC^d}$), and
note that
\[
[\dbar, \eta_i] = \partial/\partial \Bar{z}_i.
\]
Both the infinitesimal actions and this homotopical trivialization extend canonically to the Chevalley-Eilenberg chains of $\sG_{\CC^d}$ and hence to the enveloping factorization algebra and the current algebras.
(For more discussion of these ideas see \cite{BWhol} and Chapter 10 of \cite{CG2}.)

A succinct way to express this feature is to introduce a dg Lie algebra  
\[
\CC^d_{\rm hol} = \text{span}_\CC\{\partial/\partial z_1, \ldots, \partial/\partial z_d, \partial/\partial \Bar{z}_1,\ldots,\partial/\partial \Bar{z}_d, \eta_1,\ldots, \eta_d\}
\]
where the partial derivatives have degree 0 and the $\eta_i$ have degree $-1$,
where the brackets are all trivial, 
and where the differential behaves like $\dbar$ in the sense that the differential of $\eta_i$ is $\partial/\partial \Bar{z}_i$.
We just argued in the preceding paragraph that $\sG_{\CC^d}$ and its current algebras are all strictly $\CC^d_{\rm hol}$-invariant. 

When studying shifted extensions of $\sG_{\CC^d}$, 
it then makes sense to consider local cocycles that are also translation invariant in this sense.
Explicitly, we ask to work with cocycles~in
\[
\cloc^*(\sG_{d})^{\CC^d_{\rm hol}} \subset \cloc^*(\sG_{d}).
\]
Local cocycles here determine higher Kac-Moody algebras that are holomorphically translation invariant and hence yield ${\rm PDiscs}_d$-algebras.

The following result indicates tells us that we have already encountered all the relevant cocycles so long as we also impose rotation invariance, which is a natural condition.

\begin{prop}
\label{prop: trans j}
The map $\fj_{\CC^d} :  \Sym^{d+1}(\fg^*)^\fg [-1] \to \cloc^*(\sG_{d})$ factors through the  subcomplex of local cochains that are rotationally and holomorphically translation invariant.
Moreover, it determines an isomorphism on $H^1$:
\[
H^1(\fj_{\CC^d}) : \Sym^{d+1}(\fg^*)^\fg \xto{\cong} H^1 \left(\cloc^*(\sG_d))^{\CC^{d}_{\rm hol}}\right)^{U(d)} .
\]
\end{prop}

As the proof is rather lengthy, we provide it in Appendix~\ref{sec: hol trans}.

\section{Local aspects of the higher Kac-Moody factorization algebras} 
\label{sec: sphere ops}

A factorization algebra encodes an enormous amount of information, 
and hence it is important to extract aspects that are simpler to understand.
In this section we will take two approaches:
\begin{enumerate}
\item by compactifying along a sphere of real dimension $2d-1$, 
we obtain an algebra (more precisely, a homotopy-coherent associative algebra) that encodes the higher dimensional version of ``radial ordering'' of operators from two-dimensional conformal field theory, and
\item by compactifying along a torus $(S^1)^d$, 
we obtain an algebra over the little $d$-disks operad.
\end{enumerate}
In both cases these algebras behave like enveloping algebras of homotopy-coherent Lie algebras (in a sense we will spell out in detail below), which allows for simpler descriptions of some phenomena. 
It is important to be aware, however, that these algebras do not encode the full algebraic structure produced by the compactification; instead, they sit as dense subalgebras.
We will elaborate on this subtlety below.

For factorization algebras, compactification is accomplished by the pushforward operation.
Given a map $f: X \to Y$ of manifolds and a factorization algebra $\cF$ on $X$,
its {\em pushforward} $f_* \cF$ is the factorization algebra on $Y$ where
\[
f_*\cF(U) = \cF(f^{-1}(U))
\]
for any open $U \subset Y$.
The first example we treat arises from the radial projection map
\[
r: \CC^d \setminus \{0\} \to (0,\infty)
\]
sending $z$ to its length $|z|$. 
The preimage of a point is simply a $2d-1$-sphere,
so one can interpret the pushforward Kac-Moody factorization algebra $r_* \UU_\theta \cG_d$ as compactification along these spheres.
Our first main result is that there is a locally constant factorization algebra $\cA$ along $(0,\infty)$ with a natural map
\[
\phi: \cA \to r_* \UU_\theta \cG_d
\]
that is dense from the point of view of the topological vector space structure.
By a theorem of Lurie, locally constant factorization algebras on $\RR$ correspond to homotopy-coherent associative algebras,
so that we can interpret $\phi$ as saying that the pushforward is approximated by an associative algebra, in this derived sense.
We will show explicitly that this algebra is the $A_\infty$ algebra arising as the enveloping algebra of an $L_\infty$ algebra already introduced by Faonte-Hennion-Kapranov.

For the physically-minded reader, 
this process should be understood as a version of radial ordering.
Recall from the two-dimensional setting that it can be helpful to view the punctured plane as a cylinder,
and to use the radius as a kind of time parameter.
Time ordering of operators is then replaced by radial ordering.
Many computations can be nicely organized in this manner,
because a natural class of operators arises by using a Cauchy integral around the circle of a local operator.
The same technique works in higher dimensions where one now computes residues along the $2d-1$-spheres.
From this perspective, the natural Hilbert space is associated to the origin in the plane
(more accurately to an arbitrarily small disk around the origin),
and this picture also extends to higher dimensions.
Hence we obtain a kind of vacuum module for this higher dimensional generalization of the Kac-Moody algebras.

Our second cluster of results uses compactification along the projection map
\[
\begin{array}{ccc}
\CC^d \setminus \{\text{coordinate hyperplanes}\} & \to & (0,\infty)^d \\
(z_1,\ldots,z_d) & \mapsto & (|z_1|,\ldots,|z_d|).
\end{array}
\]
We construct a locally constant factorization algebra on $(0,\infty)^d$ that maps densely into the pushforward of the higher Kac-Moody algebra. 
Lurie's theorem shows that locally constant factorization algebras on $\RR^d$ correspond to $E_d$ algebras,
so we obtain a higher-dimensional analog of the spherical result.

\subsection{Compactifying the higher Kac-Moody algebras along spheres}
\label{sec: spheres}

Our approach is modeled on the construction of the affine Kac-Moody Lie algebras and their associated vertex algebras from Section 5.5 of~\cite{CG1} and~\cite{GwThesis},
so we review the main ideas to orient the reader.

On the punctured plane $\CC^*$, the sheaf $
\sG_1^{sh} = \Omega^{0,*} \otimes \fg$ is quasi-isomorphic to the sheaf $\cO \otimes \fg$.
The restriction maps of this sheaf tell us that for any open set $U$, there is a map of Lie algebras
\[
\cO(\CC^*) \otimes \fg \to \cO(U) \otimes \fg,
\]
so that we get a map of Lie algebras
\[
\cO_{\rm alg}(\CC^*) \otimes \fg = \fg[z,z^{-1}] \to  \cO(U) \otimes \fg
\]
because Laurent polynomials $\CC[z,z^{-1}] = \cO_{\rm alg}(\CC^*)$ are well-defined on any open subset of the punctured plane.
This {\em loop algebra} $L\fg = \fg[z,z^{-1}]$ admits interesting central extensions,
known as the affine Kac-Moody Lie algebras.
These extensions are labeled by elements of $\Sym^2(\fg^*)^\fg$, 
which is compatible with our work in Section~\ref{sec: g j functional}.

To apply radial ordering to this sheaf---or rather, its associated current algebras---it is convenient to study the pushforward along the radial projection map $r(z) = |z|$.
Note that the preimage of an interval $(a,b)$ is an annulus, so
\[
r_* \sG_1^{sh}((a,b)) = \sG_1^{sh}(\{a < |z| < b\})
\]
and hence we have a canonical map of Lie algebras
\[
\fg[z,z^{-1}] \to \cO(\{a < |z| < b\}) \otimes \fg \hookrightarrow r_* \sG_1^{sh}((a,b)).
\]
We can refine this situation by replacing the left hand side with the locally constant sheaf $\underline{\fg[z,z^{-1}]}$ to produce a map of sheaves $\underline{\fg[z,z^{-1}]} \to  r_* \sG_1^{sh}((a,b))$.
The Poincar\'e lemma tells us that $\Omega^*$ is quasi-isomorphic to the locally constant sheaf $\underline{\CC}$,
and so we can introduce a sheaf
\[
\mathtt{Lg}^{sh} = \Omega^* \otimes \fg[z,z^{-1}]
\]
that is a soft resolution of $\underline{\fg[z,z^{-1}]}$.
There is then a map of sheaves of dg Lie algebras
\beqn
\label{eqn:looptolinearcurrent}
\mathtt{Lg}^{sh} \to r_* \sG_1^{sh}
\eeqn
that sends $\alpha \otimes x\, z^n$ to $[r^*\alpha]_{0,*} \cdot z^n \otimes x$, with $x \in \fg$, $\alpha$ a differential form on $(0,\infty)$, and $[r^*\alpha]_{0,*}$ the $(0,*)$-component of the pulled back form.
This map restricts nicely to compactly support sections $\mathtt{Lg} \to r_* \sG_1$.
By taking Chevalley-Eilenberg chains on both sides, we obtain a map of factorization algebras
\beqn
\label{eqn:Uoflooptolinearcurrent}
\UU\mathtt{Lg} = \cliels(\mathtt{Lg}) \to \cliels(r_* \sG_1) = r_*\UU\sG_1.
\eeqn
The left hand side $\UU\mathtt{Lg}$ encodes the associative algebra $U(L\fg)$, the enveloping algebra of $L\fg$,
as can be seen by direct computation (see section 3.4 of \cite{CG1}) or by a general result of Knudsen~\cite{Knudsen}.
The right hand side contains operators encoded by Cauchy integrals, 
and it is possible to identify such as operator, up to exact terms, as the limit of a sequence of elements from~$U(L\fg)$.

We extend this argument to the affine Kac-Moody Lie algebras by working with suitable extensions on $\mathtt{Lg}$.
It is a deformation-theoretic argument, as we view the extensions as deforming the bracket.

We wish to replace the punctured plane $\CC^*$ by the punctured $d$-dimensional affine space 
\[
\pAA^d = \CC^d \setminus \{0\},
\] 
the current algebras of $\sG_1$ by the current algebras of $\sG_d$,
and, of course, the extensions depending on $\Sym^2(\fg^*)^\fg$ by other local cocycles.
There are two nontrivial steps to this generalization:
\begin{enumerate}
\item finding a suitable replacement for the Laurent polynomials, so that we can recapitulate (without any issues) the construction of the maps \eqref{eqn:looptolinearcurrent} and \eqref{eqn:Uoflooptolinearcurrent}, and
\item deforming this construction to encompass the extensions of $\sG_d$ and hence the twisted enveloping factorization algebras~$\UU_\theta \sG_d$.
\end{enumerate}
We undertake the steps in order.

\subsubsection{Derived functions on punctured affine space}
\label{sec:functionsofpunctureddspace}

When $d=1$, we note that
\[
\CC[z,z^{-1}] \subset \cO(\CC^*) \xto{\simeq} \Omega^{0,*}(\CC^*),
\]
and so the Laurent polynomials are a dense subalgebra of the Dolbeault complex.
When $d >1$, Hartog's lemma tells us that every holomorphic function on punctured $d$-dimensional space extends through the origin:
\[
\cO(\pAA^d) = \cO(\AA^d).
\]
This result might suggest that $\pAA^d$ is an unnatural place to seek a generalization of the loop algebra,
but such pessimism is misplaced because $\pAA^d$ is not affine 
and so its {\em derived} algebra of functions, 
given by the derived global sections $\RR \Gamma(\pAA^d, \cO)$, 
is more interesting than the underived global sections~$\cO(\pAA^d)$.

Indeed, a straightforward computation in algebraic geometry shows
\[
H^*(\pAA^{d}, \sO_{\rm alg}) = 
\begin{cases} 
0, & * \neq 0, d-1 \\ 
\CC[z_1,\ldots,z_d], & * = 0 \\ \CC[z_1^{-1},\ldots,z_d^{-1}] \frac{1}{z_1 \cdots z_d}, & * = d-1 
\end{cases}.
\]
(For instance, use the cover by the affine opens of the form $\AA^d \setminus \{z_i =0\}$.)
When $d=1$, this computation recovers the Laurent polynomials,
so we should view the cohomology in degree $d-1$ as providing the derived replacement of the polar part of the Laurent polynomials.
A similar result holds in analytic geometry, of course,
so that we have a natural map
\[
\RR \Gamma(\pAA^d, \cO_{\rm alg}) \to \RR \Gamma(\pAA^d, \cO_{\rm an}) \simeq \Omega^{0,*}(\pAA^d)
\]
that replaces our inclusion of Laurent polynomials into the Dolbeault complex on~$\pAA^d$.

For explicit constructions, it is convenient to have an explicit dg commutative algebra that models the derived global sections.
It should be no surprise that we like to work with the Dolbeault complex,
but there is also an explicit dg model $A_d$ for the algebraic version derived global sections due to Faonte-Hennion-Kapranov \cite{FHK} and based on the Jouanolou method for resolving singularities. 
In fact, they provide a model for the algebraic $p$-forms as well.

\begin{dfn}
Let $a_d$ denote the algebra  
\[
\CC[z_1,\ldots,z_d, z_1^*,\ldots,z_d^*][(z z^*)^{-1}]
\]
defined by localizing the polynomial algebra with respect to $zz^* = \sum_i z_i z^*_i$.
View this algebra $a_d$ as concentrated in bidegree $(0,0)$, 
and consider the bigraded-commutative algebra $R^{*,*}_d$ over $a_d$ that is freely generated in bidegree $(1,0)$ by elements
\[
\d z_1,\ldots , \d z_d,
\] 
and in bidegree $(0,1)$ by
\[
\d z_1^*,\ldots, \d z_d^*.
\]
We care about the subalgebra $A^{*,*}_d$ where $A^{p,m}_d$ consisting of elements $\omega \in R^{p,m}_d$ such that
\begin{itemize}
\item[(i)] the coefficient of $\d z^*_{i_1} \cdots \d z^*_{i_m}$ has degree $-m$ with respect to the $z_k^*$ variables, and
\item[(ii)] the contraction $\iota_\xi \omega$ with the Euler vector field $\xi = \sum_{i} z_i^* \partial_{z_{i}^*}$ vanishes.
\end{itemize}
This bigraded algebra admits natural differentials in both directions:
\begin{enumerate}
\item define a map $\dbar : A_d^{p,q} \to A_d^{p,q+1}$ of bidegree $(0,1)$~by
\[
\dbar = \sum_i \d z^*_i \frac{\partial}{\partial z_i^*},
\]
\item define a a map of bidegree $(1,0)$~by
\[
\partial = \sum_i \d z_i \frac{\partial}{\partial z_i} .
\]
\end{enumerate}
These differentials commute, so $\dbar \partial = \partial \dbar$,
and each squares to zero.
\end{dfn}

We denote the subcomplex with $p=0$~by 
\[
(A_d, \dbar) = (\bigoplus_{q = 0}^d A_d^{q}[-q], \dbar),
\] 
and it has the structure of a dg commutative algebra.
For $p>0$, the complex $A^{p,*}_d = (\oplus_q A^{p,q}[-q], \dbar)$ is a dg module for $(A_d, \dbar)$.

From the definition, one can guess that the variables $z_i$ should be understood as the usual holomorphic coordinates on affine space $\CC^d$ and the variables $z^*_i$ should be understood as the antiholomorphic coordinates $\zbar_i$.
The following proposition confirms that guess;
it also summarizes key properties of the dg algebra $A_d$ and its dg modules $A_{d}^{p,*}$,
by aggregating several results of \cite{FHK}.

\begin{prop}[\cite{FHK}, Section 1]
\label{prop: Ad} $\;$
\begin{enumerate}
\item
The dg commutative algebra $(A_d,\dbar)$ is a model for $\RR \Gamma(A^{d\times}, \sO^{alg})$:
\[
A_d \simeq \RR\Gamma(\AA^{d \times}, \sO^{alg}) .
\]
Similarly, $(A_d^{p,*},\dbar) \simeq \RR \Gamma(\AA^{d\times}, \Omega^{p,alg})$.
\item There is a dense map of commutative bigraded algebras
\[
\jou : A^{*,*}_d \to \Omega^{*,*}(\CC^d \setminus \{0\}) 
\]
sending $z_i$ to $z_i$, $z_i^*$ to $\Bar{z}_i$, and $\d z_i^*$ to $\d \zbar_i$, and the map intertwines with the $\dbar$ and $\partial$ differentials on both sides.
\item There is a unique $\GL_n$-equivariant residue map
\[
{\rm Res}_{z=0} : A_d^{d,d-1} \to \CC
\]
that satisfies
\[
\Res_{z=0} \left(f(z) \omega_{BM}^{alg}(z,z^*) \d z_1 \cdots \d z_d\right) = f(0)
\]
for any $f (z) \in \CC[z_1,\ldots,z_d]$. 
In particular, for any $\omega \in A^{d,d-1}_d$,
\[
{\rm Res}_{z=0} (\omega) = \oint_{S^{2d-1}} \jou(\omega)
\]
where $S^{2d-1}$ is any sphere centered at the origin in $\CC^d$. 
\end{enumerate}
\end{prop}

It is a straightforward to verify that the formula for the Bochner-Martinelli kernel makes sense in the algebra $A_d$.
That is, we define
\[
\omega_{BM}^{alg} (z,z^*) = \frac{(d-1)!}{(2 \pi i)^d} \frac{1}{(zz^*)^d} \sum_{i=1}^d (-1)^{i-1} z_i^* \d z_1^* \wedge \cdots \wedge \Hat{\d z_i^*} \wedge \cdots \wedge \d z_d^*,
\]
which is an element of~$A_d^{0,d-1}$. 

\subsubsection{The sphere algebra of $\fg$}

The loop algebra $L\fg = \fg[z,z^{-1}]$ arises as an algebraic model of the mapping space $\Map(S^1,\fg)$,
which obtains a natural Lie algebra structure from the target space~$\fg$.
For a topologist, a natural generalization is to replace the circle $S^1$, which is equal to the unit vectors in $\CC$, by the sphere $S^{2d-1}$, which is equal to the unit vectors in $\CC^d$.
That is, consider the ``sphere algebra'' of $\Map(S^{2d-1},\fg)$.
An algebro-geometric sphere replacement of this sphere is the punctured affine $d$-space $\pAA^{d}$ or a punctured formal $d$-disk,
and so we introduce an algebraic model for the sphere algebra.

\begin{dfn}
For a Lie algebra $\fg$, the {\em sphere algebra} in complex dimension $d$ is the dg Lie algebra~$A_d \otimes \fg$.
Following \cite{FHK} we denote it by~$\fg^\bullet_d$.
\end{dfn}

There are natural central extensions of this sphere algebra as {em $L_\infty$ algebras},
in parallel with our discussion of extensions of the local Lie algebras.
For any $\theta \in \Sym^{d+1}(\fg^*)^\fg$, Faonte-Hennion-Kapranov define the cocycle
\[
\label{fhk cocycle}
\begin{array}{cccc}
\theta_{FHK} : & (A_d \tensor \fg)^{\tensor (d+1)} & \to & \CC\\ 
& a_0 \otimes \cdots \otimes a_d & \mapsto & \Res_{z=0} \theta(a_0,\partial a_1,\ldots,\partial a_d)
\end{array}.
\]
This cocycle has cohomological degree $2$ and so determines an unshifted central extension as $L_\infty$ algebras of~$A_d \tensor \fg$:
\beqn\label{gdt}
\CC \cdot K \to \widetilde{\fg}^\bullet_{d, \theta} \to A_d \tensor \fg .
\eeqn
Our aim is now to show how the Kac-Moody factorization algebra $\UU_\theta \sG_d$ is related to this $L_\infty$ algebra,
which is a higher-dimensional version of the affine Kac-Moody Lie algebras. 

\subsubsection{The case of zero level}

Here we will consider the higher Kac-Moody factorization algebra on $\CC^d \setminus \{0\}$ ``at level zero," namely the factorization algebra $\UU(\sG_{\CC^d \setminus\{0\}})$.
In this section we will omit $\CC^d \setminus \{0\}$ from the notation, and simply refer to the factorization algebra by $\UU(\sG_d)$. 
Our construction will follow the model case outlined in the introduction to this section.
Recall that $r: \pAA^d \to (0,\infty)$ is the radial projection map that sends $(z_1,\ldots,z_d)$ to its length $\sqrt{z_1\zbar_1 + \cdots z_d \zbar_d}$.

\begin{lem}
There is a map of sheaves of dg commutative algebras on~$\RR_{>0}$
\[
\pi: \Omega^* \to r_* \Omega^{0,*}
\]
sending a form $\alpha$ to the $(0,*)$-component of its pullback $r^*\alpha$.
\end{lem}

This result is straightforward since the pullback $r^*$ denotes a map of dg algebras to $r_* \Omega^{*,*}$ and we are simply postcomposing with the canonical quotient map of dg algebras $\Omega^{*,*} \to \Omega^{0,*}$. 

We also have a map of dg commutative algebras $A_d \to \Omega^{0,*}(U)$ for any open set $U \subset \pAA^d$,
by postcomposing the map $\jou$ of proposition~\ref{prop: Ad} with the restriction map.
We abusively denote the composite by $\jou$ as well.
Thus we obtain a natural map of dg commutative algebras
\[
\pi_A: \Omega^* \otimes A_d \to r_* \Omega^{0,*}
\]
sending $\alpha \otimes \omega$ to $\pi(\alpha) \wedge \jou(\omega)$.
By tensoring with $\fg$, we obtain the following.

\begin{cor}
There is a map of sheaves of dg Lie algebras on~$\RR_{>0}$
\[
\pi_{\fg, d}: \Omega^* \otimes \fg^\bullet_d \to r_* (\Omega^{0,*}\otimes \fg) = r_*(\sG_d^{sh})
\]
sending $\alpha \otimes x$ to $\pi(\alpha) \otimes x$.
\end{cor}

Note that $\Omega^* \otimes \fg^\bullet_d = \Omega^* \otimes A_d \otimes \fg$, so $\pi_{\fg, d}$ is simply $\pi_A \otimes \id_\fg$.

This map preserves support and hence restricts to compactly-supported sections.
In other words, we have a map between the associated cosheaves of complexes (and precosheaves of dg Lie algebras).
In summary, we have shown our key result.

\begin{prop}
\label{prop: fact lie}
The map
\[
\pi_{\fg, d}: \Omega^*_{\RR_{>0},c} \otimes \fg^\bullet_d \to r_*\sG_d 
\] 
is a map of precosheaves of dg Lie algebras.
It determines a map of factorization algebras
\[
\cliels(\pi_{\fg, d}) : \UU\left(\Omega^{*}_{\RR_{>0}} \tensor \fg^\bullet_d\right) \to r_*\left(\UU \sG_d \right) .
\]
\end{prop}

The map of factorization algebras follows from applying the functor $\clieu_*(-)$ to the map $\pi_{\fg, d}$;
this construction commutes with push-forward by inspection. 

Both maps are dense in every cohomological degree with respect to the natural topologies on these vector spaces,
leading to the following observation.

%\owen{This phrasing isn't great. Can you think of something better?}

\begin{cor}
By Theorem \ref{thm:knudsen} of Knudsen, 
the enveloping $E_1$ algebra of the sphere algebra $\fg^\bullet_d$ is dense inside the pushforward factorization algebra $r_*\left(\UU \sG_d \right)$.
\end{cor}

\subsubsection{The case of non-zero level}

Pick a $\theta \in \Sym^{d+1}(\fg^*)^\fg$. 
This choice determines a higher Kac-Moody factorization algebra $\UU_\theta \sG_d$,
and we would like to produce maps akin to those of Proposition~\ref{prop: fact lie}.

The simplest modification of the level zero situation is to introduce a central extension of the precosheaf
\[
\mathtt{G}_d = \Omega^*_{\RR_{>0},c} \otimes \fg^\bullet_d
\] 
as a precosheaf of $L_\infty$ algebras on $\RR_{>0}$,
with the condition that this extension intertwines with the extension $r_*\sG_{d,\theta}$ of~$r_* \sG_{d}$.
In other words, we need a map 
\[
\xymatrix{
0 \ar[r] & \CC \cdot K [-1]  \ar[d]^{=} \ar[r] & \mathtt{G}_{d,\Theta'} \ar[d]^{{\Hat{\pi}}_{\fg,d}} \ar[r] & \mathtt{G}_d \ar[d]^{\pi_{\fg, d}} \ar[r] & 0 \\
0 \ar[r] & \CC \cdot K [-1] \ar[r] & r_*\sG_{d,\theta} \ar[r] & r_* \sG_{d} \ar[r] & 0 .
}
\]
of central extensions of $L_\infty$ algebras.
This condition fixes the problem completely, 
because we simply pull back the extension defining $r_*\sG_{d,\theta}$.
Let us extract an explicit description,
which will be useful later.
On an open $U \subset \RR_{>0}$, the extension for $r_*\sG_{d,\theta}$ is given by an integral
\[
\int_{r^{-1}(U)} \theta(\alpha_0,\partial \alpha_1,\ldots,\partial \alpha_d) = \int_U \int_{S^{2d-1}} \theta(\alpha_0,\partial \alpha_1,\ldots,\partial \alpha_d)
\]
that can be factored into a double integral. 
This formula indicates that $\Theta'$ must be given by the cocycle whose value on elements $\phi_i \otimes a_i \in \Omega^*_c \otimes \fg^\bullet_d$~is
\begin{align*}
\Theta'(\phi_0 \otimes a_0, \ldots, \phi_d \otimes a_d)
&= \int_U \int_{S^{2d-1}} \theta(\pi(\phi_0) \wedge \jou(a_0),\partial (\pi(\phi_1) \wedge \jou(a_1)),\ldots,\partial(\pi(\phi_d) \wedge \jou(a_d))) \\
%&= \theta_{FHK}(a_0,\ldots,a_d) \int_U \phi_1 \wedge \cdots \wedge \phi_d.
\end{align*}
We thus obtain the following result.

\begin{lem} 
For $\theta \in \Sym^{d+1}(\fg^*)^\fg$,
let $\mathtt{G}_{d,\theta}$ denote the precosheaf of $L_\infty$ algebras obtained by extending $\mathtt{G}_d$ by the cocycle
\[
(\phi_0 \otimes a_0, \ldots, \phi_d \otimes a_d) \mapsto \int_U \int_{S^{2d-1}} \theta(\pi(\phi_0) \wedge \jou(a_0),\partial (\pi(\phi_1) \wedge \jou(a_1)),\ldots,\partial(\pi(\phi_d) \wedge \jou(a_d))).
\]
By construction, there is a canonical map 
\[
\pi_{\fg,d,\theta}: \mathtt{G}_{d,\theta} \to r_* \sG_{d,\theta}
\]
of precosheaves of $L_\infty$ algebras on $\RR_{>0}$, 
and hence there is a map of factorization algebras
\[
\UU(\pi_{\fg,d,\theta}): \UU_{\theta} \mathtt{G}_d \to r_* \UU_\theta \sG_d.
\]
\end{lem}

The maps remain degreewise dense, but now we are working with a twisted enveloping factorization algebra,
which is slightly different in flavor than Knudsen's construction.
The central parameter $K$ parametrizes, in fact, a family of $E_1$ algebras that specializes at $K=0$ to the enveloping $E_1$ algebra of the sphere algebra~$\fg^\bullet_d$.

\begin{cor}
There is a family of $E_1$ algebras over the affine line ${\rm Spec}(\CC[K])$ with the enveloping $E_1$ algebra of the sphere algebra $\fg^\bullet_d$ at the origin.
This family is dense within the pushforward $r_*\left(\UU_\theta \sG_d \right)$.
\end{cor}

\subsubsection{A comparison with the work of Faonte-Hennion-Kapranov}
\label{sec:FHK}

There is a variant of the preceding result that is particularly appealing in light of \cite{FHK},
which is to provide a map of factorization algebras on the positive reals
\[
\UU(\Tilde{\pi}_{\fg,d,\theta}): \UU(\Omega^*_c \otimes \Tilde{\fg}^\bullet_{d,\theta}) \to r_* \sG_{d,\theta},
\] 
where the source is the factorization algebra encoding the enveloping $E_1$ algebra of $\Tilde{\fg}^\bullet_{d,\theta}$.
Specializing the central parameters to zero on both sides must recover the map $\UU\pi_{\fg, d}$ of Proposition~\ref{prop: fact lie}.
Such a map has two connected consequences:
\begin{enumerate}
\item It shows that the higher current Lie algebras $\Tilde{\fg}^\bullet_{d,\theta}$ of \cite{FHK} ``control'' our twisted current factorization algebras $\sG_{d,\theta}$ in the same way that the affine Kac-Moody Lie algebras control their vertex algebras.
\item It shows that our factorization algebras $\sG_{d,\theta}$ know the information encoded by the Lie algebras $\Tilde{\fg}^\bullet_{d,\theta}$.
% introduced in~\cite{FHK}.
\end{enumerate}
In short this map provides a conduit for transferring insights between derived algebraic geometry (as represented by the \cite{FHK} approach) and quantum field theory (as represented by ours).

\begin{rmk}
Before embarking on the construction of the map,
we remark that it was a pleasant surprise to come upon \cite{FHK} 
and to find that they had explored terrain that we had approached from the direction exposed in this paper,
i.e., the higher dimensional generalization of results from \cite{CG1}.
Their Jouanolou model $A_d$ gave a more explicit and more tractable analogue to Laurent polynomials and hence allowed us to sharpen our results into something more punchy,
and their discussion of the global derived geometry verified natural guesses, 
which were beyond our technical powers.
Although we had found the same extensions, 
our explanations were based on finding an explicit generalization of the $d=1$ formula,
with confirmation arising from Feynman diagram computations.
By contrast, Faonte, Hennion, and Kapranov give a beautiful structural explanation via cyclic homology,
which resonates with our physical view of large $N$ limits
We come back to this structure in more detail in Section~\ref{sec: largeN}.
We thank Faonte, Hennion, and Kapranov for inspiring and enlightening conversations and correspondence on these subjects.
\end{rmk}

Constructing the map requires overcoming two issues.
First, note that 
\[
\Tilde{\mathtt{G}}_{d,\theta} = \Omega^*_c \otimes \Tilde{\fg}^\bullet_{d,\theta}
\]
can be viewed as an extension
\[
\Omega^*_c \otimes \CC \to \Tilde{\mathtt{G}}_{d,\theta} \to \mathtt{G}_d
\]
of precosheaves of $L_\infty$ algebras on $\RR_{>0}$.
By contrast, $r_* \sG_{d,\theta}$ is an extension by the constant precosheaf $\CC K[-1]$.
There is, however, a natural map of precosheaves
\[
\int: \Omega^*_c \to \CC[-1]
\]
to employ, since integration is well-defined on compactly-supported forms.
This map indicates the shape of the underlying map of short exact sequences.

The second issue looks more serious:
the two cocycles at play seem different at first glance.
The pushforward $r_* \sG_{d,\theta}$ uses a cocycle whose behavior on the image under $\pi_{\fg, d}$ is given~by
\begin{align*}
\Theta_{push}(&\phi_0 \otimes a_0, \ldots, \phi_d \otimes a_d) \\
&= \int_U \int_{S^{2d-1}} \theta(\pi(\phi_0) \wedge \jou(a_0),\partial (\pi(\phi_1) \wedge \jou(a_1)),\ldots,\partial(\pi(\phi_d) \wedge \jou(a_d))),
\end{align*}
where we use  elements of the form $\phi_i \otimes a_i \in \Omega^*_c(U) \otimes \fg^\bullet_d$ with $U$ an open subset of $\RR_{>0}$.
On the other hand, on those same elements, the FHK extension is given~by
\begin{align*}
\Theta_{FHK}(\phi_0 \otimes a_0, &\ldots, \phi_d \otimes a_d) \\
&= (\phi_0 \wedge \cdots \wedge \phi_d) \int_{S^{2d-1}} \theta(\jou(a_0),\partial (\jou(a_1)),\ldots,\partial(\jou(a_d))).
\end{align*}
(Note that in the FHK case, we do not integrate over $U$ because we extend by $\Omega^*_c$.)
The key difference here is that the FHK extension does not involve applying $\partial$ to the $(0,*)$-components of the pulled back forms $r^* \phi_i$.
It separates the $\phi_i$ and $a_i$ contributions,
whereas the other cocycle mixes them.
The tension is resolved by showing these cocycles are cohomologous.

\begin{lem}
There is a cochain $\eta$ for $\mathtt{G}_d$ such that
\[
\Theta_{push} = \int \Theta_{FHK} + \d \eta,
\]
where $\d$ here denotes the differential on the Lie algebra cochains of~$\mathtt{G}_d$.
\end{lem}

\begin{proof}
We note that the Lie algebra $\fg$ and the invariant polynomial $\theta$ play no substantive role in the problem.
The issue here is about calculus.
Hence it suffices to consider the case that $\fg$ is the one-dimensional abelian Lie algebra and $\theta$ is the unique-up-to-scale monomial of degree $d+1$ (i.e., ``$x^{d+1}$'').

Let
\[
E = r \frac{\partial}{\partial r}
\]
denote the Euler vector field on $\RR_{>0}$,
and let
\[
\d \vartheta = \sum_i \frac{\d z_i}{z_i} 
\]
denote a $(1,0)$-form on $\pAA^d = \CC^{d} \setminus 0$. 

For concision we express the element $\varphi_i \otimes a_i $ in $\Omega^*_c(U) \otimes A_d $ by $\varphi_i a_i$.
We now define
\[
\eta(\varphi_0 a_0,\ldots,\varphi_da_d) = \sum_{i=1}^d \left(\int_U \varphi_0 \left(\iota_{E} \varphi_i \right) \varphi_1 \cdots \Hat{\varphi_i} \cdots \varphi_d\right)\left(\oint \left(a_0 a_i \d \vartheta\right) \partial a_1 \cdots \Hat{\partial a_i} \cdots \partial a_d \right)  .
\]
It is a straightforward exercise in integration by parts and the bigrading of Dolbeault forms to verify that $\eta$ cobounds the difference of the cocycles.
\end{proof}

With this explicit cochain $\eta$ in hand, we can produce the desired map.

\begin{prop}
There is an $L_\infty$ map of $L_\infty$ algebras
\[
\Tilde{\pi}_{\fg,d,\theta}: \Omega^*_c \otimes \Tilde{\fg}^\bullet_{d,\theta} \rightsquigarrow r_* \sG_{d,\theta},
\]
by which we mean there is a sequence of multilinear maps
\[
\Tilde{\pi}_{\fg,d,\theta}\<n\>: \prod_{i=1}^n \Omega^*_c \otimes \Tilde{\fg}^\bullet_{d,\theta} \to r_* \sG_{d,\theta},
\]
that have degree $2-n$ and are skew-symmetric and intertwine with the $L_\infty$ brackets on both sides (cf. \cite{KonSoi, LodVal}).
The terms $\Tilde{\pi}_{\fg,d,\theta}\<n\>$ vanish for $n \neq 1, d+1$.
The $n =1$ map fits into the commuting diagram of short exact sequences
\[
\xymatrix{
0 \ar[r] & \Omega^*_c \cdot K [-1]  \ar[d]^{\int} \ar[r] & \Omega^*_c \otimes \Tilde{\fg}^\bullet_{d,\theta} \ar[d]^{\Tilde{\pi}_{\fg,d,\theta}\<n\>} \ar[r] & \mathtt{G}_d \ar[d]^{\pi_{\fg, d}} \ar[r] & 0 \\
0 \ar[r] & \CC \cdot K [-1] \ar[r] & r_*\sG_{d,\theta} \ar[r] & r_* \sG_{d} \ar[r] & 0 .
}
\]
The $n=d+1$ map sends the $d+1$-tuple $(\phi_0 \otimes a_0, \ldots, \phi_d \otimes a_d)$ to
\[
\eta(\phi_0 \otimes a_0, \ldots, \phi_d \otimes a_d).
\]

This $L_\infty$ map is equivalent to giving a map of dg conilpotent cocommutative coalgebras on the Chevalley-Eilenberg chains of these $L_\infty$ algebras,
which in fact provides a map
\[
\UU(\Tilde{\pi}_{\fg,d,\theta}): \UU(\Omega^*_c \otimes \Tilde{\fg}^\bullet_{d,\theta}) \to r_* \UU_\theta\sG_d
\]
of factorization algebras.
\end{prop}

\begin{proof}
Note that for our $L_\infty$ algebras, the only nontrivial brackets are $\ell_1$, $\ell_2$, and $\ell_{d+1}$.
We already know that the $n=1$ map intertwines with $\ell_1$ and $\ell_2$ brackets,
as it does modulo the central extensions.
We can thus set the maps for $n=2,\ldots, d$ to zero. 
The first nontrivial issue arises at $n=d+1$, as the $n=1$ map does not intertwine the $\ell_{d+1}$ brackets.
The defining property of $\eta$, however, ensures that $\Tilde{\pi}_{\fg,d,\theta}\<d+1\>$ corrects the failure.
Hence we may set the maps for $n > d+1$ to zero as well.
\end{proof}

\begin{cor}
The enveloping $E_1$ algebra of $\fg^\bullet_{d,\theta}$ is dense inside the pushforward $r_* \UU_\theta\sG_d$.
\end{cor}

\subsection{Compactifying along tori} 

There is another direction that one may look to extend the notion of affine algebras to higher dimensions.
The affine algebra is a central extension of the loop algebra on $\fg$. 
Instead of looking at higher dimensional sphere algebras, one can consider higher {\em torus} algebras, i.e., iterated loop algebras:
\[
L^d \fg = \CC[z_1^{\pm}, \cdots, z_d^{\pm}] \tensor \fg .
\]
These iterated loop algebras are algebraic versions of the torus mapping space 
\[
{\rm Map}(S^1 \times \cdots \times S^1, \fg).
\] 
We now explore what information the Kac-Moody factorization algebras encode about extensions of such iterated loop algebras.

To do this, we study the Kac-Moody factorization algebras on the complex manifold $(\CC^\times)^d$, 
which is an algebro-geometric version of the torus $(S^1)^{d}$.  
As with the punctured affine space $\pAA^d$, we compactify by pushing forward to $(\RR_>0)^d$ along a radial projection map
\[
\begin{array}{cccc}
\vec{r} : & (\CC^\times)^d & \to & (\RR_{>0})^d\\
& (z_1,\ldots,z_d) & \mapsto & (|z_1|, \cdots, |z_d|)
\end{array}.
\]
The preimage of a point $(r_1,\ldots,r_d)$ is a $d$-fold product of circles, and
the preimage of an open $d$-cube is a polyannulus---a $d$-fold product of annuli.
Observe that on a polyannulus $U$, the underived and derived algebras of functions coincide,
\[
\Gamma(U, \cO) \overset{\simeq}{\hookrightarrow} \RR\Gamma(U,\cO),
\]
as $U$ is a Stein manifold because it is a product of Stein manifolds.
Similarly, the scheme $(\AA^1 \setminus \{0\})^d$ is affine and so its structure sheaf has no higher cohomology:
\[
\RR\Gamma((\AA^1 \setminus \{0\})^d,\cO) \simeq \CC[z_1, z_1^{-1}, \ldots, z_d,z_d^{-1}].
\]
Note that the iterated loop algebras $L^d \fg$ appear precisely by tensoring $\fg$ with functions on this product of punctured affine lines.
Thus, in contrast to $\pAA^d$, we seem to be able to work in an underived setting.

This impression is misleading, however, in the sense that it ignores some additional algebraic structure that naturally appears at the level of current algebras:
there is an $E_d$ algebra that sits densely inside the pushforward~$\vec{r}_* \sG_d$.

\begin{lem}
There is a map 
\[
\rho_{d}: \Omega^*  \to \vec{r}_* \Omega^{0,*}
\]
of sheaves of dg commutative algebras on $(\RR_{>0})^d$ sending a form $\alpha$ to the projection of the pulled back $\vec{r}^* \alpha$ onto its $(0,*)$-components.
\end{lem}

As algebraic functions sit inside holomorphic functions and hence inside the Dolbeault complex,
there is a map of dg commutative algebras
\[
\CC[z_1, z_1^{-1}, \ldots, z_d,z_d^{-1}] \to \Omega^{0,*}(U)
\]
for any open $U \subset (\CC\setminus\{0\})^d$.
There is thus a map 
\[
\rho_{d}': \Omega^* \otimes \CC[z_1, z_1^{-1}, \ldots, z_d,z_d^{-1} \to \vec{r}_* \Omega^{0,*}
\]
of dg commutative algebras.
We tensor with $\fg$ to obtain the following result.

\begin{lem}
There is a map
\[
\rho_{d,\fg}: \Omega^* \otimes L^d\fg \to \vec{r}_* \sG_d^{sh}
\]
of sheaves of dg Lie algebras on $(\RR_{>0})^d$ sending an element $\alpha \otimes x$ to $\rho_d(\alpha) \otimes x$.s
As this map preserves support, it restricts to a map  
\[
\rho_{d,\fg}: \Omega^*_c \otimes L^d\fg \to \vec{r}_* \sG_d
\]
of precosheaves of dg Lie algebras on~$(\RR_{>0})^d$.
\end{lem}

By taking Chevalley-Eilenberg chains, we obtain a statement at the level of factorization algebras.

\begin{cor}
There is a map
\[
\UU(\rho_{d,\fg}): \UU(\Omega^*_c \otimes \fg) \to \vec{r}_* \UU\sG_d
\]
of factorization algebras on $(\RR_{>0})^d$.
As the source is locally constant, it corresponds to an $E_d$ algebra, 
which is the enveloping $E_d$ algebra of $L^d\fg$, by Knudsen's theorem.
\end{cor}

This map has dense image in each degree, and so we see that the enveloping $E_d$ algebra of the iterated loop algebra $L^d \fg$ ``controls'' the pushforward $\vec{r}_* \UU\sG_d$ in this sense.

\begin{rmk}
When $d=1$ one can understand the radially ordered products of operators by evaluating these current factorization algebras on nested annuli.
For $d >1$ one can read likewise understand interesting phenomena about operator products by evaluating these current factorization algebras these polyannuli.
In particular, the connection with $E_d$ algebra indicates that there is a (possibly nontrivial) $1-d$-shifted Poisson bracket between operators, even at the level of cohomology.
\end{rmk}

In the case of the extended Lie algebras $\sG_{d,\theta}$, we note that one can pull back the extension along the map $\rho_{d,\fg}$ to determine an extension of $\Omega^*_c \otimes L^d\fg$ as a precosheaf of $L_\infty$ algebras.
One can view this extension as extending $L^d\fg$ as an $L_\infty$ algebra:
\[
\CC[d-1] \to \Tilde{L^d \fg}_\theta \to L^d \fg,
\]
It is essentially immediate from the definitions that the cocycle~is
\[
L^d \theta(f_0 \tensor x_0)\tensor \cdots \tensor (f_d \tensor x_d) = \theta(x_0,\ldots,x_d)  \oint_{|z_1| = 1} \cdots \oint_{|z_d| = 1} f_0 \d f_1 \cdots \d f_d
\]
where $f_i \in \CC[z_1^{\pm 1}, \ldots, z_d^{\pm 1}]$ and $x_i \in \fg$.
This formula is just an iterated version of the usual residue pairing.

This extension then determines a twist of the enveloping $E_d$ algebra, as well.
By techniques analogous to what we did in comparing with \cite{FHK},
one can show the following.

\begin{prop}
For $\theta \in \Sym^{d+1}(\fg^*)^\fg$,
there is a map of factorization algebras
\[
\rho_{d,\fg,\theta}: \UU(\Omega^*_c \otimes \Tilde{L^d \fg}_\theta) \to \rho_* \UU_\theta \sG_d
\]
that has dense image in each degree.
\end{prop}

In this sense the enveloping $E_d$ algebra of $\Tilde{L^d \fg}_\theta$ controls the twisted enveloping factorization algebra.

\section{The holomorphic charge anomaly} \label{sec: qft}

In this section, we change our focus and exhibit a natural occurrence of the Kac-Moody factorization algebra as a symmetry of a simple class of higher dimensional quantum field theories. 
This example generalizes the free field realization of the affine Kac-Moody algebra as a subalgebra of differential operators on the loop space. 

Our approach is through the general machinery of perturbative quantum field theory developed by Costello \cite{CosRenorm} and Costello-Gwilliam \cite{CG1,CG2}.
We study the quantization of a particular {\em free} field theory, which makes sense in any complex dimension.
Classically, the theory depends on the data of a $G$-representation, and the holomorphic nature of the theory allows us the classical current algebra $\Cur^{\cl}(\sG_X)$ at ``zero level" to act as a symmetry. 
We find that upon quantization, the symmetry is broken, but in a way that we can measure by an explicit anomaly, i.e., local cocycle for $\sG_X$. 
This failure leads to a symmetry of the quantum theory via the quantum current algebra $\Cur^{\q}(\sG_X)$ twisted by this cocycle.  

\subsection{Holomorphic bosons}

We introduce a classical field theory on any complex manifold $X$ in the BV formalism whose equations of motion, in part, include holomorphic functions on $X$.
When the complex dimension is $d = 1$, our theory is identical to the chiral $\beta\gamma$ system, 
which is a bosonic version of the familiar $bc$ system in conformal field theory. 
In dimensions $d=2$ and $d=3$, this class of theories is still of physical importance.
They are equivalent to minimal twists of supersymmetric matter multiplets. 

To start, we fix a finite dimensional $\fg$-module $V$ and an integer $d > 0$.
There are two fields, a field $\gamma : \CC^d \to V$, given by a smooth function into $V$, and
a field $\beta \in \Omega^{d,d-1}(\CC^d, V^\vee)$, 
given by a differential form of Hodge type $(d,d-1)$, valued in the dual vector space $V^\vee$. 
The action functional describing the classical field theory~is
\beqn\label{actionfnl}
S(\gamma,\beta) = \int \<\beta, \dbar\gamma\>_V
\eeqn
where $\<-,-\>_V$ denotes the evaluation pairing between $V$ and its dual. 
The classical equations of motion of this theory are 
\[
\dbar \beta = 0 =\dbar \gamma
\]
and hence pick out pairs $(\gamma,\beta)$ that are holomorphic. 

The symmetry we consider comes from the $\fg$-action on $V$. 
It extends, in a natural way, to an action of the ``gauged'' Lie algebra $C^\infty(X, \fg)$ on the $\gamma$ fields: an element $x(z,\zbar) \in C^\infty(X,\fg)$ acts simply by $x(z,\zbar) \cdot \gamma(z,\zbar)$ where the dot indicates the pointwise action via the $\fg$-module structure on $V$. 
There is a dual action on the $\beta$ fields.
This Lie algebra action is compatible with the action functional (\ref{actionfnl})---that is, it preserves solutions to the equations of motion---precisely when $x(z,\zbar)$ is holomorphic: $\dbar x(z,\zbar) = 0$. 
In other words, the natural Lie algebra of symmetries is $\cO_X \otimes \fg$, the holomorphic functions on $X$ with values in~$\fg$.

Notice that the original action functional (\ref{actionfnl}) has an ``internal symmetry'' via the gauge transformation
\[
\beta \mapsto \beta + \dbar \beta' 
\]
with $\beta'$ an arbitrary element of $\Omega^{d,d-2} (X, V^*)$. 
Thus, the space $\Omega^{d,d-2} (X, V^\vee)$ provide ghosts in the BRST formulation of this theory. 
Moreover, there are ghosts for ghosts $\beta'' \in \Omega^{d,d-3}(X , V^\vee)$, and so on.
Together with all of the antifields and antighosts, the full theory consists of two copies of a Dolbeault complex.
The precise definition is the following.

\begin{dfn}
In the BV formalism the {\em classical $\beta\gamma$ system} on the complex manifold $X$ has space of fields
\[
\sE_V = \Omega^{0,*}(X , V) \oplus \Omega^{d,*}(X , V^*)[d-1],
\]
with the linear BRST operator given by $Q = \dbar$.
We will write fields as pairs $(\gamma,\beta)$ to match with the notation above.
There is a $(-1)$-shifted symplectic pairing is given by integration along $X$ combined with the evaluation pairing between $V$ and its dual: 
\[
\<\gamma, \beta\> = \int_X \<\gamma, \beta\>_V.
\] 
The action functional for this free theory is thus
\[
S_V (\beta,\gamma) = \int_X \<\beta, \dbar \gamma\>_{V} .
\]
\end{dfn}

\begin{rmk}
As usual in homological algebra, the notation $[d-1]$ means we shift that copy of the fields down by $d-1$. 
Note that the elements in degree zero (i.e., the ``physical'' fields) are precisely maps $\gamma : X \to V$ and sections $\beta \in\Omega^{d,d-1} (X ; V^\vee)$, just as in the initial description of the theory. 
The gauge symmetry $\beta \to \beta + \dbar \beta'$ has naturally been incorporated into our BRST complex (which only consists of a linear operator since the theory is free).
We note that the pairing only makes sense when at least one of the inputs is compactly-supported or $X$ is closed;
but, as usual in physics, it is the Lagrangian density that is important, rather than the putative functional it determines.
\end{rmk}

\begin{rmk}
This theory is a special case of a nonlinear $\sigma$-model, where the linear target $V$ is replaced by an arbitrary complex manifold $Y$.
When $d=1$ this theory is known as the (classical) curved $\beta\gamma$ system and has received extensive examination \cite{WittenCDO, WG2, Nek, GGW};
when a quantization exists, the associated factorization algebra of quantum observables encodes the vertex algebra known as chiral differential operators of $Y$.
The second author's thesis \cite{BWthesis} examines the theories when $d>1$ and uncovers a systematic generalization of chiral differential operators.
\end{rmk}

In parallel with our discussion above, once we include the full BV complex, 
it is natural to encode the symmetry $\cO_X \otimes \fg$ by the action of  the {\em dg Lie algebra} $\sG_X^{sh} = \Omega^{0,*}(X, \fg)$. 
The action by $\sG_X^{sh}$ extends to a natural action on the fields of the $\beta\gamma$ system in such a way that the shifted symplectic pairing is preserved. 
In other words, $\alpha$ determines a symplectic vector field on the space of fields.

This vector field is actually a Hamiltonian vector field, 
and we will encode it by an element $\alpha \in \sG_X^{sh}$ by a {\em local} functional $I_\alpha^{\sG} \in \oloc(\sE_V)$. 
It is a standard computation in the BV formalism to verify the following.

\begin{dfn/lem}
The {\em $\sG_X$-equivariant $\beta\gamma$ system} on $X$ with values in $V$ is defined by the local functional
\[
I^{\sG}(\alpha, \beta, \gamma) = \int \<\beta, \alpha \cdot \gamma\>_V \in \oloc(\sE_V \oplus \sG_X[1]) .
\]
This functional satisfies the $\sG_X$-equivariant classical master equation
\[
(\dbar + \d_{\sG}) I^{\sG} + \frac{1}{2} \{I^{\sG}, I^{\sG}\} = 0 .
\] 
\end{dfn/lem}

The classical master equation encodes the claim that the function $I^{\sG}$ defines a dg Lie algebra action on the theory $\sE_V$. 
In particular, $I^{\sG}$ determines a map of sheaves of dg Lie algebras 
\[
I^{\sG} : \sG^{sh}_X \to \oloc(\sE_V)[-1],
\] 
where the Lie bracket on the right hand side is defined by the BV bracket $\{-,-\}$. 
If we post-compose with the map $\oloc(\sE_V)[-1] \to \Der_{\rm loc}(\sE_V)$ that sends a functional $f$ to the Hamiltonian vector field $\{f,-\}$,
then we find the composite is precisely the action of $\sG^{sh}_X$ on fields already specified.

We view the sum 
\[
S(\beta,\gamma) + I^{\sG}(\alpha, \beta, \gamma)
\] 
as the action functional of a field theory in which the $\alpha$ fields parametrize a family of field theories,
i.e., provide a family of backgrounds for the $\beta\gamma$ system.
We call it the equivariant classical action functional.

\noindent {\bf Note:} For the remainder of the section we will restrict ourselves to the space $X =~\CC^d$. 

\subsubsection{The $\beta\gamma$ factorization algebra}

It is the central result of \cite{CG1,CG2} that the observables of a quantum field theory form a factorization algebra on the underlying spacetime. 

For any theory, the factorization algebra of classical observables assigns to every open set $U$, the cochain complex of polynomial functions on the fields that only depend on the behavior of the fields in $U$.
(In other words, each function must have support in $U$.)  
For the $\beta\gamma$ system, the complex of classical observables\footnote{We work here with polynomial functions but it is possible to work with formal power series instead, which is typically necessary for interacting theories. We use $\Sym$ to denote polynomials and $\widehat{\Sym}$ to denote the completion, which are formal power series.} assigned to an open set $U \subset \CC^d$~is
\[
\Obs^{\cl}_V(U) = \left(\Sym \left(\Omega^{0,*}(U)^\vee \tensor V^\vee \oplus \Omega^{d,*}(U)^\vee \tensor V [-d+1]\right), \dbar\right) .
\]
As discussed following Definition~\ref{dfn: classical currents}, we use the completed tensor product when defining the symmetric products. 
It follows from the general results of Chapter 6 of \cite{CG2} that this assignment defines a factorization algebra on $\CC^d$. 

The functional $I^{\sG}$ defines a map of dg Lie algebras $I^{\sG} : \sG_d(\CC^d) \to \Obs^{\cl}_V(\CC^d)$.
(Note  that we have switched here from $\sG^{sh}_d$ to $\sG_d$, and hence are working with compactly supported $\alpha$.)
Thanks to the shifted symplectic pairing on the fields, 
the factorization algebra $\Obs_V^{\cl}$ is equipped with a 1-shifted Poisson bracket and hence a $P_0$-structure. 
In Section \ref{sec: envelopes} we also discussed how a local Lie algebra determines a $P_0$-factorization algebra via its classical current algebra. 
The classical Noether's theorem, as proved in Theorem 11.0.1.1 of \cite{CG2}, then implies that $I^{\sG}$ determines a map between these factorization algebras. 

\begin{prop}[\cite{CG2}, Classical Noether's Theorem]
\label{prop:CNT}
The assignment that sends an element $\alpha \in \Omega^{0,*}_c(U, \fg)$ to the observable
\[
\gamma \tensor \beta \in \Omega^{0,*}(U, V) \tensor \Omega^{d,*}(U, V^*) \mapsto \int_U \<\beta, \alpha \cdot \gamma\>_V
\]
determines a map of $P_0$-factorization algebras 
\[
J^{\cl} : \Cur^{\cl} (\sG_d) \to \Obs^{\cl}_V 
\]
on the manifold $\CC^d$.
\end{prop}

This formula for $J^{\cl}$ is identical to that of the local functional $I^\fg(\alpha)$ defining the action of $\sG_d$ on the $\beta\gamma$ system,
but it is only defined for compactly supported sections $\alpha$.
Note an important point here: if $\alpha$ is not compactly supported, then $I^\fg(\alpha)$ is not a functional on arbitrary fields because the density $\<\beta, \alpha \cdot \gamma\>_V$ may not be integrable.
In general, a local functional need not determine an observable on an open set since the integral may not exist.
When $\alpha$ is compactly supported on $U$, however, then $I^\sG (\alpha)$ does determine an observable on $U$, namely the observable~$J^{\cl}(\alpha)$. 
We also want to note that the map $J^{\cl}$ is quadratic.

The challenge is to extend this relationship to the quantum situation. 
Being a free field theory, the $\beta\gamma$ system admits a natural quantization and hence a factorization algebra $\Obs^{\q}_V$ of quantum observables (whose definition we recall below). 
The natural question arises whether the symmetry by the dg Lie algebra $\sG_d$ persists upon quantization. 
We are asking if we can lift $J^{\cl}$ to a ``quantum current" $J^{\q} : \Cur^\q(\sG_d) \to \Obs^\q_V$, where $\Cur^\q(\sG_d)$ is the factorization algebras of quantum currents of Definition~\ref{dfn: quantum currents}. 
The existence of this map of factorization algebras is controlled by the equivariant quantum master equation, to which we now turn.

\subsection{The equivariant quantization}

The approach to quantum field theory we use follows Costello's theory of renormalization and the Batalin-Vilkovisky formalism developed in \cite{CosRenorm}.
The formalism dictates that in order to define a quantization, it suffices to define the theory at each energy (or length) scale and to ask that these descriptions be compatible as we vary the scale.
Concretely, this compatibility is through an exact {\em renormalization group (RG) flow} and is encoded by an operator $W(P_{\epsilon < L}, -)$ acting on the space of functionals. 
The functional $W(P_{\epsilon < L},-)$ is defined as a sum over weights of graphs which is how Feynman diagrams appear in Costello's formalism.
A theory that is compatible with the RG flow is called a ``prequantization". 
In order to obtain a quantization, one must solve the quantum master equation (QME). 
For us, the quantum master equation encodes the failure of lifting the classical $\sG_d$-symmetry to one on the prequantization.

The quantization we work with follows Costello's approach quite closely, 
but we will use a sophisticated version where some of the fields are ``background'' fields and hence are not integrated over.
This allows us to study the equivariant theory we just introduced.
(This version is discussed in more depth in \cite{CG2}.) 
The two main ingredients to construct the weight are the propagator $P_{\epsilon < L}$ and the classical interaction $I^{\sG}$. 
The propagator only depends on the underlying free theory, that is, the higher-dimensional $\beta\gamma$ system. 
As above, the interaction describes how the linear currents $\sG_d$ act on the free theory. 

The construction of $P_{\epsilon<L}$, which makes sense for a wide class theories of this holomorphic flavor, can be found in Section 3.2 of~\cite{BWhol}.
For us, it is important to know that $P_{\epsilon<L}$ satisfies the following properties:

\begin{enumerate}
\item[(1)] For $0 < \epsilon < L < \infty$ the propagator 
\[
P_{\epsilon < L} \in \sE_V \Hat{\tensor} \sE_V 
\]
is a symmetric under the $\ZZ/2$-action.
Moreover, $P_{0 < \infty} = \lim_{\epsilon \to 0}\lim_{L \to \infty}$ is a symmetric element of the distributional completion $\Bar{\sE}_V \Hat{\tensor} \Bar{\sE}_V$. 

\item[(2)] 
The propagator lies in the subspace
\[
\Omega^{d,*}(\CC^d \times \CC^d, V \tensor V^*) \oplus \Omega^{d,*}(\CC^d \times \CC^d, V^* \tensor V) \subset \sE_V \Hat{\tensor} \sE_V .
\]
If we use coordinates $(z,w) \in \CC^d \times \CC^d$, the propagator has the form
\beqn
P_{\epsilon<L} = P^{an}_{\epsilon<L}(z,w) \tensor \left({\rm id}_{V} + {\rm id}_{V^*}\right)
\eeqn
where ${\rm id}_V, {\rm id}_{V^*}$ are the elements in $V \tensor V^*, V^* \tensor V$ that represent identity maps. 
Moreover, $P^{an}_{0 < \infty} (z,w)$ is the Green's function for the operator $\dbar$ on $\CC^d$:
\[
\dbar P^{an}_{0<\infty} (z,w) = \delta (z-w) .
\]

\item[(3)] Let $K_t \in C^\infty((0,\infty)_t) \tensor \sE_V \Hat{\tensor} \sE_V$ be the heat kernel for the Hodge Laplacian
\[
\triangle_{\rm Hodge} K_t + \frac{\partial}{\partial t} K_t = 0 .
\]
Thus, $P_{\epsilon < L}$ provides a $\dbar$-homotopy between $K_\epsilon$ and $K_L$:
\[
\dbar P_{\epsilon < L} = K_{t=L} - K_{t=\epsilon} .
\]
\end{enumerate}

To define the quantization, we recall the definition of a weight of a Feynman diagram adjusted to this equivariant context.
To simplify our discussion, we introduce the notation $\sO(\sG_d[1])$ to mean the underlying graded vector space of $\clie^*(\sG_d)$, which is the (completed) symmetric algebra on the dual of~$\sG_d$. 

For the free $\beta\gamma$ system, the homotopy RG flow from scale $L>0$ to $L'>0$ is an invertible linear map 
\beqn\label{weight1}
W(P_{L < L'} , -) : \sO(\sE) [[\hbar]] \to \sO(\sE)[[\hbar]]
\eeqn
defined as a sum of weights of graphs 
\[
W (P_{L<L'}, I) = \sum_{\Gamma} W_{\Gamma}(P_{L<L'}, I). 
\]
Here, $\Gamma$ denotes a graph, and the weight $W_\Gamma$ associated to $\Gamma$ is defined as follows.
One labels the vertices of valence $k$ by the $k$th homogenous component of the functional $I$. 
The edges of the graph are labeled by the propagator $P_{L<L'}$.
The total weight is given by iterative contractions of the homogenous components of the interaction with the propagator. 
Formally, we can write the weight as
\[
e^{W(P_{\epsilon <L}, I)} = e^{\hbar \partial_{P_{\epsilon <L}}} e^{I / \hbar}
\]
where $\partial_P$ denotes contraction with $P$. 
(For a complete definition, see Chapter 2 of~\cite{CosRenorm}.)

To define the equivariant version, we extend (\ref{weight1}) to a $\sO(\sG_d[1])$-linear map
\[
W^{\sG} (P_{L < L'} , -) : \sO(\sE \oplus \sG_d[1]) [[\hbar]] \to \sO(\sE \oplus \sG_d[1])[[\hbar]] .
\]

\begin{dfn/lem}
A {\em prequantization} of the $\sG_d$-equivariant $\beta\gamma$ system on $\CC^d$ is defined by the family of functionals $\{I^{\sG}[L]\}_{L > 0}$, where
\beqn\label{prequant}
I^{\sG} [L] = \lim_{\epsilon \to 0} W^{\sG} (P_{\epsilon<L} , I^{\sG}) .
\eeqn 
This family satisfies homotopy RG flow:  
\[
I[L'] = W(P_{L<L'}, I[L]) .
\]
for all $L < L'$.
\end{dfn/lem}

\begin{proof}
The key claim to justify is why the $\epsilon \to 0$ limit of $W^{\sG} (P_{\epsilon<L} , I^{\sG})$ exists,
since it implies immediately that we have a family of actions satisfying homotopy RG flow. 
This key claim follows from the following two intermediate results:
\begin{itemize}
\item[(1)] 
Only one-loop graphs appear in the weight expansion $W^{\sG} (P_{\epsilon < L}, I^{\sG})$. 

\item[(2)] Let $\Gamma$ be a one-loop graph.
Then
\[
\lim_{\epsilon \to 0} W^{\sG}_\Gamma(P_{\epsilon < L}, I^{\sG})
\]
exists.
\end{itemize}

Claim (1) is a direct combinatorial observation.
Recall that the weight is defined as a sum over {\em connected} graphs,
and only two types of graphs appear: 
\begin{itemize}
\item trees with a $\gamma$ leg, a $\beta$ leg, and arbitrarily many $\alpha$ legs or
\item trivalent wheels with just $\alpha$ legs.
\end{itemize}
To see this, note that the inner edges that ared labeled by the propagator $P_{\epsilon < L}$, which only depends on the fields $\beta$ and $\gamma$. 
The trivalent vertex has the form $\int \beta [\alpha, \gamma]$.
If one connects two vertices, one is left with a single $\gamma$ leg and a single $\beta$ leg but two $\alpha$ legs.
Similarly, if one connects $n$ vertices, one is left  with a single $\gamma$ leg and a single $\beta$ leg but $n$ $\alpha$ legs.
If one uses a propagator to connect $\gamma$ and $\beta$ leg, one has a wheel with $n$ $\alpha$ legs,
and no more propagators can be attached.

Claim (2) follows from Theorem 3.4 of \cite{BWhol}, which asserts that the $\epsilon \to 0$ limit of the weights is finite. 
\end{proof}

As an immediate consequence of the proof, we see that only polynomial values of $\hbar$ occur in the expansion of $I^{\sG}[L]$, indeed the answer is linear in $\hbar$. 
This fact will be used later on when we make sense of the ``free field realization" of the Kac-Moody granted by this equivariant quantization. 

\begin{cor}
For each $L > 0$, the functional $I^{\sG}[L]$ lies in the subspace 
\[
\sO(\sE \oplus \sG_d[1]) \oplus \hbar \sO(\sE \oplus \sG_d[1]) \subset \sO(\sE \oplus \sG_d[1]) [[\hbar]].
\] 
\end{cor}

To define the quantum master equation, we must introduce the BV Laplacian $\Delta_L$ and the scale $L$ BV bracket $\{-,-\}_L$. 
For $L > 0$, the operator $\Delta_L : \sO(\sE_V) \to \sO(\sE_V)$ is defined by contraction with the heat kernel $K_L$ defined above. 
Similarly, $\{-,-\}_L$ is a bilinear operator on $\sO(\sE_V)$ defined by
\[
\{I,J\}_L = \Delta_L(IJ) - (\Delta_L I)J - (-1)^{|I|} I \Delta_L J .
\] 
There are equivariant versions of each of these operators given by extending $\sO(\sG_d[1])$-linearly.
For instance, the BV Laplacian is a degree one operator of the form
\[
\Delta_L : \sO(\sE \oplus \sG_d[1]) \to \sO(\sE \oplus \sG_d[1]) .
\]
A functional $J \in \sO(\sE_V \oplus \sG_d[1])$ satisfies the $\sG_d$-{\em equivariant scale $L$ quantum master equation} (QME) if
\[
(\dbar + \d_{\sG}) J+ \frac{1}{2} \{J, J\}_L + \hbar \Delta_L J = 0 .
\]

The main object of study in this section is the {\em failure} for the quantization $I^{\sG}[L]$ to satisfy the equivariant QME. 

\begin{dfn}
The $\sG_d$-{\em equivariant charge anomaly} at scale $L$, denoted $ \Theta_V [L]$, is defined~by
\[
\hbar \Theta_V [L] = (\dbar + \d_{\sG}) I^{\sG} [L] + \frac{1}{2} \{I^{\sG}[L], I^{\sG}\}_L + \hbar \Delta I^{\sG}[L] .
\]
The operator $\d_{\sG}$ is the Chevalley-Eilenberg differential $\clie^*(\sG_d) = \left(\sO(\sG_d[1]), \d_{\sG}\right)$. 
\end{dfn}

\begin{rmk}
Since the underlying non-equivariant BV theory $\sE_V$ is free, 
in the Feynman graph expansion of $I^\sG$,
{\em none} of the external edges of any 1-loop term are labeled by $\sE_V$. 
In other words, the $\hbar$ contribution is a function only of the $\alpha$ fields (i.e., the symmetries).
Similarly, the obstruction to solving the QME is only a function of the local Lie algebra~$\sG_d$.
For this reason, the QME is automatically solved modulo the space of functionals $\clie^*(\sG_d) \subset \sO(\sE \oplus \sG_d[1])$,
if we view those as the relevant ``constants.''
We are interested, however, in making the action of $\sG_d$ ``inner'' (in the terminology of \cite{CG2}).
in which case this obstruction term is relevant.
\end{rmk}

\subsection{The charge anomaly for $\beta\gamma$}

To calculate this anomaly, we utilize a general result about the quantum master equation for holomorphic field theories formulated in \cite{BWhol}. 
In general, since the effective field theory defining the prequantization $\{I^{\sG}[L]\}$ is given by a Feynman diagram expansion, the anomaly to solving the quantum master equation is also given by a potentially complicated sum of diagrams. 
As an immediate corollary of Proposition 4.4 of \cite{BWhol} for holomorphic theories on $\CC^d$, we find that only a simple class of diagrams appear in the anomaly. 

\begin{lem}\label{lem: obs}
Let $\Theta_V[L]$ be the $\sG_d$-equivariant charge anomaly for the $\beta\gamma$ system with values in~$V$.
Then 
\begin{itemize}
\item[(1)] the limit $\Theta_V = \lim_{L \to 0} \Theta_V[L]$ exists and is a {\em local} cocycle so that $\Theta_V \in \cloc^*(\sG_d)$.
\item[(2)] This element $\Theta$ is computed by the following limit
\[
\hbar \Theta_V = \frac{1}{2} \lim_{L \to 0} \lim_{\epsilon \to 0} \sum_{\Gamma \in {\rm Wheel}_{d+1} \; , e} W_{\Gamma, e}(P_{\epsilon <L}, K_\epsilon, I^{\sG}) ,
\] 
where the sum is over all wheels of valency $(d+1)$ with a distinguished internal edge~$e$, and the weight puts $K_\epsilon$ on $e$ but the propagator on all other internal edges. 
\end{itemize}
\end{lem}

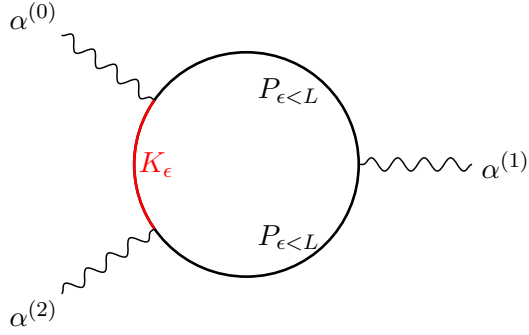
\begin{figure}
\begin{center}
\begin{tikzpicture}[line width=.2mm, scale=1.5]

%\pgfmathsetmacro{\ex}{0}
%\pgfmathsetmacro{\ey}{1}

%\draw (\ex,\ey) ++(45:.8) arc (45:-45:.8);

		\draw[fill=black] (0,0) circle (1cm);
		%\draw[fill=red] (0,0) arc (145:215:1);
		\draw[fill=white] (0,0) circle (0.99cm);
		\draw[line width=0.35mm,red] ++(145:0.995) arc (145:215:0.995);
		%\draw[red] (0,0) arc (30:60:3);

		\draw[vector](145:2) -- (145:1);
		\node at (145:2.3) {$\alpha^{(0)}$};
			%\node at (145:0.85) {$v_0$};
		\node at (60:0.75) {$P_{\epsilon<L}$};
		\node at (-60:0.75) {$P_{\epsilon<L}$};
		\draw[vector](215:2) -- (215:1cm);
		\node at (215:2.3) {$\alpha^{(2)}$};
			%\node at (215:0.85) {$v_{d}$};
		\node[red] at (180:0.8) {$K_\epsilon$};
		\draw[vector](0:2) -- (0:1);
		\node at (0:2.3) {$\alpha^{(1)}$};
			%\node at (35:0.85) {$v_{\alpha}$};
		%\node at (0:0.8) {$P_{\epsilon<L}$};
		%\node at (270:0.8) {$P_{\epsilon<L}$};
	    	\clip (0,0) circle (1cm);
\end{tikzpicture}
\caption{The diagram representing the weight $W_{\Gamma, e}(P_{\epsilon<L}, K_\epsilon, I^\fg)$ in the case $d=2$. 
On the black internal edges are we place the propagator $P_{\epsilon < L}$ of the $\beta\gamma$ system. 
On the red edge labeled by $e$ we place the heat kernel $K_\epsilon$.
The external edges are labeled by elements $\alpha^{(i)} \in \Omega^{0,*}_c(\CC^2)$.}
\label{fig:liewheel}
\end{center}
\end{figure}

This description of the local anomaly may seem obscure because it uses Feynman diagrams.
It admits, however, a very elegant algebraic characterization, using the identification of Proposition~\ref{prop: trans j}. 

\begin{prop}\label{prop: bg anomaly}
The charge anomaly for quantizing the $\sG_d$-equivariant $\beta\gamma$ system on $\CC^d$ is equal to
\[
\Theta_V = \frac{1}{(2\pi i)^d} \fj (\ch_{d+1}^{\fg}(V)),
\]
where $\fj$ is the isomorphism from Proposition~\ref{prop: trans j}.
\end{prop}

Let us unravel $\Theta_V$ in even more explicit terms:
for $\alpha_0$, \dots, $\alpha_d$ compactly-supported, $\fg$-valued Dolbeault forms,
\[
\Theta_V(\alpha_0, \ldots, \alpha_d) = \frac{1}{(2\pi i)^d} \int_{\CC^d} \Tr_V(\rho(\alpha_0) (\rho(\partial\alpha_1) \cdots \rho(\partial \alpha_d))
\]
where $\rho: \fg \to \End(V)$ denotes the action of $\fg$ on~$V$.

Based on our analysis of the local Lie algebra cohomology of $\sG_d$, 
it is clear that the obstruction must have this form, up to a scalar multiple. 
But we provide a more detailed proof.

\begin{proof}
First, we note that the element $\Theta_V \in \cloc^*(\sG_d)$ sits in the subspace of $U(d)$-invariant, holomorphic translation invariant local cocycles because both the functional $I^{\sG}$ and propagator $P_{\epsilon<L}$ are $U(d)$-invariant, holomorphic translation invariant.
By Proposition \ref{prop: trans j} we see that $\Theta_V$ must be cohomologous to a cocycle of the form
\[
(\alpha_0, \ldots, \alpha_d) \mapsto \int_{\CC^d} \theta(\alpha_0 \wedge \partial \alpha_1 \wedge \cdots \wedge \partial \alpha_d) 
\]
where $\theta$ is some element of $\Sym^{d+1}(\fg^*)^\fg$.
To use the notation of Section~\ref{sec: fact}, it is some element $\fJ_d (\theta)$. 
This cocycle factors in the following way:
\beqn
\label{composition}
\begin{tikzcd}
\left(\Omega^{0,*}_c(\CC^d) \tensor \fg\right)^{\tensor (d+1)} \ar[r,"\fa \fn"] & \left(\Omega^{0,*}_c(\CC^d) \tensor \fg\right) \tensor \left(\Omega^{1,*}_c(\CC^d)\tensor \fg\right)^{\tensor d} \ar[r, "\theta"] & \Omega^{d, *}_c(\CC^d) \ar[r, "\int"] & \CC . 
\end{tikzcd}
\eeqn
The first map is $\fa \fn : \alpha_0 \tensor \cdots \tensor \alpha_d \mapsto \alpha_0 \tensor \partial \alpha_1 \tensor \cdots \tensor \partial \alpha_d$.
The second map applies the symmetric function $\theta : \fg^{\tensor (d+1)} \to \CC$ to the Lie algebra factor and takes the wedge product of the differential forms. 

Lemma \ref{lem: obs} implies that the obstruction is given by the sum over Feynman weights associated to graphs of wheels of valency $(d+1)$.
We can identify the algebraic component, corresponding to $\theta$ in the above composition (\ref{composition}), directly from the shape of this graph. 
The propagator $P_{\epsilon<L}$ and heat kernel $K_\epsilon$ factor as
\[
P_{\epsilon<L} = P^{an}_{\epsilon<L} \tensor \left({\rm id}_{V} + {\rm id}_{V^*}\right)
\quad\quad\text{and}\quad\quad 
K_{\epsilon} = K_{\epsilon}^{an} \tensor \left({\rm id}_{V} + {\rm id}_{V^*}\right),
\]
where ${\rm id}_V, {\rm id}_{V^*}$ are the elements in $V \tensor V^*, V^* \tensor V$ representing the respective identity maps. 
The analytic factors $P^{an}_{\epsilon<L},  K_{\epsilon}^{an}$ only depend on the dimension $d$, and we recall their explicit form in Appendix~\ref{sec: feynman}. 

Each trivalent vertex of the wheel is also labeled by both an analytic factor and Lie algebraic factor. 
The Lie algebraic part of each vertex can be thought of as the defining map of the representation $\rho : \fg \to {\rm End}(V)$. 
The diagrammatics of the wheel amounts to taking the trace of the symmetric $(d+1)$st power of this Lie algebra factor. 
Thus, the Lie algebraic factor of the weight of the wheel is the $(d+1)$st component of the character of the representation
\[
{\rm ch}_{d+1}^\fg(V) = \frac{1}{(d+1)!} {\rm Tr}\left(\rho(X)^{d+1}\right) \in \Sym^{d+1}(\fg^*) .
\]

By these symmetry arguments, we know that the anomaly will be of the form $\Theta = A \fj (\ch_{d+1}^{\fg}(V))$ for some number $A \in \CC$.
In Appendix \ref{sec: feynman}, we perform an explicit calculation of this constant $A$, which depend on the specific form of the analytic propagator and heat kernel. 
\end{proof}

\subsection{The quantum observables of the $\beta\gamma$ system}

Before deducing the main consequence of the anomaly calculation, we introduce the quantum observables of the $\beta\gamma$ system. 
The quantum observables $\Obs^{\q}_V$ define a quantization of the classical observables in the sense that
as $\hbar \to 0$, they degenerate to $\Obs^\cl_V$.
More precisely, 
\[
\Obs^{\cl}_V \cong \Obs^\q_V \tensor_{\CC[\hbar]} \CC[\hbar]/(\hbar).
\]
In practice, the Costello's version of the BV formalism suggests that the quantum observables arise by 
\begin{enumerate}
\item[(a)] tensoring the underlying graded vector space of $\Obs^\cl_n$ with $\CC[[\hbar]]$ and
\item[(b)] deforming the differential to $\dbar +\hbar \Delta_L$, where $\Delta_L$ is the BV Laplacian.
\end{enumerate}
This construction actually defines a family of quantum observables, one for each length scale $L$. 
A main idea of \cite{CG2} says that by considering the collection of functionals at all length scales $L$, the observables $\Obs^\q_V$ still define a factorization algebra. 

The fact that this works is quite subtle, since naively the differential $\Delta_L$ seems to have support on all of $\CC^d$, so it is not obvious how to define the corestriction maps of the factorization algebra. 
In the case of free theories, such as the $\beta\gamma$ system, there is a way to circumvent this difficulty. 
One can work with an {\it a priori} smaller class of observables, namely those arising from smooth functionals, not distributional ones.
(A physicist might say we used ``smeared'' observables.)
The limit $\Delta = \lim_{L \to 0} \Delta_L$ then makes sense, and we just use this BV Laplacian and work at scale~0. 
This approach is developed in detail for the free $\beta\gamma$ system on $\CC$ in Chapter 5, Section 3 of~\cite{CG1}. 
The case for $\CC^d$ is essentially identical. 
This approach yields a factorization algebra~$\Tilde{\Obs}^\q_V$, as we now explain.

As shown in~\cite{CG1},
a classical result of Atiyah and Bott~\cite{AB} can be extended to show that for any complex manifold $U$, 
the inclusion
\[
\Omega^{p,*}_c(U) \subset \Bar{\Omega}^{p,*}_c(U)
\]
of compactly-supported smooth Dolbeault forms into compactly-supported smooth distributional Dolbeault forms is a quasi-isomorphism. 
Consequently we introduce the subcomplex 
\[
\Tilde{\Obs}^{\cl}_V(U) =  \left(\Sym(\Omega^{d,*}_c(U,V^*)[d] \oplus \Omega^{0,*}_c(U,V)[1]), \dbar\right)
\]
of 
\[
\Obs_V^{\cl}(U) = \left(\Sym(\Bar{\Omega}^{d,*}_c(U,V^*)[d] \oplus \Bar{\Omega}^{0,*}_c(U,V)[1]), \dbar \right).
\]
The Atiyah-Bott lemma ensures that the inclusion $\Tilde{\Obs}^{\cl}_V(U) \hookrightarrow \Obs_V^{\cl}(U)$ is a quasi-isomorphism.

The assignment $U \mapsto \Tilde{\Obs}^{\cl}_V(U)$ defines a factorization algebra on $\CC^d$, 
and so we have a quasi-isomorphism of factorization algebras $\Tilde{\Obs}^{\cl}_V \xto{\simeq} \Obs^{\cl}_V$.

\begin{dfn}\label{dfn: qobs}
The {\em smoothed quantum observables} supported on $U \subset \CC^d$ is the cochain complex
\[
\Tilde{\Obs}^\q_V(U) = \left(\Sym(\Omega^{d,*}_c(U,V^*)[d] \oplus \Omega^{0,*}_c(U,V)[1]), \dbar + \hbar \Delta\right) .
\]
\end{dfn}

By Theorem 5.3.10 of \cite{GwThesis}, the assignment $U \mapsto \Tilde{\Obs}^\q_V(U)$ defines a factorization algebra on $\CC^d$. 
Just as in the classical case, there is an induced quasi-isomorphism of factorization algebras $\Tilde{\Obs}^\q_{V} \xto{\simeq} \Obs^\q_V$,
as shown in the proof of Lemma 11.24 in~\cite{GGW}. 
Hence this smoothed version $\Tilde{\Obs}^\q_{V}$ agrees with the construction $\Obs^\q_V$ of~\cite{CG2}.

\subsection{Free field realization}

Proposition~\ref{prop:CNT} provides a factorization version of the classical Noether construction: 
there is a map of factorization algebras from the current algebra $\Cur^{\cl}(\sG_d)$ to the factorization algebra of classical observables $\Obs^\cl_V$.
It is natural to ask whether this map lifts along the ``dequantization'' map $\Obs^\q_V \to \Obs^\cl_V$, or
in other words, whether quantization preserves the symmetries.
Theorem 12.1.0.2 \cite{CG2} provides a general result about lifting classical Noether maps.
It says that if $\Theta$ is the obstruction to solving the the $\sL$-equivariant quantum master equation, 
then there is a map from the {\em twisted} quantum current algebra $\Cur^\q_{\Theta}(\sL)$ to the observables of the quantum theory. 
Thus, applied to our situation, it provides the following consequence of our Feynman diagram calculation above. 

\begin{prop}
Let $\hbar \Theta_V$ be the obstruction to satisfying the $\sG_d$-equivariant quantum master equation. 
There is a map of factorization algebras on $\CC^d$ from the twisted quantum current algebra to the quantum observables
\beqn\label{qnoether}
J^\q : \Cur_{\hbar \Theta_V}^\q (\sG_d) \to \Obs^\q_V 
\eeqn
that fits into the diagram of factorization algebras
\[
\begin{tikzcd}
\Cur^\q_{\hbar \Theta_V} (\sG_d) \ar[d, "\hbar \to 0"'] \ar[r, "J^\q"] & \Obs^\q_V \ar[d, "\hbar \to 0"] \\
\Cur^\cl (\sG_d) \ar[r,"J^\cl"] & \Obs^{\cl}_V .
\end{tikzcd}
\]
\end{prop}

The quantum current algebra $\Cur^\q_{\hbar \Theta_V}$ is a $\CC[\hbar]$-linear factorization algebra on $\CC^d$. 
It therefore makes sense to specialize the value of $\hbar$;
our convention is to take
\[
\hbar = (2 \pi i)^d .
\]
From our calculation of the charge anomaly $\Theta_V$ above, once we specialize $\hbar$, we can realize the current algebra as an enveloping factorization algebra
\[
\left. \Cur^\q_{\hbar \Theta_V} (\sG_d) \right|_{\hbar = (2\pi i)^d} \cong \UU_{\ch_{d+1}^\fg(V)} (\sG_d) .
\]
Thus, as an immediate corollary of the above proposition, $J^\q$ specializes to a map of factorization algebras
\beqn\label{free field}
J^\q : \UU_{\ch_{d+1}^\fg(V)} (\sG_d) \to \left. \Obs^\q_V \right|_{\hbar = (2\pi i)^d} .
\eeqn
We interpret this result as a {\em free field realization} of the higher Kac-Moody factorization algebra: 
the map embeds the higher Kac-Moody algebra into the quantum observables of a free theory, namely the $\beta\gamma$ system. 

\subsubsection{Sphere operators}

This formulation may seem abstract because it uses factorization algebras,
but we obtain a more concrete result once we specialize to the sphere operators. 
It realizes a representation of $\Tilde{\fg}^\bullet_{d,\theta}$ inside a Weyl algebra determined by  the $\beta\gamma$ system.

Recall from Section~\ref{sec:functionsofpunctureddspace} the dg algebra $A_d$ that provides a dg model for functions on punctured affine space.
Indeed, we have an inclusion of cochain complexes $A_d \hookrightarrow \Omega^{0,*}(\CC^d \setminus 0)$ that is dense in cohomology.

Consider the dg vector space
\[
A_d \tensor (V \oplus V^*[d-1])
\]
where $V$ is our $\fg$-representation. 
The dual pairing between $V$ and $V^*$ combined with the higher residue defines a symplectic structure $\omega_V$ on this dg vector space via
\[
\omega_V(\alpha \tensor v, \beta \tensor v^*) = \<v, v^*\>_V \oint_{S^{2d-1}} \alpha \wedge \beta\, \d^d z .
\]
This structure leads to the following dg version of the usual canonical quantization story.

\begin{dfn}
The {\em Heisenberg dg Lie algebra} $\sH_{d,V}$ of this symplectic dg vector space $A_d \tensor (V \oplus V^*[d-1])$ is the central extension
\[
\CC \to \sH_{d,V} \to A_d \tensor (V^*[d-1] \oplus V) 
\]
determined by the $2$-cocycle $\omega_V$. 
Explicitly, the nontrivial bracket is
\[
[c, b] =  \oint_{S^{2d-1}} \<c \wedge b\>_V\, \d^d z,
\]
where $\<c \wedge b\>_V$ refers to taking the wedge product in $A_d$ together with the pairing between $V$ and its dual. 
\end{dfn}

The universal enveloping algebra $U(\sH_{d,V})$ is a dg version of the Weyl algebra.
We think of this algebra as an algebraic replacement of differential operators on the mapping space $\Map(S^{2d-1},V)$.

The algebra $U(\sH_{d,V})$ is our algebraic replacement for the $E_1$ algebra associated to the pushforward factorization algebras $r_* \Obs^\q_{V}$. 
Indeed, just as in the case of the current algebra, there is a map of $E_1$ algebras
\[
U(\sH_{d,V}) \hookrightarrow r_* \Obs^\q_{V}|_{\hbar = (2 \pi i)^d}
\]
that is dense at the level of cohomology. 

In consequence there is a more concrete version of free field realization.

\begin{cor} \label{cor: free}
The map (\ref{free field}) determines a map 
\beqn\label{free field2}
\oint_{S^{2d-1}} J^\q : U\left(\Tilde{\fg}^{\bullet}_{d,\ch^\fg_{d+1}(V)}\right) \to U(\sH_{d,V}) 
\eeqn
of $E_1$-algebras.
\end{cor}

This theorem says that there is a homotopy-coherent map between two explicit algebras,
but does not spell it out formulaically.
For $d=1$ both the algebras are concentrated in degree zero and so there is no room for homological subtleties:
the map is the well-known map from the affine Kac-Moody Lie algebra to the Weyl algebra of formal loops.
This free field realization is a linear, in the sense that it arises from a representation of the Lie algebra $\fg$. 
In the next subsection we give explicit formulas in the case~$d=2$.
It would be interesting to see a richer free field realization, such as the Wakimoto realization~\cite{Wakimoto, FF1, FF2}, in higher dimensions. 

The types of free field realizations we observe are used extensively in \cite{KacInf} to prove number theoretic identities, such as the Jacobi triple product identity, by studying characters. 
In the next section we discuss character theory of the the higher Kac-Moody algebra, but leave calculations to later work.

\begin{proof}
Let $r : \CC^d \setminus 0 \to \RR_{>0}$ be the radial projection. 
Consider the induced map
\[
r_* J^\q :  r_* \UU_{\ch_{d+1}^\fg(V)} (\sG_d) \to  r_* \left. \Obs^\q_V \right|_{\hbar = (2\pi i)^d}  .
\]
Sitting inside the domain and codomain are locally constant subfactorization algebras that encode the enveloping algebras of $\Hat{\fg}_{d,\ch^\fg_{d+1}(V)}$ and $\sH_{d,V}$, respectively.
We need to understand how the map $r_* J^\q$ intertwines these subalgebras.

The key is to use the action of $U(d)$ by rotating $\CC^d$.
The classical $\beta\gamma$ system on $\CC^d$ (and $\CC^d \setminus \{0\})$ is manifestly equivariant under this rotation action, as is the classical current algebra.
The map $J^\cl$ preserves this action, and hence $_* J^\cl$ is $U(d)$-equivariant.
This equivariance persists upon quantization, since the BV Laplacian is also compatible with this action. 
Thus, $U(d)$-eigenspaces are preserved.

On the Kac-Moody side, recall that there is subfactorization algebra of $\UU_{\ch_{d+1}^\fg(V)} (\sG_d)$ given by the sum of the $U(d)$-eigenspaces. 
It is precisely $\UU_{\ch_{d+1}^\fg(V)} \mathtt{G}_d$, where $\mathtt{G}_d = \Omega^*_{\RR_{>0},c} \otimes \fg^\bullet_d$, as in Section~\ref{sec: spheres}.

On the observables side, the subfactorization algebra of $r_* \left. \Obs^\q_V \right|_{\hbar = (2\pi i)^d}$ consisting of $U(d)$-eigenspaces is also locally constant on $\RR$,
and it is 
\[
\UU_\omega (\Omega^*_{\RR_{>0},c} \otimes (V \oplus V^*[d-1)).
\]
Analogously to the Kac-Moody case, it is straightforward to see that this factorization algebra is equivalent to $U(\sH_{d,V})$ as $E_1$ algebras.
(A detailed proof is available in Chapter 3 of~\cite{BWthesis}).

Finally, note that the family of functionals $\{I^{\sG}[L]\}$ defining the Noether map are all $U(d)$-invariant.
Thus, $J^\q$ preserves the subfactorization algebras of $U(d)$-eigenspaces, and the result follows.
\end{proof}

\subsubsection{Explicit formulas for two-dimensional free field realization}

We want to show that these results are not just abstract statements but lead to explicit, useful formulas.
One can analyze this higher free field realization with bare hands.
To keep things concrete, we will work out all the details only in the case $d=2$, but our methods work without difficulty (beyond careful bookkeeping) in arbitrary dimension.

The essential idea is familiar from quantum mechanics and field theory:
the commutation relations (on the algebra side) can be identified with contractions with the propagator (on the diagrammatic side),
i.e., the two faces of Wick's formula.
In the setting of factorization algebras,
this relationship arises from our embedding of the algebra $U(\cH_{d,V})$ into the factorization algebra $r_* \Obs^q_V$.
(A pedagogical discussion, with extensive examples, can be found in Chapter 4 of~\cite{CG1}.)
We now spell out this relationship in some detail.

First, it is useful for us present our algebras in terms of residues over the sphere $S^{2d-1}$. 
This presentation is completely analogous to the mode expansions in a vertex algebra via ordinary residues.

Recall that the higher residue determines a pairing $(-,-)_{\oint}: A_d \tensor A_d \to   \CC$ with
\[
(a, b)_{\oint} = \oint_{S^{2d-1}} a \wedge b \wedge \d^d z ,
\]
and hence determines a map $A_d \to A_d^\vee [-d + 1]$ that sends an element $a$ to the linear functional $b \mapsto (a, b)_{\oint}$.
This construction makes sense after tensoring with a vector space $V$,
so an element $\alpha \in A_d \otimes V^*$ determines a functional 
\[
\begin{array}{cccc}
\oint_{S^{2d-1}} a  : & A_d \tensor V & \to & \CC \\
& a & \mapsto & \oint_{S^{2d-1}} \<a \wedge c\>_V \wedge \d^d z .
\end{array}
\]
Setting $d=1$, the reader will recognize standard manipulations with the usual residue formula.

We can now describe explicitly a simple linear operator in $U(\sH_{d,V})$ as an element of $r_* \Obs^q_V$ via residues.   
Let $I \subset \RR_{>0}$ be an interval. 
An element in the linear component 
\[
\rho \tensor a \in \Omega^*_c(I) \tensor \left(A_d \tensor V^*[d-1]\right)
\] 
determines the linear observable 
\[
\cO_{\rho \otimes a}(\gamma) = \int_{r^{-1}(I)} \rho(r) \<a(z,\zbar), \gamma(z,\zbar)\> \,\d^d z 
\] 
where $\gamma \in \Omega^{0,*}(r^{-1}(I)) \otimes V$ is a field on the shell~$r^{-1}(I)$.
In this way, one can generate, in fact, a generating collection of linear operators for the whole enveloping algebra.

We have seen how to embed simple operators,
since $A_d$ is a subcomplex of $\Omega^{0,*}(\CC^d)$,
so now we build toward describing how the commutator in $U(\sH_{d,V})$ in terms of contractions with the propagator $P = P_{0 < \infty}$ for~$\Omega^{0,*}(\CC^d)$. 

Recall that the propagator is a distributional $(0,1)$-form on $\CC^d \times \CC^d$ with values in $V^* \tensor V$, 
and it satisfies $\d^2 z\, \dbar P = \delta_{{\rm diag}} (z) \id_V$, where $\delta_{{\rm diag}}$ is the $\delta$-distribution along the diagonal $\CC^d$ in $\CC^{2d}$. 
We view $\id_V$ as an element of $V^* \tensor V$. 
We note that $P$ is smooth away from the diagonal; 
its singularities lie only along the diagonal.

The analytic and the algebraic factors decouple, 
so to simplify notations, let $P(z) = p(z) \tensor \id_V$, 
where $p(z)$ is the differential form part of the propagator above. 
Since $\d^d z \, \dbar p(z)$ is the $\delta$-function on $\CC^d \setminus 0$,
one has the residue formula
\[
\oint_{S^{2d-1}} p(z) \,\d^d z = 1 .
\]
Thus, $p(z)$ is the $\delta$-function for the $(2d-1)$-sphere.
In the case~$d=2$,
the propagator is given by
\[
P (z,\zbar) = \frac{1}{(2\pi i)^d} \frac{\zbar_1 \d \zbar_2 - \zbar_2 \d \zbar_1}{|z|^4} \tensor \id_{V}.
\]
We can finally examine the commutator-contraction relation.

Consider the quadratic observables in $U (\sH_{d,V})$ given~by
\[
F_a(b,c) = \oint_{S^3} \d^2 z\, \<b, (a \cdot c)\>_V 
\]
and
\[
F_{a'}(b,c) = \quad \oint_{S^3} \d^2 z\, \<b, (a'\cdot c)\>_V,
\]
where $a,a' \in A_2 \tensor \fg$. 
The commutator bracket in $U (\sH_{d,V})$~is
\beqn\label{bracket}
\left[\oint_{z \in S^{3}} \<b , a \cdot c \>_V , \oint_{w \in S^{3}} \<b , a' \cdot c\>_V \right],
%\oint_{z \in S^{3}} \oint_{w \in S^{3}} \<a \tensor b ,{\rm Cas}_{V} \cdot \omega_{BM}(z,w)\> .
\eeqn
and it should be computable as the sum of three types of terms: 
\begin{itemize}
\item the terms that arise by contracting one $b$ from one observable with one $c$ field from the other observable, using the propagator --- this yields a new quadratic observable;
\item the term that contracts away all the fields --- this yields a constant term; and
\item a contact term.
\end{itemize}
The first kind of term has the form
\[
\oint_{z \in S^3} \d^2 z\, \oint_{w \in S^3} \d^2 w\, p(z-w) \<b(z,\zbar), [a(z,\zbar), a'(w,\wbar)] \, c(w,\wbar)\>_V  .
\]
Let $\eta(z,\zbar;w,\wbar) = \<b(z,\zbar), [a(z,\zbar), a'(w,\wbar)] \cdot c(w,\wbar)\>_V$.
Since $\d^2 w\, \dbar p(z-w) = \delta(z-w)$, we can write 
\beqn\label{contact}
\oint_{w \in S^3} \d^2 w\, p(z-w) \eta(z,\zbar;w,\wbar) = \eta(z,\zbar;z,\zbar)- \int_{N_{\epsilon}} \d^2 w\, p(z-w) \rho(w,\wbar) \dbar_w \eta(z,\zbar;w,\wbar) .
\eeqn
Here, $N_\epsilon$ is a neighborhood of the $3$-sphere inside of $\CC^2$ and $\rho(w,\wbar)$ is a compactly supported function that is identically $1$ near the inner boundary of the neighborhood and identically zero near the outer boundary.

The second term in the expression (\ref{contact}) spoils the compatibility of the bracket in $A_2 \tensor \fg$ and the commutator in the Heisenberg algebra.
It arises precisely because our differential form s may not be holomorphic.
By contrast, in the $d=1$ case, the model $A_1 = \CC[z,z^{-1}]$ consists of holomorphic Laurent polynomials, and so the Lie brackets agree on the nose. 
The failure of this bracket is corrected by an $A_\infty$ morphism, as we will see in the section below.

There is potentially another term in the expansion of (\ref{bracket}) involving a {\em double} contraction and hence has integrand proportional to $p(z-w)^2$. 
In fact, this term vanishes identically since $p(z-w)$ is a $(0,1)$-form. 
Thus, the double contraction does not appear when $d=2$ (but it could in other dimensions).
Again, this result differs from the $d=1$ case, where the double contraction produces the term equal proportional to the level of the affine Kac-Moody algebra. 

This example of the commutator of quadratic observables encodes the key information in this situation.
We now make a clean general statement.

The map of interest
\[
\oint_{S^3} J^\q : U\left(\Tilde{\fg}_{2, \ch_3(V)}\right) \to U(\sH_{2})
\] 
goes from an $A_\infty$ algebra to the enveloping algebra of a strict dg Lie algebra, which is a strict dg associative algebra. 
Indeed, by definition, $U\left(\Tilde{\fg}_{2, \ch_3(V)}\right)$ is the enveloping algebra of an $L_\infty$ algebra, so it has nonzero higher multiplications
\[
m_k : U\left(\Tilde{\fg}_{2, \ch_3(V)}\right)^{\tensor k} \to U\left(\Tilde{\fg}_{2, \ch_3(V)}\right)
\] 
for $k = 1, 2$, and $3$ because the $L_\infty$ algebra $\Tilde{\fg}_{2, \ch_3(V)}$ only has nonzero $\ell_1 = \dbar, \ell_2 = [-.-]$ and $\ell_3$. 
The multiplications satisfy $m_k = 0$ for $k > 2$. 

\begin{prop}\label{prop: ainfinty}
In dimension $d = 2$, an $A_\infty$ model for the map (\ref{free field2}) of $E_1$-algebras is given by the sequence of maps $(\sJ_k)_{k = 1,2,\ldots}$, where
\[
\sJ_k : U\left(\Tilde{\fg}_{2, \ch_3(V)}\right)^{\tensor k} \to U(\sH_V)[2-k],
\]
with $\sJ_k = 0$ for~$k > 2$,
\[
\sJ_1 (a) = \oint_{S^3} \<b, a \cdot c\>_V
\]
and 
\[
\sJ_2 (a, a') = \oint_{z \in S^3} \int_{w \in N_{\epsilon}}  \<b(z), [a(z),a'(z)] c(w) \>_V \, \rho(w) p(z-w) .
\] 
where $N_\epsilon$ is some neighborhood of $S^3$ and $\rho$ is a compactly supported function as in the  preceding computation. 
\end{prop}

\begin{proof}
In the computation of the commutator of quadratic observables,
we saw that the $1$-ary map $\sJ_1$ fails to be compatible with the commutator.
This failure is precisely corrected by the $2$-ary map $\sJ_2$: 
we will show
\beqn\label{infinity relation}
\dbar \sJ_2(a, a') \pm \sJ_2(\dbar a, a') \pm \sJ_2(a, \dbar a') = \sJ_1([a,a']) - [\sJ_1(a), \sJ_1(a')] .
\eeqn
Suppose for simplicity that $a,a'$ are holomorphic.
In that case, the desired identity simplifies to 
\[
\dbar \sJ_2(a, a') = \sJ_1([a,a']) - [\sJ_1(a), \sJ_1(a')].
\]
Above, we saw that the right-hand side~is
\[
\oint_{z \in S^3} \int_{w \in N_{\epsilon}} \d^2 w\,  \<b(z), [a(z),a'(z)] \dbar_w c(w) \>_V \, \rho(w) p(z-w),
\]
which is precisely the functional~$\dbar \sJ_2(a,a')$. 

If $a,a'$ are not holomorphic, there are extra contact terms in the expansion of $[\sJ_1(a), \sJ_1(a')]$ which exactly cancel the terms  $\sJ_2(\dbar a, a') \pm \sJ_2(a, \dbar a')$ on the left-hand side of (\ref{infinity relation}).

It remains to show how $\sJ_2$ is compatible with the $L_\infty$ extension,
i.e., that the $\ell_3$ term is the homotopy in the Jacobi relation up to homotopy.
For $a,a',a'' \in A_d \tensor \fg$: we want
\beqn\label{relationsb}
\oint_{S^3} \Tr_V \left(a\wedge \partial a' \wedge \partial a''\right) = \displaystyle \left(\sJ_2(a, [a',a'']) + \;{\rm permutations} \right) + \left(\left[\sJ_1(a), \sJ_2(a',a'') \right] + \; {\rm permutations} \right)
\eeqn
There are two types of terms present here. 

First, there are the single contractions of the propagator, which contributes terms of the form
\[
\oint_{z \in S^3} \d^2 z\, \int_{w \in N_{\epsilon}} \d^2 w\,  \<b(z), [a(z), [a(z),a'(z)] ] c(w) \>_V \, \rho(w) p(z-w).
\]
These single contraction terms cancel the first term in parentheses in Equation~(\ref{relationsb}). 

The second type of term involves a double contraction of the propagator. 
It contributes terms of the form
\[
\oint_{z \in S^3} \int_{w \in N_\epsilon} \oint_{u \in S^3} \Tr_V(a(z) \wedge a'(w) \wedge a''(w)) \rho(w) p(z-w) p(z-u) p(w-u) .
\]
We note the similarities of this term and the weight of the $3$-vertex wheel appearing the calculation of the anomaly to satisfying the equivariant quantum master equation. 
In fact, similar manipulations as in Appendix \ref{sec: feynman} show that it equals the left-hand side of (\ref{relationsb}), and hence the $A_\infty$-relation is satisfied.
\end{proof}

 \section{Some global aspects of the higher Kac-Moody factorization algebras}

A compelling aspect of factorization algebras is that they are local-to-global objects,
and hence the global sections---the factorization homology---can contain quite interesting information.
For instance, in the case of a one-dimensional locally constant factorization algebra, the global sections along a circle encodes the Hochschild homology of the corresponding associative algebra. 
In the complex one-dimensional situation, the factorization homology along Riemann surfaces is closely related to the conformal blocks of the associated vertex algebra. 

In the first part of this section, we direct our attention to a class of complex manifolds called {\em Hopf manifolds},
whose underlying smooth manifold has the form $S^1 \times S^{2d-1}$.
They provide a natural generalization of elliptic curves, and hence to generalizations of interesting phenomena from chiral CFT in complex dimension one.
In particular, the factorization homology on Hopf manifolds serves as a natural home for characters of representations of the sphere algebra~$\Tilde{\fg}_{d,\theta}^\bullet$.
We demonstrate that by identifying the factorization homology with the Hochschild homology of~$\Tilde{\fg}_{d,\theta}^\bullet$.

In the next section, we provide examples of such characters using field theory.
In physical terms, the factorization homology is related to the partition function of $\sG$-equivariant holomorphic field theories, such as the higher dimensional $\beta\gamma$ and $bc$ systems.
We compute the partition functions on these Hopf manifolds, which are close cousins to superconformal indices, 
and give a concise description in terms of Hochschild homology, in Proposition~\ref{prop: hopf}. 

Finally, we return to the LMNS variants of the twisted higher Kac-Moody factorization algebra that exist on other closed $d$-folds and assert a relationship to the ordinary Kac-Moody algebra on Riemann surfaces. 

\subsection{Hopf manifolds and Hochschild homology}

Given any (possibly dg) algebra $A$, a trace is a linear map
\[
\tr : A \to \CC
\]
which vanishes on pure commutators $\tr ([a,b]) = 0$, $a,b \in A$. 
Thus, every trace determines (and is determined by) a linear map $\tr : A / [A,A] \to \CC$. 

There is a mathematical object associated to an algebra, called the {\em Hocschild homology} $HH_*(A)$ of $A$, which describes all traces simultaneously. 
For an ordinary algebra, in degree zero, the Hochschild homology is precisely $HH_0(A) = A / [A,A]$. 
Thus, the space of all traces for $A$ can be identified with its dual $HH_0(A)^\vee$. 
A natural source of traces, of course, come from modules. 
Given a finite dimensional module, $V$, the ordinary trace is defined and hence determines an element $\tr_V \in HH_0(A)^\vee$. 

A (possibly $A_\infty$) associative algebra is equivalent to the data of a locally constant factorization algebra on $\RR$. 
The global sections, or factoriation homology, of the factorization algebra along the circle $S^1$ is precisely the Hochschild homology. 

In this section, we view the higher Kac-Moody factorization algebra as a factorization algebra on $\CC^d \setminus 0 \cong S^{2d-1} \times \RR$. 
By $U(d)$-equivariance, this determines a factorization algebra on any Hopf manifold (which we define momentarily), which is topologically a product of odd spheres $S^{2d-1} \times S^1$. 
Pushing forward along the radial projection map allows us to view the global sections of the factorization homology as the Hochshild homology of the $A_\infty$ algebra $U(\Tilde{\fg}^{\bullet}_{d,\theta})$.
Combining techniques of factorization algebras and ordinary algebra, we thus arrive at complete description of traces for the higher modes algebra.

\subsubsection{Overview of Hopf manifolds}

Recall that for every complex number $q$ such that $0< |q| < 1$, 
there is a natural action of $\ZZ$ on the punctured plane $\CC^\times$ where $n \cdot z = q^n z$.
We will denote this multiplicative action with the succinct notation $q^\ZZ$.
The quotient space $\CC^\times/q^\ZZ$ is then an elliptic curve,
and the punctured unit disk $\{0< |q|<1\}$ parametrizes elliptic curves in a convenient way.

This story admits an obvious higher dimensional generalization, first explored by \\Hopf~\cite{Hopf}.

\begin{dfn}
Let $d$ be a positive integer.
Let ${\bf q} = (q_1,\ldots, q_d)$ be a $d$-tuple of complex numbers where $0 < |q_i| < 1$ for $i = 1, \ldots, d$. 
The $d$-dimensional {\em Hopf manifold} $X_{\bf q}$ is the quotient of punctured affine space $\CC^d \setminus \{0\}$ by the multiplicative action of~$\ZZ$:
\[
X_{\bf q} = \left. \left(\CC^d \setminus \{0\}\right) \right/ {\bf q}^\ZZ.  
\]
In other words, we quotient by the equivalence relation
\[
(z_1,\ldots,z_d) \sim (q_1^{n} z_1, \ldots,q_d^{n} z_d)
\]
where $n$ runs over $\ZZ$.
\end{dfn}

We denote the obvious quotient map by $p_{\bf q} : \CC^d \setminus \{0\} \to X_{\bf q}$. 
It is a straightforward exercise to check that $X_{\bf q}$ is diffeomorphic to $S^{2d-1} \times S^1$ as smooth manifolds;
the clearest situation is where $q_1 = q_2 = \cdots = q_d$,
which is the most direct generalization of elliptic curves.

\begin{rmk}
By definition, a Hopf manifold of dimension $d$ is a complex manifold diffeomorphic to $S^{2d-1} \times S^1$. 
Our arguments below extend with only a little extra effort to an arbitrary Hopf manifold, 
but this class is interesting and easy to work with.
Note that for $d>1$, $H^{2}_{dR} (X_{\bf q}) = 0$ and so Hopf manifolds are {\em not} K\"{a}hler in complex dimensions bigger than one. 
\end{rmk}

A key fact for us is that the Dolbeault complex of a Hopf manifold admits a small model.

\begin{lem}\label{lem:tanre}
For any Hopf manifold $X_{\bf q}$, there is a quasi-isomorphic inclusion of bigraded complexes
\[
(\CC[a, b, \alpha, \beta], \delta) \hookrightarrow (\Omega^{*,*}(X_{\bf q}),\dbar)
\]
where the generator $a$ has bidegree $(1,1)$, $b$ has bidegree $(d,d-1)$, $\alpha$ has bidegree $(0,1)$, and $\beta$ has bidegree $(1,0)$, and where $\delta(a) = 0$, $\delta(b) = a^{d}$, $\delta(\alpha) = 0$, and $\delta(\beta) = a$.
Here $\dbar$ and $\delta$ both have bidegree~$(0,1)$.
\end{lem}

We borrow this claim from chapter 4 of~\cite{Tanre}, particularly example 4.63,
and simply sketch the main points.
Notably, a Hopf manifold is the total space of a fibration
\[
S^1 \times S^1 \to X_{\bf q} \to \CC\PP^{d-1}.
\]
Topologically, this fibration is the product of a circle with the Hopf fibration $S^1 \to S^{2d+1} \to \CC\PP^{d-1}$. 
Each torus fiber can be equipped with the complex structure of some elliptic curve, so that the fibration is holomorphic with respect to the natural complex structures on $X_{\bf q}$ and~$\CC\PP^{d-1}$.
Both projective space and an elliptic curve are K\"ahler and hence admit small models, 
following~\cite{DGMS}.
As the Hopf manifold is a fibration, one obtains a model for its Dolbeault complex by twisting the differential on the tensor product of those models.
The lemma above specifies the relevant twisting.

It is also possible to obtain a small model for the {\em de Rham} complex by a further twisting. 
(See Theorem 4.70 of~\cite{Tanre}.)
For the Hopf manifold, however, life is particularly simple.

\begin{lem}[Example 4.72, \cite{Tanre}]
For $X_{\bf q}$ a Hopf manifold,
the Dolbeault cohomology coincides with the de Rham cohomology $H^*(S^1) \otimes H^*(S^{2d-1})$.
Moreover, the complex $(\CC[a, b, \alpha, \beta], \delta)$ is also a de Rham model and hence is quasi-isomorphic to the de Rham complex.
\end{lem}

Our primary interest, however, is in the $(0,*)$-forms.
The complex $\Omega^{0, *}(X_{\bf q})$ sits inside $\Omega^{*, *}(X_{\bf q})$ as a summand,
and similarly the complex $\CC[\alpha]$ sits inside $(\CC[a, b, \alpha, \beta], \delta)$ as a summand,
yielding the following result.

\begin{lem}
There is a quasi-isomorphic inclusion 
\[
(\CC[\alpha], 0) \hookrightarrow (\Omega^{0,*}(X_{\bf q}), \dbar)
\]
induced by the inclusion in lemma~\ref{lem:tanre}.
\end{lem}

Note that the source is precisely the small model for the Dolbeault complex $\Omega^{0, *}(E)$ of an elliptic curve,
and---more importantly---for the de Rham complex of a circle.

\begin{eg}
Consider 
the $(0,1)$-form 
\[
\dbar (\log |z|^2) = \sum_i \frac{z_i\d \zbar_i}{|z|^2} 
\]
on $\CC^d \setminus \{0\}$.
It is $\ZZ$-invariant, and hence descends along the map $p_{\bf q} : \CC^d \setminus 0 \to X$
when ${\bf q} = (q,\ldots,q)$ with $|q| < 1$.
The descended form thus provides an explicit Dolbeault representative for~$\alpha$.
\end{eg}

Finally, the complex  $\Omega^{d, *}(X_{\bf q})$ will appear later. 
We note that the corresponding subcomplex inside $(\CC[a, b, \alpha, \beta], \delta)$ is the summand
\[
\CC b \oplus \CC a^{d-1} \beta \to \CC b \alpha \oplus \CC a^{d} \oplus \CC a^{d-1} \alpha \beta \to \CC a^{d} \alpha,
\]
concentrated between bidegrees $(d,d-1)$ and $(d,d+1)$.
As $\delta(a^{d-1} \alpha \beta) = a^d \alpha$, and $\delta(a^{d-1} \beta) = a^d$, 
we see that the cohomology is spanned by $b$ in bidegree $(d,d-1)$ and $b \alpha$ in bidegree~$(d,d)$.

\subsubsection{Current algebras on Hopf manifolds}

For any choice of ${\bf q} = (q_1,\ldots,q_d)$, we have the local Lie algebra $\sG_{X_{\bf q}} = \Omega^{0,*}(X_{\bf q}, \fg)$, and the corresponding Kac-Moody factorization algebra obtained by the enveloping factorization algebra $\UU\sG_{X_{\bf q}}$.
A choice of invariant polynomial $\theta \in \Sym^{d+1}(\fg^*)^\fg$ defines a $\CC[K]$-linear twisted factorization enveloping algebra $\UU_{\theta}(\sG_{X_{\bf q}})$.  
To reduce clutter, we will drop the subscript $X_{\bf q}$ from $\sG_{X_{\bf q}})$, as our construction is uniform in~${\bf q}$.

Our first result is a computation of the global sections of this factorization algebra.  

\begin{prop}
\label{prop: hopf}
Let $X_{\bf q}$ be a Hopf manifold and let $\theta \in \Sym^{d+1}(\fg^*)^\fg$ be an $\fg$-invariant polynomial of degree $(d+1)$. 
There is a quasi-isomorphism of $\CC[K]$-modules
\[
\UU_\theta \sG (X_{\bf q}) \xto{\simeq} \Hoch_*(U \fg)[K] .
\]
\end{prop}

\begin{rmk}
A character on $\fg$ is a conjugation-invariant function of $\fg$, and so an element of $\Sym(\fg^*)^\fg$.
It thus determines a linear functional on the Hochschild homology group 
\[
{\rm HH}_0(U\fg) = U\fg/[U\fg,U\fg] \cong \Sym(\fg)_\fg.
\]
This map then says that any local current (i.e., holomorphic symmetry) determines a $K$-dependent family of elements of ${\rm HH}_0(U\fg)$, i.e., a kind of cocharacter.
A careful examination of the proof will show that in the $\theta$-twisted case, 
the map can be computed in the style of Wick contractions, except that $\theta$ determines contractions of $d+1$-tuples of inputs.
(When $d=1$, one recovers the Wick-like flavor of the usual Kac-Moody algebras.)
Another interesting feature of this map is its ${\bf q}$-dependence,
so that varying over the moduli of Hopf manifolds, we obtain a ${\bf q}$-cocharacter formula.
In the next subsection, we pair these maps with the characters determined by holomorphic free field theories.
\end{rmk}

\begin{proof}
Fix ${\bf q}$ and write simply $X$ for $X_{\bf q}$ and $\sG$ for $\sG_X$. 
We first consider the untwisted case, with $\theta = 0$, where the statement reduces to $\UU\sG(X) \simeq \Hoch_*(U \fg)$.
The factorization homology on the left hand side is computed by
\[
\UU \sG (X) = \clieu_*(\Omega^{0,*}(X) \tensor \fg) .
\]
Our goal is to compute the Lie algebra homology of the dg Lie algebra $\Omega^{0,*}(X) \tensor \fg$.  

We know that $H^{0,*}(X) \cong \CC[\alpha]$ where $\alpha$ has degree 1, by the arguments above.
Hodge theory will let us produce a quasi-isomorphism $(\Omega^{0,*}(X), \dbar) \to \CC[\alpha]$.
Indeed, any Hermitian metric on $X$ determines an adjoint $\dbar^*$ to $\dbar$ and hence a Dolbeault Laplacian $\Delta_{\dbar} = [\dbar,\dbar^*]$.
Denote by $\sH^{0,*}_{\dbar}(X)$ the graded vector space of harmonic $(0,*$-forms, i.e., those annihilated by $\Delta_{\dbar}$. 
In light of lemma~\ref{lem:tanre}, the orthogonal projection determines a quasi-isomorphism
\beqn\label{hopfquasi}
\pi_{\sH}^{0,*} : \left(\Omega^{0,*}(X), \dbar \right) \xto{\simeq} \sH^{0,*}_{\dbar}(X) \cong \CC[\alpha].
\eeqn
Tensoring with $\fg$, we obtain a quasi-isomorphism 
\beqn\label{Lieproj}
\pi_{\sH}^{0,*} \otimes \id_\fg :\sG (X) \to \fg[\alpha]
\eeqn
of dg Lie algebras, where the target has trivial differential.

\begin{rmk}\label{rmk: hpl1}
In fact, by ellipticity, the orthogonal projection extends to a deformation retraction of dg Lie algebras
\begin{equation*}\label{contraction}
  \begin{tikzpicture}[baseline=(L.base),anchor=base,->,auto,swap]
     \path node (L) {$(\fg[\alpha],0)$} ++(3,0) node (M) {$(\sG (X), \dbar)$} 
     (M.mid) +(0,.075) coordinate (raise) +(0,-.075) coordinate (lower);
     \draw (L.east |- lower) to node {$\scriptstyle \iota$} (M.west |- lower);
     \draw (M.west |- raise) to node {$\scriptstyle \pi_{\sH}^{0,*} $} (L.east |- raise);
     \draw (M.south east) ..controls +(1,-.5) and +(1,.5) .. node {$\scriptstyle \eta$} (M.north east);
  \end{tikzpicture}\end{equation*}
The map $\iota$ is the inclusion of harmonic forms. 
The operator $\eta$ is constructed from the Green's function $p(z,w) \in \Omega^{d,*}(X \times X)$ of the $\dbar$ operator on $X$, 
which satisfies 
\[
\dbar_w p(z,w) = \omega_{diag}
\]
where $\omega_{diag}$ is the volume form along the diagonal in $X \times X$. 
The homotopy $\eta$ is defined by
\[
(\eta \mu)(z) = \int_{w \in X} p(z,w) \mu(w)
\]
where $\mu$ is a $(0,*)$-form on~$X$.
This data satisfies the homotopy retraction condition
\[
\iota \circ \pi - \id_{\sG} = \dbar \eta +\eta \dbar = [\dbar,\eta],
\]
and hence ensures that we know precisely how $\sG(X)$ retracts onto its cohomology~$\fg[\alpha]$.
\end{rmk}

Applying Chevalley-Eilenberg chains to Equation (\ref{Lieproj}), we obtain the following quasi-isomorphism for the global sections of the untwisted Kac-Moody factorization algebra:
\beqn\label{hopfquasi3}
\begin{tikzcd}
\clieu_*(\pi_{\sH}^{0,*}) : \displaystyle \UU\sG (X) = \clieu_*(\Omega^{0,*}(X , \fg)) \ar[r,"\simeq"] &\clieu_*(\CC[\epsilon] \tensor \fg) .
\end{tikzcd}
\eeqn
Unpacking the right hand side, we have
\[
\clieu_*(\CC[\alpha] \tensor \fg) = \clieu_*(\fg \oplus \fg[-1]) = \clieu_*(\fg, \Sym (\fg^{ad})),
\] 
where $\Sym(\fg^{ad})$ is the symmetric product of the adjoint representation of $\fg$. 
By the Poincar\'{e}-Birkoff-Witt theorem, there is an isomorphism of vector spaces $\Sym(\fg) \cong U \fg$, so we can interpret this cochain complex as~$\clieu_*(\fg, U \fg^{ad})$.

Any $U(\fg)$-bimodule $M$ is automatically a module for the Lie algebra $\fg$ by the formula $x \cdot m = xm - mx$ where $x \in \fg$ and $m \in M$.
Moreover, for any such bimodule there is a quasi-isomorphism of cochain complexes 
\[
\clieu_*(\fg, M) \xto{\simeq} {\rm Hoch}_*(U\fg, M) 
\]
which is induced from the inclusion of $\fg \hookrightarrow U \fg$. 
(See, for instance, Theorem 3.3.2 of~\cite{LodayCyclic}.)
Applied to the bimodule $M = U\fg$ itself we obtain a quasi-isomorphism 
\[
\clieu_*(\fg , U\fg^{ad}) \xto{\simeq} {\rm Hoch}(U\fg).
\]
The right hand side is the Hochschild homology of $U\fg$ with values in $U\fg$ equipped with the standard bimodule structure. 
Composing with the quasi-isomorphism (\ref{hopfquasi3}) we obtain a quasi-isomorphism $\UU\sG(X) \xto{\simeq} \Hoch(U\fg)$, as desired.

We now consider the twisted case. 
Let $\theta$ be a nontrivial degree $(d+1)$ invariant polynomial on $\fg$. 
The factorization homology is then 
\[
\UU_\theta (\sG)(X) = \left(\Sym(\Omega^{0,*}(X) \tensor \fg)[K] , \dbar + \d_{CE} + K \cdot \d_\theta\right) .
\]
We wish to show that this cochain complex admits a quasi-isomorphism to $\Hoch_*(U \fg)[K]$.
The twisted complex is a $K$-linear deformation of the ordinary Lie algebra homology of $\sG(X)$. 
In particular, it does not follow that the orthogonal projection (\ref{Lieproj}) defines a quasi-isomorphism to $\Hoch_*(U \fg)[K]$.
In order to find an explicit quasi-isomorphism, we appeal to the homological perturbation lemma.
For more details on this result, see Section 2.5 of~\cite{GwThesis}. 

In the untwisted case, upon tensoring with $\CC[K]$, Remark \ref{rmk: hpl1} implies that we have a deformation retraction of cochain complexes
\begin{equation*}
  \begin{tikzpicture}[baseline=(L.base),anchor=base,->,auto,swap]
     \path node (L) {$\clieu_*(\fg[\alpha])[K]$} ++(4,0) node (M) {$\clieu_*(\sG (X))[K]$} 
     (M.mid) +(0,.075) coordinate (raise) +(0,-.075) coordinate (lower);
     \draw (L.east |- lower) to node {$\scriptstyle \iota$} (M.west |- lower);
     \draw (M.west |- raise) to node {$\scriptstyle \pi_{\sH}^{0,*} $} (L.east |- raise);
     \draw (M.south east) ..controls +(1,-.5) and +(1,.5) .. node {$\scriptstyle \eta$} (M.north east);
  \end{tikzpicture}\end{equation*}
To obtain the twisted complex, we turn on the deformation $K \, \d_\theta$ on the left-hand side. 
The homological perturbation lemma provides formulas for the resulting deformations of the projection, inclusion, and homotopy maps.
Explicitly, these formulas involve the formal inverse to the operator $\id_{\sG} - K \,\d_\theta \circ \eta$ defined by
\[
(\id_{\sG} - K \d_\theta \eta)^{-1} = \sum_{n \geq 0} \frac{K^n}{n!} (\d_\theta \circ \eta)^n  .
\] 
Note that acting on any element in the symmetric algebra~$\Sym(\Omega^{0,*}(\Omega^{0,*}(X) \tensor \fg[1])$, this formula is well-defined since only finitely many terms in the series act nontrivially.
(To see this, observe that $\d_\theta$ lowers symmetric power and $\eta$ preserves it, so any polynomial will eventually be annihilated.)

With this operator in hand, we define the maps
\begin{eqnarray*}
\widetilde{\pi}^{0,*}_{\sH} & = & \pi + K \cdot \pi \circ (\id_{\sG} - K \d_\theta \eta)^{-1} \circ \d_\theta \circ \eta, \\
\widetilde{\eta} & = & \eta + K \cdot \eta \circ (\id_{\sG} - K \d_\theta \eta)^{-1} \circ \d_\theta \circ \eta .
\end{eqnarray*}
Note that modulo $K$, these reduce to the original maps above. 
The inclusion map $\iota$ and the differential on $\clieu_*(\fg[\alpha])$ do not get deformed in our situation because the perturbed piece of the differential $\d_{\theta}$ vanishes identically on the harmonic forms. 
The homological perturbation lemma implies that the resulting diagram
\begin{equation*}
  \begin{tikzpicture}[baseline=(L.base),anchor=base,->,auto,swap]
     \path node (L) {$\clieu_*(\fg[\alpha])[K]$} ++(3,0) node (M) {$\UU_\theta (\sG)(X) $} 
     (M.mid) +(0,.075) coordinate (raise) +(0,-.075) coordinate (lower);
     \draw (L.east |- lower) to node {$\scriptstyle \iota$} (M.west |- lower);
     \draw (M.west |- raise) to node {$\scriptstyle \Tilde{\pi}_{\sH}^{0,*} $} (L.east |- raise);
     \draw (M.south east) ..controls +(1,-.5) and +(1,.5) .. node {$\scriptstyle \Tilde{\eta}$} (M.north east);
  \end{tikzpicture}\end{equation*}
is a deformation retraction of cochain complexes. 
With the quasi-isomorphism $\Tilde{\pi}^{0,*}_\sH$ in hand, the result of the proposition now follows from the same argument as in the untwisted case. 
\end{proof}

\subsubsection{Twisted Hochschild homology}

We deduce a consequence of this calculation for the Hochschild homology of the $A_\infty$ algebra $U(\Tilde{\fg}^\bullet_{d,\theta})$.
Let $p_{\bf q} :  \CC^d \setminus \{0\} \to X$ be the quotient map and consider the commuting diagram
\[
\xymatrix{
\CC^d \setminus \{0\} \ar[r]^-{p_{\bf q}} \ar[d]^-{r} & X \ar[d]^{\Bar{r}} \\
\RR_{>0} \ar[r]^-{\Bar{p}_{\bf q}} & S^1
}
\]
where $r$ is the radial projection map and $\Bar{r}$ is the induced map on the quotient.
The action of $\ZZ$ on $\CC^d \setminus\{0\}$ gives $\sG_{\CC^d \setminus \{0\}}$ the structure of a $\ZZ$-equivariant factorization algebra. 
In turn, $\ZZ$ acts on the pushforward factorization algebra .
We have seen that there is a locally constant subfactorization algebra on $\RR_{>0}$, equivalent as an $E_1$ (or $A_\infty$) algebra to $U(\Tilde{\fg}^\bullet_{d,\theta})$, that sits as dense subalgebra of the pushforward factorization algebra $r_* \sG_{\CC^d \setminus \{0\}}$.
The action of $\ZZ$ preserves the dense subalgebra.

This relationship induces a map at the level of global sections on the circle~$S^1$,
and it is quite interesting due to a nontrivial dependence on ${\bf q}$.
The subtlety here is that global sections coincide with the $\ZZ$-invariant global sections on $\RR$, 
i.e., the sections that are ``periodic'' with respect to the action of $\ZZ$.
For instance, the global sections of the sub-factorization algebra are {\em not} Hochschild chains of $U(\Tilde{\fg}^\bullet_{d,\theta})$, 
but a version that takes into account the monodromy around the circle.
Systematic discussions of this phenonema can be found in Section 5.5.3 of \cite{LurieHA}, Lemma 3.18 of~\cite{AFTopMan}, or Section 7.4 of~\cite{CG1}.
We denote this ${\bf q}$-twisted Hochschild homology by $\Hoch_*(U(\Tilde{\fg}^\bullet_{d,\theta}), {\bf q})$.
Concretely, it is the Hochschild homology of the $E_1$ algebra $U \Tilde{\fg}_{d,\theta}^\bullet$ with coefficients in the bimodule $U \Tilde{\fg}_{d, \theta}^\bullet$, equiped with the ordinary left module structure and right module structure determined by the automorphism corresponding to the element $1 \in \ZZ$ on the algebra.

As the locally constant factorization algebra on $\RR_{>0}$ sits inside the pushforward, 
we obtain a canonical map of global sections
\[
\Hoch_*(U(\Tilde{\fg}^\bullet_{d,\theta}), {\bf q}) \to \Bar{r}_* \UU_\alpha(\sG_X) (S^1) ,
\]
which is, in fact, a quasi-isomorphism, by our results in the preceding section. 

Now, the global sections of the pushforward factorization algebra agree with the global sections of the factorization algebra on the source space,
so we have a quasi-isomorphism
\[
\Bar{\rho}_* \UU_\alpha(\sG_X) (S^1) \simeq \UU_{\alpha} (\sG)(X) .
\]
It follows that there is a quasi-isomorphism of Hochschild homologies
\beqn\label{hoch1}
\Hoch_*(U(\Tilde{\fg}^\bullet_{d,\theta}), {\bf q}) \simeq \Hoch_* (U\fg)[K] .
\eeqn
Peculiarly, this statement is purely algebraic as the dependence on the manifold for which the Kac-Moody factorization algebra lives has dropped out.
The thing to note is that the quasi-isomorphism {\em does} depend on ${\bf q}$. 

\subsection{Character formulas by coupling to a free theory}

\def\Cl{{\rm Cl}}

We turn to a class of free field theories on Hopf manifolds that have a symmetry by the local Lie algebra $\sG_X$. 
Following Section \ref{sec: qft}, we study this situation by coupling the local Lie algebra $\sG_X$ to the free theory.
Our main result in this section, Proposition \ref{prop: twistedchar}, is an interpretation of this coupling at the {\em quantum level} as a character of~$\fg$. 

There are two main differences between the theory we consider here and the one considered in Section~\ref{sec: qft}.
First, in this section we are working on a closed $d$-fold, namely the Hopf surface $X = X_{\bf q}$. 
Although there is still a factorization algebra of observables on $X$, the main statement in this section concerns the global sections of this factorization algebra.
Since the theory actually makes sense on any complex manifold, 
our result --- which is specific to Hopf manifolds --- is an avatar of a large class of analogous results.

Second, the theory we consider here is a free theory of {\em fermions}. 
Thus, we will work with super vector spaces and super cochain complexes.
These lead to minor modifications to the approach of Section~\ref{sec: qft}, 
but yields a statement that is easiest to understand.

Before delving into the details, 
we note for physicists that we develop here a holomorphic version of the Adler-Bardeen-Jackiw anomaly,
as we are studying fermionic matter fields coupled to a background holomorphic gauge field.
(See \cite{Rabinovich} for the traditional ABJ anomaly as seen within this Costello formalism.)
A more exact comparison is with the Konishi anomaly, 
as these holomorphic theories sometimes arise as twists of supersymmetric theories.
By computing global sections on Hopf manifolds,
we recover analogues of the superconformal indices,
since a Hopf manifold has $S^1 \times S^{2d-1}$ as its underlying manifold.

\begin{rmk}
As a matter of convention, if $V$ is a super vector space, we denote by $\Pi(V)$ the super vector space obtained by reversing the parity. 
\end{rmk}

\subsubsection{The free $bc$ system}

To define the theory, we again start with a $\fg$-module $V$.
The theory is very similar in spirit to what is known as the $bc$ system in conformal field theory (which is usually considered in the context of the gauging the bosonic string). 
Hence, we borrow the terminology. 

\begin{dfn}
The {\em classical $bc$ system} valued in the super vector space $W$ on a complex manifold $X$ has space of fields
\[
\sE_{bc} (X) = \Omega^{0,*}(X , W) \oplus \Omega^{d,*}(X , W)[d-1],
\]
with the linear BRST operator given by $Q = \dbar$.
We will write fields as pairs $(c,b)$. 
There is a $(-1)$-shifted symplectic pairing is given by integration along $X$ combined with the evaluation pairing between $W$ and its dual: 
\[
\<c,b\> = \int_X \<c, b\>_{W}.
\] 
The action functional for this free theory is thus
\[
S_{bc} (c,b) = \int_X \<b , \dbar c\>_{W} .
\]
\end{dfn}

\begin{rmk}
%Note that the complex fields of this theory is a super cochain complex concentrated in completely odd super degree. 
Note that this theory is a modest variant of the definition of the higher $\beta\gamma$ system given in Section~\ref{sec: qft}. 
The only difference is that we allow for values in a super vector space $W$, as opposed to an ordinary (bosonic) one. 
When $d=1$ this theory is the usual $bc$ system (valued in $W$) from chiral conformal field theory.
\end{rmk}

Being a free theory, there is a natural BV quantization defined for any $X$.
Its definition mirrors Definition \ref{dfn: qobs} for the $\beta\gamma$ system. 
We denote the resulting factorization algebra of quantum observables by~$\Obs^\q_{bc}$. 

Before moving on to studying $\sG_X$-equivariance of this factorization algebra, we characterize the global observables of the $bc$ system with values in $W$ evaluated on Hopf manifolds. 
To state the result we introduce the following definition, whose bosonic version is familiar. 

\begin{dfn}
Let $W$ be a super vector space, and view $W \oplus W^*$ as an abelian super Lie algebra. 
Define the central extension of super Lie algebras
\[
\CC \cdot \hbar \to {\rm Heis}_\hbar (W \oplus W^*) \to W \oplus W^*
\]
arising from the $2$-cocycle defined by the natural pairing between $W$ and its dual. 
The $\hbar$-dependent {\em Weyl algebra} associated to $W$ is
\[
\Weyl_\hbar(W \oplus W^*) := U({\rm Heis}_\hbar(W \oplus W^*)) .
\]
\end{dfn}

\begin{rmk}
When $W = V$ is purely bosonic, this definition recovers the usual ($\hbar$-dependent) Weyl algebra of $V \oplus V^*$. 
When $W = \Pi(V)$ is purely fermionic, $\Weyl_\hbar(\Pi(V) \oplus \Pi(V^*))$ is the ($\hbar$-dependent) Clifford algebra of $V\oplus V^*$ associated to the natural quadratic form.  
\end{rmk} 

\begin{lem}\label{lem: hopfcliff}
Let $X$ be a Hopf manifold. 
Consider the observables of the higher $bc$ system on $X$ valued in the super vector space $W$. 
There is a natural quasi-isomorphism
\[
\Obs^\q_{bc}(X) \xto{\simeq} {\rm Hoch}_*(\Weyl_\hbar(W \oplus W^*)) .
\]
\end{lem}

\begin{proof}
The observables of any free BV theory can be modeled as the Lie algebra chains of a certain dg Lie algebra. 
(See chapter 4 of \cite{CG1}.)
For the higher $bc$ system valued in $W$, the dg Lie algebra is a (shifted) central extension of the form
\[
\CC  \,\hbar [-1] \to \Hat{\sL} \to \sL
\]
where
\[
\sL = \Omega^{d,*}(X, W^*)[d-1] \oplus \Omega^{0,*}(X, W)
\]
is an abelian dg Lie algebra and the cocycle defining the extension is
\[
\begin{array}{ccc}
\sL \times \sL & \to & \CC \, \hbar [-1]\\
 (c, b) &\mapsto& \hbar \int_X \<c, b\>_W .
\end{array}
\] 
As a cochain complex, $\Hat{\sL} = \sL \oplus \CC  \,\hbar[-1]$. 

Our proof strategy is thus to compute this Lie algebra homology by finding a small model with an obvious identification with the relevant Hochschild homology.

First, we identify this dg Lie algebra $\Hat{\sL}$ with a smaller dg Lie algebra, via Hodge theory, 
as we did in the proof of Proposition~\ref{prop: hopf}.
Thanks to  (\ref{hopfquasi}), we know how to deal with the $c$ fields,
so we turn to replacing the $b$ fields by a simpler model.

Fix a Hermitian metric and hence obtain an orthogonal projection onto the $(d,*)$-harmonic forms
\[
\pi_{\sH}^{d,*} : \left(\Omega^{d,*}(X), \dbar \right) \xto{\simeq} \sH^{d,*}_{\dbar}(X) \cong \CC\{b, b \alpha\},
\]
where $b$ has bidegree $(d,d-1)$ and $b \alpha$ has bidegree $(d,d)$.
(Note that we are using the notation of lemma~\ref{lem:tanre}.)
Observe that $\CC\{b, b \alpha\} = \CC[\alpha]b$, which is naturally isomorphic to $\CC[\alpha][-(d-1)]$,
the shift of the complex $\CC[\alpha]$ up by degree~$d-1$.

The fields with values in $W^*$ are $\Omega^{d,*}(X, W^*)[d-1] $, 
so the harmonic representative $b$ contributes a shift up by $d-1$ that cancels the shift down by $d-1$ in the definition of the fields. 
Hence we obtain a quasi-isomorphism of the $b$ fields onto their cohomology:
\[
\pi_{\sH}^{d,*} \otimes W^* [d-1]: \Omega^{d,*}(X, W^*)[d-1] \xto{\simeq} W^* \otimes \CC[\alpha],
\]
In conjunction with the map (\ref{hopfquasi}), 
we obtain a quasi-isomorphism of dg Lie algebras
\beqn
\label{hopfquasi2}
\Hat{\sL} \xto{\simeq}  \CC[\alpha] \tensor (W^* \oplus W) \oplus \CC \, \hbar [-1]. 
\eeqn
We denote the small complex on the right hand side by~$\widetilde{\sL}$.

The bracket for $\widetilde{\sL}$ is determined by the formula
\[
[1 \tensor w^* , \alpha \tensor w] = \hbar \<w^*, w\>,
\]
and the analogous formula with the roles of $w$ and $w^*$ swapped.

By directly unraveling the definitions, one finds
\[
\clieu_*(\widetilde{\sL}) = \clieu_*({\rm Heis}_\hbar(W \oplus W^*), U({\rm Heis}_\hbar(W \oplus W^*))).
\]
As we know there is a quasi-isomorphism
\[
\clieu_*(\fg, U\fg^{ad}) \xto{\simeq} \Hoch_*(U\fg)
\]
for any dg Lie algebra $\fg$,
we have a sequence of quasi-isomorphisms
\[
\clieu_*(\Hat{\sL}) \xto{\simeq} \clieu_*(\widetilde{\sL}) \xto{\simeq} \Hoch_*(U({\rm Heis}_\hbar(W \oplus W^*))).
\]
A standard result about free BV theories (see section 4.2 of \cite{CG1}) provides a quasi-isomorphism
\[
\Obs^{\q}_{bc} (X) \xto{\simeq} \clieu_*(\Hat{\sL}),
\] 
and so we have the claim, by composing all these quasi-isomorphisms.
\end{proof}

We are most interested in the case that $W = \Pi(V)$, where $V$ is an ordinary (bosonic) vector space. 
In this case, the lemma implies that there is a quasi-isomorphism
\[
\Obs^q_{bc}(X) \xto{\simeq} {\rm Hoch}_*(\Cl_\hbar(V \oplus V^*)) 
\]
where $\Cl_\hbar(V \oplus V^*)$ denotes the $\hbar$-dependent Clifford algebra. 
In this case, there is a lovely simplification of the Hochschild homology. 

If we choose a basis $\{v_i\}$ of $V$, and dual basis $\{v_i^*\}$ of $V^*$, then there is a homomorphism 
\[
\int_{Ber} : {\rm Hoch}_* (\Cl (V \oplus V^*)) \to \CC 
\]
determined by picking off the coefficient of the element $v_1 \cdots v_{\dim(V)} v_1^* \cdots v^*_{\dim(V)}$ in the Hochschild complex. 
In other words, this map is precisely the Berezin integral that projects onto the ``top fermion.'' 

It is a standard fact that the Clifford algebra is Morita trivial \cite{KasselCliff}, so that ${\rm Hoch}_*(\Cl(V \oplus V^*)) \simeq {\rm Hoch}_* (\CC) \cong \CC$.
Hence, $\int_{Ber}$ is a quasi-isomorphism. 

After inverting $\hbar$ and invoking Lemma \ref{lem: hopfcliff}, we obtain a composition of quasi-isomorphisms 
\beqn\label{hopfcliffquasi}
\Obs^\q_{bc} (X) [\hbar^{-1}] \xto{\simeq} {\rm Hoch}_*(\Cl_\hbar (V \oplus V^*)[\hbar^{-1}]) \xto{\simeq}\CC[\hbar,\hbar^{-1}] .
\eeqn
The first quasi-isomorphism is the one from Lemma~\ref{lem: hopfcliff}, and the second is Berezin integration. 

We summarize our computations as follows.

\begin{lem}
On a Hopf manifold $X_{\bf q}$, there is a natural quasi-isomorphism
\[
\Obs^\q_{bc} (X) [\hbar^{-1}] \xto{\simeq} \CC[\hbar,\hbar^{-1}]
\]
out of the quantum observables for a free fermion $bc$ system.
\end{lem}

This map encodes the ``expected value'' of a observables for this system.

\subsubsection{Quantum equivariance}

We now give our fermionic fields charge, by equipping the odd vector space $W = \Pi(V)$ with the structure of a  $\fg$-representation. 
The objective is to extract a character on the observables of the $bc$ system from the $\sG_X$-equivariant quantization. 
By the same formalism as in Section~\ref{sec: qft}, the equivariant quantization determines a quantum Noether map
\[
J_X^\q : \Cur^\q_{\Theta_X}\sG(X) [\hbar^{-1}]  \to \Obs^\q_{bc}(X) [\hbar^{-1}] ,
\]
where $\Theta_X$ is the obstruction to solving the $\sG_X$-equivariant quantization. 
Recall that as a plain factorization algebra, there is an isomorphism
\[
\UU_{\Theta_X}\sG(X) [\hbar^{-1}]  \cong \Cur^\q_{\Theta_X}\sG(X) [\hbar^{-1}] .
\]
The explicit form of the obstruction is irrelevant in what immediately follows, but we discuss it in more detail in Section~\ref{sec: hopfobs} below. 

Combining the quasi-isomorphisms of Lemma \ref{prop: hopf} and Equation (\ref{hopfcliffquasi}), we obtain a commutative diagram
\[
\begin{tikzcd}
\UU_{\Theta_X} (\sG)(X) [\hbar^{-1}]  \ar[r, "J_X^\q"] \ar[d, "\simeq"] & \Obs^\q_{bc}(X) [\hbar^{-1}] \ar[d, "\simeq"] \\
\Hoch_*(U \fg) [\hbar^{-1}] \ar[r,dotted,"{\rm ch}_{X,V}"] & \CC[\hbar,\hbar^{-1}] .
\end{tikzcd}
\] 
The dotted map exists since the quantum Noether map preserves the projection onto the harmonic forms from which both quasi-isomorphisms are constructed.
At the level of $H^0$ we obtain the following.

\begin{prop}\label{prop: twistedchar}
The $\sG_X$-equivariant quantization of the $bc$ system on $X$ valued in the $\fg$-representation $\Pi(V)$ determines a map 
\[
{\rm ch}_{X,V} : {\rm HH}_0(U \fg)[\hbar^{-1}] = \Sym(\fg)_{\fg} [\hbar^{-1}] \to \CC[\hbar,\hbar^{-1}] .
\] 
This map is natural in the representation~$V$.
\end{prop}

As discussed earlier, a character of $\fg$ is a linear functional on ${\rm HH}_0(U \fg)$,
so we have produced an $\hbar$-dependent character from each Hopf manifold~$X$ and finite-dimensional $\fg$-representation~$V$.
Although we will not pursue an explicit formula here, 
this character ${\rm ch}_{X,V}$ varies in a beautiful way over the moduli space of Hopf manifolds,
so that one can obtain $q$-character formulas.

%Thus, $\fc(X)$ must be a 
%\[
%H^*(\Obs^{\q}_{eq}(X_{\bf q})) \xto{e^{I^\q / \hbar}} H^*_{\rm Lie}(\sG(X_{\bf q})) \tensor H^*(\Obs^\q_{bc} (X_{\bf q})) \xto{\cong} H^*(\fg, \Sym(\fg^*)) ((\hbar))
%\]

\subsubsection{A remark on the anomaly}\label{sec: hopfobs}

The $bc$ system on any manifold $X$ is {\em free}, and the anomaly $\Theta_X$ to solving the $\sG_X$-equivariant QME parametrizes the central extension of $\sG_X$ that acts on the quantum theory.

We stress that these constructions work for {\em any} complex $d$-fold $X$, 
so we can consider the higher $bc$ system on $X$ with values in the $g$-representation $\Pi(V)$. 
Furthermore, this theory is natural in the complex manifold $X$, in the sense that any holomorphic embedding of complex $d$-folds $f : X \to X'$ pulls back the $bc$ system on $X'$ to the $bc$ system on $X$. 
There is thus a ``universal'' $bc$ system on a site ${\rm Hol}_d$ of complex $d$-folds and local biholomorphisms (maps that are locally holomorphic isomorphisms).
The construction $\sG$ also determines a sheaf on this site, sending $X$ to $\Omega^{0,*}(X, \fg)$, 
and so we can consider the universal equivariant $bc$ system of $\sG$ acting on~$\sE_{bc}$. 

The anomaly $\Theta_X$ to solving the one-loop $\sG_X$-equivariant QME is encoded by a local functional of the form
\[
I_{\Theta_X}(\fc(X))=\int_X \fc(X) \wedge \Tilde{\fj}(\alpha),
\]
where $\fc(X) \in \Omega^*(X)$ and where $\Tilde{\fj}$ is a linear map of the form
\[
\Tilde{\fj} : \Sym(\sG_X[1]) \to \Omega^*(X) .
\]
Universality then puts restrictions on the differential forms that can appear in the one-loop anomaly. 

Indeed, since the anomaly $\Theta_X$ must be natural with respect to holomorphic embeddings, we see that $\fc(X)$ must be some polynomial in the Chern classes of $X$. 
Indeed, since $\Theta_X$ must also be natural with respect to holomorphic embeddings, we see that $\fc(X)$ must be some polynomial in the Chern classes of~$X$. 

For a Hopf manifold $X$, the Chern classes $c_i(X)$ vanish when $i=1,\ldots,d-1$ for degree reasons. 
Also, if we consider a local biholomorphism of the form $\CC^d \hookrightarrow X$, we see that the anomaly $\Theta_X$ must pull back to the anomaly on $\CC^d$ computed in Section~\ref{sec: qft}. 
Thus, for $X$ a Hopf manifold, the anomaly $\Theta_X$ must be proportional to a functional of the form
\[
\int_X \theta(\alpha \wedge \partial \alpha \wedge \cdots \wedge \partial \alpha)
\]
where $\theta \in \Sym^{d+1}(\fg^*)^\fg$ and $\alpha \in \sG_X$. 
In other words, we know the anomaly up to a scalar factor, which depends in some way on the representation~$V$.

To summarize, we have argued that the local class representing the extension of $\sG_X$ acting on the quantization of the $bc$ system on any Hopf manifold must be of a multiple of $\fj_X(\ch_{d+1}^\fg(V))$. 
To fully identify this class, we need machinery that we hope to develop in future work. 

\subsection{The Kac-Moody vertex algebra and compactification} 

We turn briefly to the variant of the Kac-Moody factorization algebra associated to the cocycles from Section ~\ref{sec: nekext}.
This class of cocycles is related to the ordinary Kac-Moody vertex algebra on Riemann surfaces through compactification, as we now show. 

Consider the complex manifold $X = \Sigma \times \PP^{d-1}$, 
where $\Sigma$ is a Riemann surface and $\PP^{d-1}$ is $(d-1)$-dimensional complex projective space.
Let $\omega \in \Omega^{d-1,d-1}(\PP^{d-1})$ be the natural volume form, 
which clearly satisfies the conditions of Lemma~\ref{lem: cocycle KM} and so determines a degree one cocycle $\phi_{\kappa, \omega} \in \cloc^*(\sG_{\Sigma \times \PP^{d-1}})$ after a choice of $\fg$-invariant bilinear form $\kappa : \fg \times \fg \to \CC$. 
Consider then the twisted enveloping factorization algebra of $\sG_{\Sigma \times \PP^{d-1}}$ by the cocycle~$\phi_{\kappa, \omega}$. 

Recall that if $p : X \to Y$ and $\sF$ is a factorization algebra on $X$, then the pushforward $p_* \sF$ on $Y$ is defined on opens by $p_* \sF : U \subset Y \mapsto \sF(p^{-1} U)$. 

\begin{prop}
Let $\pi : \Sigma \times \PP^{d-1} \to \Sigma$ be the projection. 
There is a quasi-isomorphism between the following two factorization algebras on $\Sigma$:
\begin{enumerate}
\item $\pi_* \UU_{\phi_{\kappa, \theta}} \left(\sG_{\Sigma \times \PP^{d-1}}\right)$, the pushforward along $\pi$ of the Kac-Moody factorization algebra on $\Sigma \times \PP^{d-1}$ of type $\phi_{\kappa,\omega}$, and
\item $\UU_{{\rm vol}(\omega) \kappa} (\sG_\Sigma)$, the Kac-Moody factorization algebra on $\Sigma$ associated to the invariant pairing ${\rm vol}(\omega) \cdot \kappa$. 
\end{enumerate}
\end{prop}

The twisted enveloping factorization on the right-hand side is the familiar Kac-Moody factorization alegbra on Riemann surfaces associated to a multiple of the pairing $\kappa$.
The twisting ${\rm vol}(\omega) \kappa$ corresponds to a cocycle of the type in the previous section 
\[
J({\rm vol}(\omega) \kappa) = {\rm vol}(\omega) \int_\Sigma \kappa(\alpha, \partial \beta)
\]
where ${\rm vol}(\omega) = \int_{\PP^{d-1}} \omega$. 

\begin{proof}
Let $U \subset \Sigma$ be an open subset. 
The factorization algebra $\pi_* \UU_{\phi_{\kappa, \theta}} \left(\sG_{\Sigma \times \PP^{d-1}}\right)$ assigns to $U$, the cochain complex
\beqn\label{KMPn}
\left(\Sym \left(\Omega^{0,*} (U \times \PP^{d-1})\right)[1] [K], \dbar + K \phi_{\kappa, \omega}|_{U \times \PP^{d-1}} \right),
\eeqn
where $\phi_{\kappa, \omega}|_{U \times \PP^{d-1}}$ is the restriction of the cocycle to the open set $U \times \PP^{d-1}$. 
Projective space is Dolbeault formal: its Dolbeault complex is quasi-isomorphic to its cohomology.
Thus, we have\footnote{Here, $\Hat{\tensor}$ is the completed projective tensor product.} 
\[
\Omega^{0,*} (U \times \PP^{d-1}) = \Omega^{0,*}(U) \Hat{\tensor} \Omega^{0,*}(\PP^{d-1}) \simeq \Omega^{0,*}(U) \Hat{\tensor} H^*(\PP^{d-1}, \sO) \cong \Omega^{0,*}(U) .
\]
Under this quasi-isomorphism, the restricted cocycle has the form
\[
\phi_{\kappa,\omega}\Big|_{U \times \PP^{d-1}} (\alpha \tensor 1, \beta \tensor 1) = \int_{U} \kappa(\alpha, \partial \beta) \int_{\PP^{n-1}} \omega 
\]
where $\alpha,\beta \in \Omega^{0,*} (U)$ and $1$ denotes the unit constant function on $\PP^{d-1}$. 
But the right hand side is precisely the value of the local functional ${\rm vol}(\omega) J_\Sigma (\kappa)$ on the open set $U \subset \Sigma$. 
Thus, the cochain complex (\ref{KMPn}) is quasi-isomorphic to 
\beqn
\left(\Sym \left(\Omega^{0,*} (U) \right)[1] [K], \dbar + K {\rm vol}(\omega) J_\Sigma (\kappa) \right) .
\eeqn
We recognize this complex as the value of the Kac-Moody factorization algebra on $\Sigma$ of type ${\rm vol}(\omega) J_\Sigma (\kappa)$.
It is immediate to see that identifications above are natural with respect to maps of opens, so that the factorization structure maps are the desired ones,
completing the proof. 
\end{proof}

Now, pick Riemann surfaces $\Sigma_1,\Sigma_2$ and let $\omega_1,\omega_2$ be their K\"{a}hler forms. 
Consider the two projections
\[
\begin{tikzcd}
& \Sigma_1 \times \Sigma_2 \arrow[dl,"\pi_1"'] \arrow[dr,"\pi_2"] & \\
\Sigma_1 & & \Sigma_2
\end{tikzcd}
\]
Consider the closed $(1,1)$-form $\omega = \pi_1^* \omega_1 + \pi_2^* \omega_2 \in \Omega^{1,1}(\Sigma_1 \times \Sigma_2)$. 
According to the proposition above, for any symmetric invariant pairing $\kappa \in \Sym^2 (\fg^*)^\fg$ this form determines a bilinear local functional
\[
\phi_{\kappa,\omega}(\alpha) = \int_{\Sigma_1 \times \Sigma_2} \omega \wedge \kappa(\alpha, \partial \alpha) 
\]
on the local Lie algebra $\sG_{\Sigma_1\times \Sigma_2}$.
A similar calculation as in the previous example implies that the pushforward factorization algebra $\pi_{i*}\UU_{\phi_{\kappa, \omega}}\sG$, $i=1,2$, is isomorphic to the Kac-Moody factorization algebra on the Riemann surface $\Sigma_i$ with level equal to the Euler characteristic $\chi(\Sigma_j)$, where $j \ne i$. 
This result was alluded to in Section 5 of Johansen \cite{JohansenKM}, 
where it is shown that there exists a copy of the Kac-Moody chiral algebra inside the operators of a twist of the $\cN=1$ supersymmetric multiplet (both the gauge and matter multiplets, in fact) on the K\"{a}hler manifold $\Sigma_1 \times \Sigma_2$. 
In Section~\ref{sec: qft} we saw how the $d = 2$ Kac-Moody factorization algebra embeds inside the operators of a free holomorphic theory on a complex surface. 
This holomorphic theory, the $\beta\gamma$ system, is the minimal twist of the $\cN=1$ chiral multiplet.
Thus, we obtain an enhancement of Johansen's result to a two-dimensional current algebra.

\section{Large $N$ limits} \label{sec: largeN}

\def\cycls{{\rm Cyc}_*}
\def\lqt{{\ell q t}}
\def\colim{{\rm colim}}
\def\sl{\mathfrak{sl}}

We take a detour from the main course of this paper to examine the case that the ordinary Lie algebra underlying the current algebra is $\gl_N$, and study the behavior as $N$ goes to infinity.
This provides a clean explanation for the nature of the most important local cocycles that we have studied throughout this work.

The essential fact is the remarkable theorem of Loday-Quillen \cite{LQ} and Tsygan~\cite{Tsy},
which yields a natural map
\[
\lqt(A) : \underset{N \to \infty}{\colim} \, \cliels(\gl_N(A)) \cong \cliels(\gl_\infty(A)) \to \Sym(\cycls(A)[1])
\]
for any dg algebra $A$ over a field $k$ of characteristic~0.
Naturality here means that it works over the category of dg algebras and maps of dg algebras.
(This construction works even for $A_\infty$ algebras.)
When $A$ is unital, this map is a quasi-isomorphism.

This construction makes sense even when working with the {\em local} Lie algebra cochains,
once we introduce a local version of the cyclic cochains.
In consequence we obtain natural local cocycles for all $\sG \ell_{N} = \fgl_N \tensor \Omega^{0,*}$
from cyclic cocycles of $\Omega^{0,*}$.
This uniform-in-$N$ construction illuminates the simplicity of the chiral anomaly.

Our approach here is modeled on prior work of Costello-Li \cite{CLbcov2} and Movshev-Schwarz~\cite{MovSch},
but it is also satisfyingly parallel to the approach of \cite{FHK},
as we explain below.

\subsection{Local cyclic cohomology}

We need a local notion of a cyclic cocycle. 
Our approach is modeled on the work we undertook earlier in this paper,
where we used the concept of a local Lie algebra earlier as a natural setting for currents. 
In practice, we replace a (dg) Lie algebra with a (dg) associative algebra and replace Lie algebra cochains with cyclic cochains, 
always keeping locality in place.

\begin{dfn}\label{def: localalg}
A {\em $C^\infty$-local dg algebra} on a smooth manifold $X$~is:
\begin{enumerate}
\item[(i)] a $\ZZ$-graded vector bundle $A$ on $X$ of finite total rank, whose sheaf of sections we denote~$\sA^{sh}$;
\item[(ii)] a degree one differential operator $\d : \sA^{sh} \to \sA^{sh}$;
\item[(iii)] a degree zero bidifferential operator $\cdot : \sA^{sh} \times \sA^{sh} \to \sA^{sh}$
\end{enumerate}
such that the collection $(\sA^{sh}, \d, \cdot)$ has the structure of a sheaf of associative dg algebras.
\end{dfn}

\begin{rmk}
It's perhaps abusive to use the term ``local algebra" here, since in the conventional mathematical sense a local algebra refers to an ordinary algebra with a unique maximal ideal. 
We choose the terminology in analogy with the concept of a local Lie algebra on a manifold but stress the difference with the usual commutative algebra definition by adding the adjective~$C^\infty$. 
\end{rmk}

We reserve the notation $\sA$ for the cosheaf of compactly supported sections of the bundle $A \to X$.
By the assumptions, this is a cosheaf of dg associative algebras. 
We will abusively refer to a $C^\infty$-local algebra $(A, \d, \cdot)$ simply by its cosheaf $\sA$.

\begin{eg}
The sheaf of smooth functions provides a trivial example of a $C^\infty$-local algebra on any manifold. 
On a complex manifold, the basic example for us is the Dolbeault complex $\Omega^{0,*}_X$.
This example is, of course, also commutative. 
\end{eg}

Any bundle of finite dimensional associative (dg) algebras defines a $C^\infty$-local algebra where the structure maps are differential operators of order zero.

There is a forgetful functor from $C^\infty$-local algebras to local Lie algebras, by remembering only the commutator determined by $\cdot$. 
Thus, every $C^\infty$-local algebra is a local Lie algebra (with same underlying bundle). 

For $C^\infty$-local algebras, there is an appropriate notion of cohomology respecting the locality, 
analogous to local Lie algebra cohomology. 
To define it, first consider the underlying $\ZZ$-graded vector bundle $A$ of a $C^\infty$-local algebra. 
The $\infty$-jet bundle $JA$ of $A$ is a graded left $D_X$-module via the canonical Grothendieck connection on $\infty$-jets,
as is true for any graded vector bundle,
but it has additional structure as well.
Because the differential and product on $A$ are differential operators, 
they intertwine with the $D_X$-module structure on $JA$.
Hence $JA$ is also a dg associative algebra in the category of dg $D_X$-modules,
using the symmetric monoidal product~$- \otimes_{C^\infty_X} -$. 

In this symmetric monoidal dg category, 
one can mimic many standard constructions from homological algebra.
For our current purposes, we are interested in cyclic cohomology,
and hence as a first step, in $\Hoch^*(R,R^*)$, the Hochschild cohomology of an algebra $R$ with coefficients in its linear dual $R^*$.
The usual formulas apply verbatim in the dg category of dg $D_X$-modules.
Hence, the dg $D$-module of Hochschild cochains on $JA$~is 
\[
\Hoch^* (JA, JA^\vee) = \prod_{n \geq 0} {\rm Hom}_{C^\infty_X} (JA^{\tensor n}, C^\infty_X)[-n]
\]
with the usual Hochschild differential.
(We note that the superscript $\otimes n$ means $\otimes_{C^\infty_X}$ iterated $n$ times.)

The {\em reduced} Hochschild cochains is the product without the $n=0$ component. 

\def\Hoch{{\rm Hoch}}
\def\Hochloc{{\rm Hoch}_{\rm loc}}
\def\Cyc{{\rm Cyc}}
\def\Cycloc{{\rm Cyc}_{\rm loc}}

\begin{dfn}\label{dfn: hochloc}
The {\em $C^\infty$-local Hochschild cochains} of a $C^\infty$-local algebra $\sA$ on $X$~is the sheaf
\[
\Hochloc^*(\sA) = \Omega^*_X[2d] \tensor_{D_X} \Hoch^*_{red} (JA, JA^\vee) .
\] 
We denote the global sections of this sheaf of cochain complexes by~$\Hochloc^*(\sA(X))$.
\end{dfn}

The reader will observe its similarity to its counterpart in local Lie algebra cohomology introduced in Section~\ref{sec: localcocycle}. 
Just as in local Lie algebra cohomology, we can concretely understand an element in $\Hochloc^*(\sA(X))$ as follows.
It is a power series on $\sA(X)$ that is a sum of functionals of the form
\[
\alpha_1 \tensor \cdots \otimes \alpha_k \mapsto \int_X  D_1(\alpha_1) \cdots D_k(\alpha_k) \, \omega_X
\]
where each $D_i$ is a differential operator from $\sA$ to~$C^\infty(X)$ and $\omega_X$ is a smooth top form on~$X$. 

There is a cyclic version of this cohomology.
For each $n$, there is an action of the cyclic group $C_n$ on $JA^{\tensor n}$,
and hence on the $n$th component of the reduced Hochschild complex $\Hoch_{red}^* (JA, JA^\vee)$.
Taking the termwise quotient $D_X$-module, we obtain the {\em reduced cyclic cochains}
\[
\Cyc_{red}^* (JA, JA^\vee) = \prod_{n > 0} {\rm Hom}_{C^\infty_X} (JA^{\tensor n}, C^\infty_X) / C_n .
\]
The Hochschild differential restricts to this subspace to yield a dg $D_X$-module. 
We mimic Definition~\ref{dfn: hochloc} for the $C^\infty$-local version of cyclic cohomology of a $C^\infty$-local algebra~$\sA$. 

\begin{dfn}\label{dfn: cycloc}
The {\em $C^\infty$-local cyclic cochains} of a $C^\infty$-local algebra $\sA$ on $X$ is the sheaf
\[
\Cycloc^*(\sA) = \Omega^*_X[2d] \tensor_{D_X} \Cyc^*_{red} (JA) .
\] 
We denote the global sections of this sheaf of cochain complexes by~$\Cycloc^*(\sA(X))$.
\end{dfn}

To make things concrete, 
consider the most relevant $C^\infty$-local algebra for us: the Dolbeault complex $\Omega^{0,*}_X$ on a complex manifold $X$. 
For this $C^\infty$-local Lie algebra, there is a natural degree zero cocycle in $C^\infty$-local cyclic cohomology.

\begin{lem}
\label{lem: univ}
In complex dimension $d$, 
the functional on $\Omega^{0,*}$ defined by
\[
\Theta^\infty_d (\alpha_0 \tensor \cdots \tensor \alpha_d) = \alpha_0 \wedge \partial \alpha_1 \cdots \wedge \partial \alpha_d
\]
is a degree zero cocycle in $\Cycloc^*(\Omega^{0,*})$. 
\end{lem}

This cocycle is ``universal'' in the sense that it only depends on dimension.

\begin{proof}
The proof is similar to that of Proposition \ref{prop j map}. 
Note that the differential on $C^\infty$-local cochains consists of two terms: the $\dbar$ operator and the ordinary Hochschild differential. 
It follows from graded commutativity of the wedge product that the cochain is cyclic and closed for the Hochschild differential. 
To see that it is closed for the other piece of the differential, observe that
\[
\dbar \Theta^\infty_d(\alpha_0,\cdots,\alpha_d) = \Theta^\infty_d(\dbar \alpha_0, \alpha_1,\ldots,\alpha_d) \pm \Theta_d^\infty(\alpha_0, \dbar \alpha_1,\ldots \alpha_d) \pm \cdots \pm \Theta_d^\infty(\alpha_0, \alpha_1,\ldots \dbar \alpha_d) .
\]
The right hand side is the cocycle $\Theta_d^\infty$ evaluated on the derivation $\dbar$ applied to the element $\alpha_0 \tensor \cdots \tensor \alpha_d$. 
The left hand side is a total derivative and hence vanishes in the $C^\infty$-local cochain complex. 
\end{proof}

\subsection{Local Loday-Quillen-Tsygan theorem and the chiral anomaly}

We now turn to the relationship between cyclic cocycles for a $C^\infty$-local algebra $\sA$ and cocycles for the local Lie algebras $\gl_N( \sA)$ and~$\gl_\infty (\sA)$.
The Loday-Quillen-Tsygan theorem implies the following,
since the map $\lqt$ is natural and hence respects locality everywhere.

\begin{prop}
\label{prop: cycloc}
Let $\sA$ be a $C^\infty$-local algebra.
For every positive integer $N$, there is a map of sheaves
\[
\lqt_N^* : \Cycloc^*(\sA)[-1] \to \cloc^*(\gl_N( \sA)) 
\] 
that factors through a map of sheaves
\[
\lqt^* : \Cycloc^*(\sA)[-1] \to \cloc^*(\gl_\infty( \sA)) = \lim_{N \to \infty} \cloc^*(\gl_N( \sA))  .
\]
\end{prop}

\begin{rmk}
A version of this result was given in~\cite{CLbcov2} for $\sA = \Omega^{0,*}(X)$, 
where $X$ is a Calabi-Yau manifold.
They interpret $C^\infty$-local cocycles for $\Omega^{0,*}(X) \tensor \fgl_\infty$ as the space of ``admissible'' deformations for holomorphic Chern-Simons theory on $X$,
and they identify the cyclic side in terms of Kodaira-Spencer gravity on~$X$.
\end{rmk}

Proposition \ref{prop: cycloc} sends a degree zero $C^\infty$-local cyclic cocycle to a degree one local Lie algebra cocycles for $\fgl_N(\sA)$.
Of particular interest is the case $\sG l_{N} = \fgl_N \tensor \Omega^{0,*}$. 
The degree zero cocycle $\Theta_d^\infty \in \Cycloc^*(\Omega^{0,*})$ from Lemma~\ref{lem: univ} thus determines a degree one cocycle 
\[
\lqt^*_N(\Theta_d^\infty) \in \cloc^*(\sG l_N)
\]
for each $N > 0$. 
In fact, we have already met this class of cocycles for~$\sG l_{N}$. 

\begin{dfn}
For each $N$ and $k$, the functional $\theta_{k,N}(A) = {\rm tr}_{\fgl_N} (A^k)$ defines a homogenous degree $k$ polynomial on $\fgl_N$ that is $\fgl_N$-invariant.
\end{dfn}

\begin{lem}
\label{lem:pullbackofthetainfinity}
For every $N$, 
\[
\lqt_N^*(\Theta_d^\infty) = \fj(\theta_{d+1, N})
\]
where $\fj$ from Definition~\ref{dfn: j}.
\end{lem}

In a sense $\Theta^\infty_d$ is the ``universal'' cocycle --- in that it only depends on the complex dimension and not on any Lie algebraic data --- that determines the most important local cocycles we have encountered before.

This universality is perhaps most apparent when we view cocycles as anomalies to solving the quantum master equation. 
For concreteness, consider the $\beta\gamma$ system with values in $V$ as in Section \ref{sec: qft}.
This theory is natural in the vector space $V$ in the sense that if $V \to W$ is a map of vector spaces, then there is an induced map of theories from the theory based on $V$ to the theory based on $W$.\footnote{This means, for instance, that there is an induced map between the spaces of solutions to the equations of motion.}
Formal aspects of BV quantization implies that anomalies to solving the QME get pulled back along such maps between theories. 

If we choose an identification $V \cong \CC^N$, this implies the the anomaly to solving the $\gl_N = \gl(V)$-equivariant QME is pulled back from the anomaly to solving the $\gl_\infty$-equivariant QME. 
For the $\beta\gamma$ system on $\CC^d$ with values in $\CC^\infty = \cup_{N > 0} \CC^N$, the anomaly to solving the $\gl_\infty$-equivariant QME is precisely the class $\Theta_{d}^\infty$. 

This is consistent with our calculations in Section \ref{sec: qft} and this Lemma \ref{lem:pullbackofthetainfinity}. 
Indeed, if $V$ is additionally a $\fg$-representation, we can further pull-back the anomaly along the map of theories induced by the defining map $\rho : \fg \to \fgl(V)$ of the representation. 

\begin{proof}(of Lemma \ref{lem:pullbackofthetainfinity})
Let $A$ be a dg algebra.
Consider the Lie algebra $\fgl_N(A)$ and the colimit $\gl_\infty(A) = {\rm colim} \; \fgl (A)$. 
At the level of homology, the ordinary Loday-Quillen-Tsygan map is of the form
\[
\begin{array}{ccl}
\clieu_{*+1}(\fgl_N(A)) & \to & {\rm Cyc}_{*}(A) \\
X_0 \wedge \cdots \wedge X_n & \mapsto & \sum_{\sigma \in S_n} (-1)^{\sigma} {\rm tr} \left(X_0 \tensor X_{\sigma(1)} \tensor\cdots \tensor X_{\sigma(n)} \right), 
\end{array}
\] 
which induces a dual map in cohomology ${\rm Cyc}^*(A, A^\vee) \to \clie^{*+1}(\fgl_N(A))$. 
In the formula, we have used the generalized trace map
\[
{\rm tr} : {\rm Mat}_N(A)^{\tensor(n+1)} \to A^{\tensor (n+1)} 
\]
that maps an $(n+1)$-tuple $X_0\tensor \cdots \tensor X_d$ to 
\[
\sum_{i_0,\ldots,i_n} (X_0)_{i_0 i_1} \tensor (X_1)_{i_1i_2} \tensor \cdots \tensor (X_n)_{i_n i_0}
\]
where $(X_k)_{ij} \in A$ denotes the $ij$ matrix entry of~$X_k$.

The map on local functionals is essentially this ordinary (dual) Loday-Quillen-Tsygan map applied to the $\infty$-jets of the commutative algebra $\Omega^{0,*}$. 
Since $\Omega^{0,*}$ is commutative, the generalized trace is simply the trace of the product.

We can thus read off the image of $\Theta^\infty_d$ under the $\ell q t_N^*$ as the local Lie algebra cocycle
\begin{align*}
\ell q t_N^*(\Theta_d^\infty)\left(\alpha_0, \cdots, \alpha_d) = {\rm tr}_{\fgl_N}(\alpha_0 \wedge \partial \alpha_1\wedge \cdots \wedge \partial \alpha_d\right),
\end{align*}
which is precisely $\fj(\theta_{d+1,N})$. 
\end{proof}

\subsection{Holomorphic translation invariant cohomology}

We turn our attention to $C^\infty$-local cyclic cocycles defined on affine space $\CC^d$ that are both translation invariant and $U(d)$-invariant. 
We show that up to homotopy there is a unique such cyclic cocycle on the $C^\infty$-local algebra  $\Omega^{0,*}(\CC^d)$ given by $\Theta^\infty_d$. 

\begin{prop}\label{prop: cyctrans}
The class $\Theta^\infty_d$ spans the $U(d)$-invariant, holomorphic translation invariant, $C^\infty$-local cyclic cohomology of $\Omega^{0,*}(\CC^d)$ in degree zero.
Thus
\[
H^0\left(\Cycloc^*(\Omega^{0,*}(\CC^d))^{U(d) \ltimes \CC^d_{hol}} \right) \cong \CC .
\] 
\end{prop}

For a definition of the notation used in the proposition we refer to Appendix \ref{sec: hol trans}.

\begin{proof}
The calculation is similar to that of the holomorphic translation invariant local Lie algebra cohomology of $\sG_d$ given in Appendix \ref{sec: hol trans}. 
We list the steps of the calculation first, and we will justify them below. 
\begin{enumerate}
\item[(1)] There is an identification of the holomorphic translation invariant deformation complex 
\beqn\label{step1}
\Cycloc^*(\Omega^{0,*}(\CC^d)) \simeq \CC \cdot \d^d z \tensor^{\mathbb{L}}_{\CC[\partial_{z_i}]} {\rm Cyc}_{red}^*(\CC[[z_1,\ldots,z_d]])[d] .
\eeqn
Notice the overall shift down by the dimension $d$. 
\item[(2)] We can recast the right-hand side as the Lie algebra homology of the $d$-dimensional abelian Lie algebra $\CC^d = {\rm span}\left\{\partial_{z_i}\right\}$ with coefficients in the module 

\noindent$\Cyc_{red}^*(\CC[[z_1,\ldots,z_d]]) \d^d z [d]$:
\[
\CC \cdot \d^d z \tensor^{\mathbb{L}}_{\CC[\partial_{z_i}]} {\rm Cyc}_{red}^*(\CC[[z_1,\ldots,z_d]])[d] \cong {\rm C}^{\rm Lie}_*\left(\CC^d ; \Cyc_{red}^*(\CC[[z_1,\ldots,z_d]]) \d^d z \right)[d] .
\] 
\item[(3)] The $U(d)$-invariant subcomplex is quasi-isomorphic to $\left(\CC[t] / \CC\right) [2d]$, where $t$ is a formal variable of degree $+2$. 
From this, the claim follows. 
\end{enumerate}

Step (1) follows from a result completely analogous to Corollary 2.29 in \cite{BWthesis} for local Lie algebra cohomology. 
The commutative algebra $\CC[\partial_{z_i}]$ is equal to the enveloping algebra of the abelian Lie algebra $\CC^d = {\rm span} \{\partial_{z_i}\}$. 
Hence, the right hand side of Equation (\ref{step1}) is precisely the Lie algebra homology in step (2). 

We now justify Step (3). 
First, we apply the Hochschild-Kostant-Roesenberg theorem to the cyclic homology of the ring $\CC[[z_1,\ldots,z_d]]$.
It asserts a quasi-isomorphism 
\[
{\rm Cyc}_*(\CC[[z_1,\ldots,z_d]]) \simeq \left(\CC[[z_i]][\d z_i] [t^{-1}], t \d_{dR}\right)
\]
where the $\d z_i$ have degree $-1$ and $t$ is a formal parameter of degree $+2$ (note that the operator $t \d_{dR}$ is of degree $+1$). 
The formal Poincar\'{e} lemma applied to $\hD^d$ then implies a quasi-isomorphism
\[
{\rm Cyc}_*(\CC[[z_1,\ldots,z_d]]) \simeq \CC [t^{-1}] .
\]
Thus, the holomorphic invariant subcomplex is quasi-isomorphic to
\beqn\label{step3}
\clieu_*(\CC^d ; (\CC [t] / \CC) \cdot \d^d z) [d] .
\eeqn
Here, we have identified the dual of $\CC[t^{-1}]$ with $\CC[t]$ and quotiented out by the constant term since we are taking reduced cohomology. 
Notice that $(\CC [t] / \CC) \cdot \d^d z$ has a trivial $\CC^d$-action. 

We have yet to take $U(d)$-invariants. 
The complex (\ref{step3}) is equal to
\[
\Sym^* \left(\CC^d[1]\right) \tensor (\CC [t] / \CC) \cdot \d^d z) [d] .
\]
A $U(d)$-invariant element must be proportional to the factor $\partial_{z_1} \cdots \partial_{z_d} \in \Sym^d(\CC^d[1])$.
Hence, the $U(d)$-invariant subcomplex is
\[
\CC \cdot (\partial_{z_1} \cdots \partial_{z_d}) \tensor (\CC [t] / \CC) \cdot \d^d z) [2d] = \left(\CC[t] / \CC\right) [2d]
\] 
as desired. 
The class of $\Theta^\infty_d$ corresponds to the element $t^{d}$ in this presentation. 
\end{proof}

Consider the dg algebra $A_d$ that we have used as an algebraic model for the Dolbeault complex of punctured affine space $\Omega^{0,*}(\CC^d \setminus 0)$. 
In Theorem 2.3.5 of \cite{FHK}, they show that there is a unique $U(d)$-invariant class in the cyclic cohomology of $A_d$ in degree one given by the functional
\[
a_0 \tensor \cdots \tensor a_d \mapsto \oint a_0 \wedge \partial a_1 \cdots \wedge \partial a_d .
\]
Up to our conventional degree shifts, we are seeing the analogous uniqueness result at the level of local functionals.

\subsection{A noncommutative example}

The main objects that have appeared in this section so far are the cyclic chains and cochains of the commutative dg algebra $\Omega^{0,*}(X)$. 
In this subsection, we display a variant of the above examples where we introduce a noncommutative deformation of this algebra. 
Specifically, we assume $X$ is a holomorphic symplectic manifold and assume we have a deformation quantization of holomorphic functions.
This introduces a dg algebra deformation of the Dolbeault complex, and we can consider the resulting deformation of the current algebra.
We display the flexibility of our techniques by exhibiting a free field realization of the resulting current algebra using a noncommutative version of the $\beta\gamma$ system. 

Noncommutative gauge theories appear in the description of the open sectors of superstring theories~\cite{WittenNonComm}, and our primary interest in this class of examples is that we expect them to appear as a symmetries in the corresponding sectors of supergravity and $M$-theory. 
More definitive results in this direction have appeared in the program for studying the superstring theory through its holomorphic twists developed in the papers of Costello and Li in~\cite{CostelloLiSUGRA} and by Costello in~\cite{CosOmega, CosM2}. 

As usual, suppose $X$ is a complex manifold, and as above, consider the local Lie algebra $\sG l_N = \Omega^{0,*}(X) \tensor \fg$ on $X$ for every $N > 0$. 
If $X$ is additionally holomorphic symplectic, we obtain a deformation of this family of local Lie algebras described in the following way. 
Suppose that $\star_\epsilon$ is a formal holomorphic deformation quantization of $(X,\omega)$. 
This is an $\epsilon$-dependent associative product on holomorphic functions 
\[
\star_\epsilon : \sO^{hol}(X) \times \sO^{hol}(X) \to \sO^{hol}(X)[[\epsilon]]
\]
where, term-by-term in $\epsilon$, the product is given by a holomorphic bidifferential operator. 
This associative product on $\sO^{hol}(X)[[\epsilon]]$ extends to one on the Dolbeault complex, giving the following definition. 

\begin{dfn}
Define the sheaf of associative dg algebras
\[
\sA_\epsilon := (\Omega^{0,*}(X)[[\epsilon]], \dbar, \star_{\epsilon})
\]
where the differential is the usual $\dbar$ operator, and $\star_{\epsilon}$ is the Moyal product induced from the deformation quantization.
\end{dfn}

In fact, $\sA_\epsilon$ is essentially a $C^\infty$-local algebra in the sense of Definition~\ref{def: localalg}. 
The only subtlety is that $\sA_\epsilon$ is not given by the sections of a finite rank vector bundle.
However, it is a pro-$C^\infty$-local algebra in the sense that it can be expressed as a limit of $C^\infty$-local algebras
\[
\sA_\epsilon = \lim_{k \to \infty} \sA_\epsilon / \epsilon^{k+1} .
\] 

This algebra allows us to define a non-commutative variant of the current algebra. 
Namely, we can consider the Lie algebra of $N \times N$ matrices with values in $\sA_\epsilon$ that we denote by $\gl_N(\sA_{\epsilon})$. 
Again, this is not a local Lie algebra in the strict sense, since the underlying vector bundle is infinite dimensional. 
However, it is finite rank over the ring $\CC[[\epsilon]]$, and all of the same constructions of local Lie algebras still make sense in this context. 
Note that this current algebra reduces modulo $\epsilon$ to the local Lie algebra $\sG l_N = \Omega^{0,*}_X \tensor \gl_N$:
\[
\sG l_N = \lim_{\epsilon \to 0} \gl_N(\sA_{\epsilon})  .
\]

\subsubsection{Classical Noether current}

Just like in the case of the ordinary current algebra associated to $\sG l_N$, we can contemplate a free field realization of $\gl_N(\sA_{\epsilon})$.
The simplest way to do this is to consider the analogue of the $\beta\gamma$ system in this noncommutative context. 
The $\beta\gamma$ system was built from the Dolbeault complex on the complex manifold $X$. 
The non-commutative variant is obtained by replacing the Dolbeault complex with the dg algebra~$\sA_\epsilon$. 

Let $V$ be a finite dimensional $\CC$-vector space.
The free theory we consider has fields 
\[
(\gamma, \beta) \in \sA_\epsilon \tensor V \oplus \sA_\epsilon \tensor V^* [d-1] 
\]
and action functional
\[
S(\beta,\gamma) = \int_X \Tr_V(\beta \star_{\epsilon} \dbar \gamma)
\]
By trace we mean the usual map $\Tr_V : \End(V) = V \tensor V^* \to \CC$. 
We will refer to this as the ``non-commutative $\beta\gamma$ system" on $X$ with values in $V$.

\begin{rmk}
Note that this is not a classical theory in a strict sense because the space of fields is not the sections of a {\em finite} rank vector bundle. 
We can make sense of this rigorously by considering our theory as one defined over the base ring $\CC[[\epsilon]]$. 
In other words, we have defined a family of field theories over the formal disk with coordinate~$\epsilon$. 
\end{rmk}

\begin{lem}\label{lem: nonbg}
As a classical BV theory, the non-commutative $\beta\gamma$ system with values in $V$ is equivalent to the ordinary $\beta\gamma$ system with values in $V$ (considered as a trivial family of field theories over the formal disk with coordinate~$\epsilon$). 
\end{lem}

\begin{proof}
Locally, on $\CC^d$, the $\star_{\epsilon}$-product has the form
\[
f \star_{\epsilon} g = fg + \epsilon \varepsilon_{ij} \frac{\partial f}{\partial z_i} \frac{\partial g}{\partial z_j} + \cdots
\]
From this, we see that $\beta \star_{\epsilon} \dbar \gamma$ and $\beta \dbar \gamma$ differ by a total derivative. 
Thus, locally, this non-commutative $\beta\gamma$ system is equivalent to the usual one (up to adjoining the formal parameter~$\epsilon$).
\end{proof}

It appears that adding the non-commutative deformation does not deform the free holomorphic field theory. 
Once we consider symmetries, however, we see a deformation of the usual free field realization. 

Fix an identification of $V \cong \CC^N$, for some $N \geq 1$. 
As in the non-commutative case, there is a symmetry of this $\beta\gamma$ system by the current algebra built from the ordinary local Lie algebra $\sG l_N$, but this does not use the symplectic structure on $X$. 
However, once we turn on the non-commutative deformation, we see that the $\beta\gamma$ system has a symmetry by the deformed current algebra built from $\gl_N(\sA_{\epsilon})$. 

Indeed, there is a Noether current in this setup given by
\[
I_{\epsilon,N} (\alpha, \beta,\gamma) = \int \Tr_{V}(\beta \wedge (\alpha \star_{\epsilon} \gamma)) 
\]
where $\alpha \in \gl_N(\sA_{\epsilon})$.
By $\alpha \star_\epsilon \gamma$ we mean the algebra action of $\gl_N(\sA_{\epsilon})$ on $\sA_\epsilon \tensor V$. 

\begin{lem}
This Noether current determines a map of factorization algebras on $X$
\[
J^\cl_\epsilon: \UU(\gl_N(\sA_{\epsilon})) \to \Obs^{\cl}_{\epsilon, N}
\]
where $\Obs^{\cl}_{\epsilon, N}$ is the factorization algebra of classical observables of the non-commutative $\beta\gamma$ system with values in $V = \CC^N$.
Modulo $\epsilon$, this map reduces to the map of factorization algebras in Proposition~\ref{prop:CNT}.
\end{lem}

\subsubsection{Equivariant quantization}

Since the noncommutative $\beta\gamma$ system is still free, there exists a unique quantization $\Obs^\q_{\epsilon, N}$ as a factorization algebra on $X$ for each $N$. 

Let's turn to the quantization of the classical $\fgl_N(\sA_\epsilon)$ symmetry, where the situation is similar to the $\sG_X$-equivariant $\beta\gamma$ system studied in Section \ref{sec: qft}.
Although the global case is interesting, we will restrict ourselves to the simplified local situation where 
\[
X = \CC^d = \CC^{2n}
\] 
and $\omega$ is the standard symplectic form. 
We can employ analogous Feynman diagrammatic methods to contemplate quantum equivariance in the noncommutative context.

We ask that the Noether current $I_{\epsilon,N}$ solves the $\fgl_N(\sA_{\epsilon})$-equivariant quantum master equation.
Locally, on $\CC^d$, the obstruction to satisfying the QME is given by the following local cocycle 
\beqn\label{noncommobs}
\int \Tr_{\fgl_N} (\alpha \star_{\epsilon} \partial \alpha \star_{\epsilon} \cdots \star_{\epsilon} \partial \alpha) \in \cloc^*(\fgl_N(\sA_\epsilon)) .
\eeqn
In the ordinary commutative case, we were able to characterize this anomaly as being determined by an element in $\Sym^{d+1}(\fg^\vee)$.
For the noncommutative situation, we do not have a direct way of identifying this local cocycle.

We arrive at an explicit characterization by taking the large $N$ limit, where we are able to identify this anomaly algebraically. 
Indeed, we have the Loday-Quillen-Tsygan map for local functionals
\[
\lqt^* : \Cycloc^*(\sA_\epsilon)[-1] \to \cloc^*(\gl_\infty(\sA_\epsilon)) = \lim_{N \to \infty} \cloc^*(\gl_N( \sA))  .
\]
Thus, the large $N$ anomaly must come from a class in $\Cycloc^*(\sA_\epsilon)$ of cohomological degree zero. 

By a similar proof as in Proposition \ref{prop: cyctrans}, one can show that the cohomology of the translation invariant subcomplex of $\Cycloc^*(\sA_\epsilon)$ is equal to (a shift of) the cyclic cohomology of the formal Weyl algebra
\[
\Cyc^*(\Hat{A}_{2n}, \Hat{A}_{2n}^\vee) [2n] .
\]
Here, $\Hat{A}_{2n}$ is the formal Weyl algebra on generators $\{x_1,\ldots, x_n, y_1,\ldots y_n\}$ satisfying the commutation relation
\[
[x_i, y_j] = \epsilon \delta_{ij} .
\] 
This cyclic cohomology is studied in depth in \cite{Willwacher}, where it is shown that there is {\em unique}, up to scaling, nontrivial class in the cyclic cohomology
\[
\Theta_{\epsilon}^\infty \in HC^{2n} (\Hat{A}_{2n}, \Hat{A}_{2n}^\vee) .
\]
For us, a multiple of this class represents the anomaly to the equivariant quantization the noncommutative $\beta\gamma$ system at large $N$. 

We can now use the universal nature of this class to characterize anomalies at finite $N$ to obtain the following quantum Noether map. 

\begin{prop}
The $\fgl_N(\sA_\epsilon)$-equivariant quantization determines a map of factorization algebras on $\CC^d = \CC^{2n}$: 
\[
J^\q_\epsilon: \UU_{a \Theta_{\epsilon, N}} (\fgl_N(\sA_\epsilon)) \to \Obs^\q_{\epsilon, N}
\]
where $a \Theta_{\epsilon, N} \in H^1_{loc}(\fgl_N(\sA_{\epsilon}))$ is scalar multiple the class obtained from the universal cyclic cocycle $\Theta_{\epsilon}^\infty$ under the Loday-Quillen-Tsygan map
\[
\lqt^* : \Cycloc^*(\sA_\epsilon)[-1] \to \cloc^*(\gl_N( \sA))  .
\]
\end{prop}

\begin{rmk}
In order to nail down the constant $a$ would require a tedious, albeit seemingly straightforward, Feynman diagram analysis akin to Section \ref{sec: qft}. 
\end{rmk}

By Lemma \ref{lem: nonbg}, we see that on $\CC^{2n}$ the factorization algebra of the noncommutative $\beta\gamma$ system $\Obs^\q_{\epsilon, N}$ is actually isomorphic to the factorization algebra
\[
\Obs^\q_{\CC^{N}} \tensor_{\CC} \CC[[\epsilon]]
\]
where $\Obs^\q_{\CC^{N}}$ is the ordinary $\beta\gamma$ system of maps $\CC^{2n} \to \CC^N$. 
Thus, as an immediate corollary, we see that the quantum Noether map is of the form
\[
\UU_{a \Theta_{\epsilon, N}} (\fgl_N(\sA_\epsilon)) \to \Obs^\q_{\CC^{N}} \tensor_{\CC} \CC[[\epsilon]] .
\]
This means that inside the $\beta\gamma$ system we have {\em two} different free field realizations: (1) the one from Section \ref{sec: qft} where we realized the ordinary current algebra at some central extension in $\Obs^\q_{\CC^{N}}$, and (2) the one we have just exhibited, which realizes a central extension of the algebra $\fgl_N(\sA_{\epsilon})$. 

\appendix

\addtocontents{toc}{\protect\setcounter{tocdepth}{1}}

\section{Computing the deformation complex}\label{sec: hol trans}

In this appendix we prove Proposition~\ref{prop: trans j}. 
That is, we compute the holomorphically translation invariant component of $H_{\rm loc}(\sG_d)$, 
the Lie algebra cohomology of the local Lie algebra $\sG_d = \Omega^{0,*}_c \tensor \fg$ on~$\CC^d$. 

\subsection{Holomorphic translation invariance}

We have already discussed the local cohomology cochain complex $\cloc^*(\sG_d)$ in Section~\ref{sec:cloc}.
To pick out the subcomplex of holomorphically translation invariant elements,
we introduce yet another dg Lie algebra $\CC^{d}_{\rm hol}$ whose invariants are precisely this subcomplex.

\begin{dfn}
Let $\CC^{d}_{\rm hol} = \CC^{2d} \oplus \CC^d[1]$ be generated by the partial derivatives $\partial/\partial z_i$ and $\partial/\partial \zbar_i$ in degree 0 and by elements $\{\Bar{\eta}_i\}_{i=1}^d$ in degree $-1$.
Equip it with a trivial bracket and with a differential that $\eta_i$ to $\frac{\partial}{\partial \zbar_i}$.
\end{dfn}

There is a canonical inclusion of dg Lie algebras
\[
\CC\{\partial / \partial z_1, \ldots, \partial / \partial z_d\} \hookrightarrow \CC^{d}_{\rm hol}
\]
so that any representation ``forgets'' down to an action of holomorphic infinitesimal translations.
But a dg representation of this abelian dg Lie algebra has an action of all the partial derivatives,
but where the actions of the $\partial/\partial \zbar_i$ are trivial homotopically.
In this sense $\CC^{d}_{\rm hol}$ encodes the idea of infinitesimal translations that are purely holomorphic up to homotopy.

Directly from these definitions one can verify the following.

\begin{lem}
The canonical inclusion of enveloping algebras
\[
\CC[\partial / \partial z_1, \ldots, \partial / \partial z_d] \hookrightarrow U(\CC^{d}_{\rm hol})
\]
is a quasi-isomorphism.
\end{lem}

In other words, $U(\CC^{d}_{\rm hol})$ is quasi-isomorphic to the algebra of constant coefficient holomorphic differential operators on~$\CC^d$. 

\subsection{Language to phrase the main result}

We now turn to the main objects of interest here.

\begin{dfn}
Let $\cloc^*(\sG_d)^{\CC^d_{\rm hol}}$ denote the subcomplex in $\cloc^*(\sG_d)$ consisting of elements strictly invariant under $\CC^d_{\rm hol}$.
Let
\[
\cloc^*(\sG_d)^{U(d) \ltimes \CC^d_{\rm hol}}
\]
denote the subcomplex of elements that are invariant under both translation by $\CC^d_{\rm hol}$ and rotation by the unitary group~$U(d)$.
\end{dfn}

We are interested in the map $\fj$, from Section~\ref{sec: hol trans main}, for the affine space $\CC^d$.
We will use this map to completely characterize the degree one $U(d)$-invariant, holomorphically translation invariant local functionals on~$\sG_d$. 

The degree one result will follow from a stronger, general result on the cochain level.
To formulate it, we introduce some notation.

\subsubsection{De Rham forms for dg Lie algebras}

Let $\cL$ denote an arbitrary dg Lie algebra. 
Interpret the dg commutative algebra given by the Chevalley-Eilenberg cochains $\clie^*(\cL)$ as functions on a formal moduli space~$B \cL$:
\[
\sO(B \cL) = \clie^*(\cL) .
\]
In the same line of thought, define the $k$-forms on $B\cL$~by
\begin{align*}
\Omega^k(B \cL) & := \clie^*(\cL ; \Lambda^k(\cL^\vee [-1])) \\
 & =  \clie^*(\cL ; \Sym^k(\cL^\vee))[-k] .  
\end{align*}
Here, $\cL^\vee$ denotes the coadjoint representation of~$\cL$. 

\begin{eg}
A simple example gives evidence that this interpretation is not so far-fetched.
Consider the case $\cL = \CC^n [-1]$, a purely abelian Lie algebra.
Then
\[
\sO(B \cL) = \clie^*(\cL) = \CC[[t_1,\ldots,t_n]]
\]
with generators $t_i$ in degree 0.
(These generators are the coordinates on the formal $n$-disk.)
Similarly, the de Rham forms are
\begin{align*}
\Omega^k(B \cL) 
&= \sO(B \cL)  \otimes \Lambda^k(\cL^\vee) \\
 &=\CC[[t_1,\ldots, t_n]] \tensor \Lambda^k[\d t_1, \cdots, \d t_n],
\end{align*}
where we use $\d t_i$ to denote a basis for the coadjoint representation~$\cL^\vee$.
(We use $\Lambda^k$ denote the $k$th exterior power of the vector space spanned by those generators.)
Everything is in cohomological degree zero.
Manifestly everything agrees with the usual constructions of algebraic de Rham forms.
\end{eg}

Let $\partial : \Omega^{k}(B\cL) \to \Omega^{k+1}(B\cL)$ denote the de Rham operator for $B\cL$. 
The space of {\em closed} $k$-forms is defined by the totalization of the double complex
\[
\Omega^{k}_{cl}(B \cL) = {\rm Tot}\left( \Omega^k(B\cL) \xto{\partial} \Omega^{k+1}(B \cL)[-1] \to \cdots \right).
\]
The case where $k=0$ is the usual de Rham complex, which we will denote by~$DR(B\cL)$.

The constant functions on $B\fg$ can be appended to obtain a complex
\[
DR_{\rm aug}(B\cL) = {\rm Tot}\left( \CC[1] \to \Omega^0(B\cL) \xto{\partial} \Omega^{1}(B \cL)[-1] \to \cdots \right),
\]
which is acyclic.
(Consider the spectral sequence for the underlying double complex where one uses the de Rham differential first. The Poincar\'e lemma ensures the cohomology vanishes on this page.)
The inclusion map $\Omega^{k}_{cl}(B \cL) \to DR_{\rm aug}(B\cL)$ has quotient given by the opposite truncation
\[
 {\rm Tot}\left( \CC[1] \to \Omega^0(B\cL) \xto{\partial} \cdots \xto{\partial} \Omega^{k-1}(B \cL)\right).
\]
We denote it by~$\tau_{< k} DR_{\rm aug}(B \cL)$. 
The de Rham differential $\partial: \Omega^{k-1}(B \cL) \to \Omega^{k}(B \cL)$ determines a cochain map
\[
\tau_{< k} DR_{\rm aug}(B \cL)[d] \xto{\partial} \Omega^{k}_{cl}(B \cL),
\]
whose cone is the augmented de Rham complex.
This observation implies that the map determines a quasi-isomorphism from the truncation to the closed $k$-forms.

\subsubsection{Improving $\fj$}

Before stating the main result of this appendix, we note that there is a natural enhancement of the cochain map
\[
\fj : \Sym^{d+1} (\fg^*)^{\fg} [-1] \to \cloc^*(\sG_d)
\]
from Section~\ref{sec: hol trans main} to a cochain map
\beqn\label{fj1}
\fj : \Omega^{d+1}_{cl} (B \fg) [d] \to \cloc^*(\sG_d)  
\eeqn
that we now describe. 

Because $\Omega^d(B \fg) \cong \clie^*(\fg, \Lambda^{d} \fg^\vee)$, 
a $d$-form $\eta$ of cohomological degree $m$ determines a linear map
\[
\eta: \Lambda^d \fg \to \clie^m(\fg).
\]
We can extend this map over $\Omega^{0,*}$ to obtain a linear map
\[
\widetilde{\eta}: \Lambda^d \sG_d \to \clie^*(\sG_d),
\]
and an element of $\clie^*(\sG_d)$ can be evaluated on an element of $\sG_d$ to obtain a de Rham form.
Hence, we define the element $\widetilde{\fj}(\eta)$ in $\clie^*(\sG_d)$ by 
\[
\widetilde{\fj} (\eta) (\alpha) = \widetilde{\eta}(\partial \alpha \wedge \cdots \wedge \partial \alpha) (\alpha) .
\]
We extend $\widetilde{\fj}$ to forms $\Omega^k(B\fg)$ with $k < d$ as the zero map.

Direct computation then shows the following.

\begin{lem}
The construction above determines a cochain map $\widetilde{\fj}: \tau_{<d+1} DR_{\rm aug}(B\fg)[2d] \to \cloc^*(\sG_d)$. 
\end{lem}

As this truncated de Rham complex $\tau_{<d+1} DR_{\rm aug}(B\fg)$ is quasi-isomorphic to  $\Omega^d_{cl}(B\fg)$, 
we obtain the existence of the desired map~(\ref{fj1}),
although we do not provide an explicit formula.

\subsection{The main result}

We now state the main result.

\begin{prop}\label{prop: local def}
The map $\fj$ factors through the subcomplex of invariants under rotation and holomorphic translation:
\beqn
\fj : \Omega^{d+1}_{cl}(B \fg) [d] \to \cloc^*(\sG_d)^{U(d) \ltimes \CC^d_{\rm hol}}.
\eeqn
In particular, if $\fg$ is an ordinary Lie algebra (i.e., concentrated in degree zero), then we obtain an isomorphism
\[
H^1(\fj) : \Sym^{d+1}(\fg^\vee)^\fg \xto{\cong} H^1  \left(\cloc^*(\sG_d)\right)^{U(d) \ltimes \CC^d_{\rm hol}}.
\] 
\end{prop}

Note that this result contains Proposition~\ref{prop: trans j}, since for an ordinary Lie algebra one has
\[
H^1( \Omega^{d+1}_{cl}(B \fg) [d]) = H^{d+1} ( \Omega^{d+1}_{cl}(B \fg) ) = H^0 (\fg, \Sym^{d+1}(\fg^\vee)) .
\]

In brief, the proof involves two central ideas.
The first is that the translation-invariant local functionals ought to be built from translation-invariant differential operators and translation-invariant measures,
and such functionals are thus pinned down by their behavior at one point.
The second is that rotation invariance then drastically cuts down the remaining possibilities.
Indeed, as the proposition indicates, the only freedom is concentrated in the dependence on the Lie algebra $\fg$ and not on the spatial directions along~$\CC^d$.

We start by ignoring the differentials and simply figure out the graded subspaces of invariant elements.
Note that for a complex $V$, we use $V^\#$ to denote the underlying graded vector space.

\begin{lem}
The subspace $\cloc^\#(\sG_d)^{\CC^d}$ of elements invariant under translation along $\CC^d$ is isomorphic to 
\[
(\Omega^\#(\CC^d))^{\CC^d}[2d] \otimes \cred^\#(\fg[[z_1,\ldots,z_d, \zbar_1,\ldots, \zbar_d, \d\zbar_1,\ldots, \d\zbar_n]]).
\]
\end{lem}

Note the overall downward shift by degree~$d$.
The translation-invariant differential forms are
\[
\CC[\d z_1,\ldots, \d z_d, \d \zbar_1,\ldots, \d \zbar_d].
\] 
The graded Lie algebra underlies the dg Lie algebra of Dolbeault forms on the formal $d$-dimensional disk with values in~$\fg$,
which we interpret at the fiber at the origin of the jet bundle of~$\sG_d$.

%Somewhat abusively, we are using the $z_i$ and $\zbar_i$ for both coordinates on $\CC^d$ and in the jet direction.

\begin{proof}
Here we are just noting a simple fact: 
for any trivial bundle on $\CC^d$, 
translation-invariant sections are thus determined by their value at a single point,
which we can take to be the origin.

Each bundle $\Lambda^k T_\CC^* \to \CC^d$, whose sections are complex-valued $k$-forms, admits a natural trivialization by the frame arising from the choice of linear coordinates. 
For instance, the complexified cotangent bundle $T^*_\CC \to \CC^d$ has the frame $\{\d z_1,\ldots,\d z_d, \d \zbar_1,\ldots,\d \zbar_d\}$; for other $k$, take wedge powers of that frame. 
The bundle $\fg \times \CC^d \to \CC^d$ is likewise trivial,
and hence the jet bundle inherits a trivialization.
The trivialization is explicitly given by using the linear coordinate system arising from identifying the manifold with the vector space $\CC^d$;
it gives a natural basis for differential operators and hence for jets.

Putting these observations together, 
the fiber at the origin of the jet bundle for $\sG_d$ can be seen as Dolbeault forms on the formal $d$-dimensional disk with values in $\fg$.
As $\cloc^*$ is a version of reduced Lie algebra cochains, we obtain the claim.
\end{proof}

We would now like to trivialize homotopically the action of the antiholomorphic derivatives.
On the formal $d$-dimensional disk, there is a natural trivialization (by contraction with the vector fields $\partial_{\zbar_i}$),
which also makes sense on $\CC^d$ globally.
The strict invariants for the extended Lie algebra $\CC^d_{\rm hol}$ are thus expressions that have no dependence on the antiholomorphic coordinates~$\zbar_i$.

\begin{lem}\label{lem: a8}
The underlying graded subspace of the $\CC^d_{\rm hol}$-invariants $\cloc^*(\sG_d)^{\CC^d_{\rm hol}}$ is isomorphic to 
\[
\CC[\d z_1,\ldots, \d z_d][2d] \otimes \cred^\#(\fg[[z_1,\ldots,z_n]]),
\]
the reduced Lie algebra cochains of the Lie algebra~$\fg[[z_1,\ldots,z_n]]$.
\end{lem}

Here $\CC[\d z_1,\ldots, \d z_d][d]$ forms the translation-invariant {\em holomorphic} differential forms, 
shifted down by degree~$d$.
This is the underlying associated graded of the translation invariant subcomplex of the de Rham complex 
\[
\Omega^* (\CC^d, M)
\]
where $M$ is the $D$-module $\cred^*(J \sO^{hol}(\CC^d))$. 
The flat connection on this $D$-module is induced from the canonical one on the $\infty$-jets of the trivial bundle. 

Thanks to some standard results in invariant theory, there is then a simple answer for which such elements are $U(d)$-invariant.

\begin{lem}
The $U(d) \ltimes \CC^d_{hol}$-invariant subspace $\cloc^\#(\sG_d)^{U(d) \ltimes \CC^d}$
is canonically isomorphic to the (shift of the) reduced de Rham forms
\[
\Omega^\sharp_{\rm red}(B\fg) [2d] = \sO_{red}(B \fg)[2d] \oplus \Omega^{1} (B \fg)[2d-1] \oplus \cdots \oplus \Omega^{d}(B \fg)[d]. 
\]
\end{lem}

Here we mean that there is no de Rham differential, 
but the $k$-forms are put in their ``usual'' place 
(i.e., in our motivating example, the $k$-forms would sit in degree $k$).
By $\sO_{red}(B \fg)$ we mean that we quotient out the copy $\Sym^0(\fg^\vee)$ of the constants from~$\clie^*(\fg)$.

\begin{proof}
Sitting inside of $U(d)$ is its center, a copy of $U(1)$ as multiples of the identity.
This group equips the $\CC^d_{hol}$-invariant subcomplex with a weight grading, as follows.
The group $U(d)$ acts in the defining way on $\CC^d$,
so each coordinate $z_i$ has weight $1$ and so $\d z_i$ also has weight 1.
Each $k$-form has weight $k$; for instance, the volume element $\d^d z$ has weight~$d$.
Let $\Lambda^k[\d z_1,\ldots,\d z_d]$ denote the vector space of translation-invariant holomorphic $k$-forms.
Every element in this space has weight~$k$.

On the other hand, $z_i^\vee$ has weight~$-1$. 
Let 
\[
\Sym^{>0} \left(\fg^\vee [z_1^\vee,\ldots,z_d^\vee][-1] \right)_{(-k)}
\]
denote the subspace of elements with weight~$-k$.
This space is spanned by symmetric words built from monomials of the form $x \otimes (z_1^\vee)^{i_1} \cdots (z_d^\vee)^{i_d}$, where $x \in \fg^\vee$,
and the sum of the $z^\vee$-degrees over all the monomials must add to~$k$.

Our complex is built from both kinds of elements.
To have total weight zero, a monomial in these terms $z_i^\vee$ and $\d z_i$ must have an equal number of $z_i^\vee$ and $\d z_i$.
In other words, the weight zero elements of our complex decomposes as a direct sum
\beqn\label{decomp}
\bigoplus_{k = 0}^n \Lambda^k[\d z_1,\ldots,\d z_d] \otimes \Sym^{>0} \left(\fg^\vee [z_1^\vee,\ldots,z_d^\vee][-1] \right)_{(-k)} .
\eeqn
But we wish to go a step further and pick out the invariants under the action of the entire group~$U(d)$.

We will denote by $V$ the fundamental $U(d)$-representation, and $V^\vee$ its dual.
W can rewrite the decomposition (\ref{decomp}) as 
\[
\bigoplus_{k = 0}^n \Lambda^k(V)[-k] \tensor \Sym^{>0} \left(\fg^\vee [-1] \tensor \Sym(V^\vee)\right)_{(-k)} .
\]
We expand the term $\Sym^{>0} \left(\fg^\vee [-1] \tensor \Sym(V^\vee)\right)$ as
\beqn\label{gack}
\Sym^{>0} \left(\fg^\vee[-1] \tensor \left(\CC \oplus V^\vee  \oplus \Sym^2(V^\vee)\oplus \cdots \right)\right) .
\eeqn
Using the relation 
\beqn\label{symrel}
\Sym^{>0} (W \oplus Z) = \Sym^{>0} (W) \oplus (\Sym(W) \tensor \Sym^{>0}(Z)),
\eeqn
we see that this expression~(\ref{gack}) is equal to 
\begin{align}
\Sym^{>0} & \left(\fg^\vee[-1] \oplus \fg^\vee [-1] \tensor V^\vee \right) \oplus \label{sym1} \\
& \Sym \left(\fg^\vee[-1] \oplus \fg^\vee [-1] \tensor V^\vee\right) \tensor \Sym^{>0} \left(\fg^\vee[-1] \tensor \left(\Sym^2(V^\vee)\oplus \cdots \right) \label{sym2} \right)
\end{align}
In fact, we want to the $U(d)$-invariants of the tensor product of this enormous mess with the totally antisymmetric representation~$\Lambda^*(V)$. 
Thus, none of the terms $\Sym^k(V^\vee)$ can contribute, so we can forget about the second line~(\ref{sym2}) when we taking $U(d)$-invariants. 

Using the relation~(\ref{symrel}) again, we expand the first line~(\ref{sym1}) as
\[
 \Sym^{>0} \left(\fg^\vee[-1]\right) \oplus \Sym \left(\fg^\vee[-1]\right) \tensor \Sym^{>0} \left(\fg^\vee [-1] \tensor V^\vee\right) .
 \]
Note that the first term has $U(d)$-weight zero. 
Thus, we find that the space of $U(d)$-invariants is equal to the $U(d)$-invariants of
\[
 \Sym^{>0} \left(\fg^\vee[-1]\right) \oplus \bigoplus_{k = 1}^n \Lambda^k(V) [-k] \tensor \left(\Sym \left(\fg^\vee[-1]\right) \tensor \Sym^{>0} \left(\fg^\vee [-1] \tensor V^\vee\right) \right)_{(-k)} .
\]
Once we turn on the Lie differential, the first term above (corresponding to $k=0$ in our original notation) is precisely 
\[
\left(\Sym^{>0}(\fg^\vee[-1]), \d_{CE}\right) = \sO_{red}(B \fg) .
\]
Note that $U(d)$ acts trivial here.

The $k=1$ term is of the form
\[
V[-1] \tensor \Sym (\fg^\vee[-1]) \tensor (\fg^\vee \tensor V^\vee[-1]) .
\]
We are left to compute the $U(d)$-invariants of $V \tensor V^\vee$, which is one dimensional generated spanned by the identity $\id_V \in V \tensor V^\vee$. 
Thus, the space of $U(d)$-invariants corresponding to the $k=1$ term is equal to 
\[
\Sym (\fg^\vee [-1]) \tensor \fg^\vee [-2] 
\]
which we identify with $\Omega^{1}(B\fg)[-1]$ once we turn on the Lie differential. 

In general, we see that for each $k$ we are looking at the $U(d)$-invariants of
\[
\Lambda^k(V) \tensor \Sym (\fg^\vee[-1]) \tensor \Sym^k(\fg^\vee) \tensor \Lambda^k(V^\vee) [-2k]  .
\]
Extracting the dependence on $V$, we must compute the $U(d)$-invariants of $\Lambda^k(V) \tensor \Lambda^k(V^\vee)$.

It is a standard fact in invariant theory that the $U(d)$-invariants of $V^{\tensor k} \tensor (V^\vee)^{\tensor l}$ is zero unless $k=l$, in which case the 
space of invariants is spanned by permutations of the identity morphism $V^{\tensor k} \to V^{\tensor k}$. 
See, for instance, Theorem 2.1.4 of \cite{Fuks}. 
Since we are taking the antisymmetric product, each permutation is equal to some multiple of the identity. 
Thus, the $U(d)$-invariants of $\Lambda^k(V) \tensor \Lambda^k(V^\vee)$ is one-dimensional spanned by the identity. 

It follows that once we turn on the Lie differential, the $U(d)$-invariants of the degree $k$ piece in the decomposition is equal to
\[
\Sym (\fg^\vee[-1]) \tensor \Sym^k(\fg^\vee) [-2k] = \Omega^{k} (B \fg) [-k] .
\]
Accounting for the overall shift by $2d$, we obtain the result. 
\end{proof}

\begin{proof}[Proof of Proposition~\ref{prop: local def}]
We have observed that before turning on the external differential, the graded vector space of $U(d)$-invariant, holomorphic translation invariant local functionals is equal to
\[
\label{bg def complex1}
\xymatrix{
\ul{-2d} & \cdots & \ul{-d-1} & \ul{-d} \\
\sO_{red}(B \fg) & \cdots & \Omega^{d-1} (B \fg) & \Omega^{d}(B \fg) .
}
\]
The differential is the restriction of the de Rham differential on the de Rham complex $\Omega^*(\CC^d, M)$ as we pointed out following Lemma \ref{lem: a8}. 
This is precisely the de Rham differential, as one can immediately verify, on $B \fg$ 
\[
\partial_{B \fg} : \Omega^k(B \fg) \to \Omega^{k+1} (B \fg)
\]
which completes are proof. 
\end{proof} 

\section{Normalizing the charge anomaly} \label{sec: feynman}

In this section we conclude the proof of Proposition \ref{prop: bg anomaly} by an explicit calculation of the Feynman diagrams controlling the charge anomaly for the $\beta\gamma$ system on $\CC^d$. 
We have already identified the algebraic piece of the anomaly with the $(d+1)$st component of the Chern character of the representation. 
The only thing left to compute is the analytic factor. 
We can therefore assume that we have an abelian Lie algebra, and simply compute the weight of the wheel $\Gamma$ with $(d+1)$-vertices where the external edges are labeled by elements $\alpha \in \Omega_c^{0,*}(\CC^d)$.
After choosing a numeration of the internal edges $e = 0,\ldots d$, we can label the edges $e = 0,\ldots, d-1$ by the analytic propagator by $P^{an}_{\epsilon<L}$ and the label the edge $e = d$ by the analytic heat kernel $K_\epsilon^{an}$. 
We recall the precise form of these kernels in the proof below. 
The vertices are labeled by the trivalent functional $I^{an} (\alpha, \beta,\gamma) = \int \alpha \wedge \beta \wedge \gamma$ (there is no Lie bracket since the algebra is abelian). 
Denote the resulting weight, which is a functional on the space $\Omega^{0,*}_c(\CC^d)$, by
\[
W^{an}_{\Gamma}(P_{\epsilon < L}, K_\epsilon, I^{an}) .
\]
The main computation left to do is the $\epsilon \to 0, L \to 0$ limit of this weight.

For more details on the notations, such as the explicit forms of the heat kernels and propagators, we use in the proof below we refer the reader to \cite{BWhol}, where the general prescription for quantizing holomorphic theories has been written down. 

\begin{lem} 
As a functional on the abelian dg Lie algebra $\Omega_c^{0,*}(\CC^d)$, one has
\[
\lim_{L \to 0} \lim_{\epsilon \to 0} W^{an}_{\Gamma}(P^{an}_{\epsilon < L}, K^{an}_\epsilon, I^{an})(\alpha^{(0)},\ldots, \alpha^{(d)}) = \frac{1}{(2 \pi i)^d} \frac{1}{(d+1)!} \int \alpha^{(0)} \partial \alpha^{(1)} \cdots \partial \alpha^{(d)}  .
\]
\end{lem}

\begin{proof}

We enumerate the vertices by integers $a = 0,\ldots, d$. 
Label the coordinate at the $i$th vertex by $z^{(a)} = (z_1^{(a)}, \ldots, z_d^{(a)})$. 
The incoming edges of the wheel will be denoted by homogeneous Dolbeault forms 
\[
\alpha^{(a)} = \sum_{J} A^{(a)}_J \d \zbar_J^{(a)} \in \Omega_c^{0,*}(\CC^d) .
\]
where the sum is over the multiindex $J = (j_1,\ldots, j_k)$ where $j_a = 1,\ldots, d$ and $(0,k)$ is the homogenous Dolbeault form type. 
For instance, if $\alpha$ is a $(0,2)$ form we would write
\[
\alpha = \sum_{j_1 < j_2} A_{(j_1,j_2)} \d \zbar_{j_1} \d\zbar_{j_2} .
\]
Denote by $W^{an}_L$ weight $\epsilon \to 0$ limit of the analytic weight of the wheel with $(d+1)$ vertices.
The $L\to 0$ limit of $W^{an}_L$ is the local functional representing the one-loop anomaly $\Theta$. 

The weight has the form
\[
W^{an}_L(\alpha^{(0)},\ldots,\alpha^{(d)}) = \lim_{\epsilon \to 0} \int_{\CC^{d(d+1)}} \left(\alpha^{(0)}(z^{(0)}) \cdots \alpha^{(d)}(z^{(d)}) \right) K^{an}_\epsilon(z^{(0)},z^{(d)}) \prod_{a =1}^d P^{an}_{\epsilon,L} (z^{(a-1)}, z^{(a)}) .
\]
We introduce coordinates
\begin{align*}
w^{(0)} & = z^{(0)} \\
w^{(a)} & = z^{(a)} - z^{(a-1)} \;\;\; 1 \leq a \leq d .
\end{align*}
The heat kernel and propagator part of the integral is of the form
\[
\begin{array}{ccl}
\displaystyle
K^{an}_\epsilon(w^{(0)},w^{(d)}) \prod_{a =1}^d P^{an}_{\epsilon,L} (w^{(a-1)}, w^{(a)}) & = & \displaystyle \frac{1}{(2 \pi i \epsilon)^d} \int_{t_1,\ldots,t_d = \epsilon}^L \frac{\d t_1 \cdots \d t_d}{(2 \pi i t_1)^d \cdots (2 \pi i t_d)^d} \frac{1}{t_1\cdots t_d}  \\ & & \displaystyle \times \d^d w^{(0)} \prod_{i=1}^d (\d \Bar{w}^{(1)}_i + \cdots + \d \Bar{w}^{(d)}_i) \\ & \times &  \displaystyle \prod_{a = 1}^d \d^d w^{(a)} \left(\sum_{i = 1}^d \Bar{w}_i^{(a)} \prod_{j \ne i} \d \Bar{w}_{j}^{(a)}\right) e^{-\sum_{a,b = 1}^d M_{a b} w^{(a)} \cdot \Bar{w}^{(b)}}
\end{array}
\]
Here, $M_{ab}$ is the $d \times d$ square matrix satisfying
\[
\sum_{a,b = 1}^d M_{a b} w^{(a)} \cdot \Bar{w}^{(b)} = |\sum_{a = 1}^d w^{(a)} |^2 / \epsilon + \sum_{a = 1}^d |w^{(a)}|^2 / t_a .
\]
Note that
\[
\prod_{i=1}^d (\d \Bar{w}^{(1)}_i + \cdots + \d \Bar{w}^{(d)}_i) \prod_{a = 1}^d \left(\sum_{i = 1}^d \Bar{w}_i^{(a)} \prod_{j \ne i} \d \Bar{w}_{j}^{(a)}\right) = \left( \sum_{i_1,\ldots i_d} \epsilon_{i_1\cdots i_d} \prod_{a=1}^d \Bar{w}^{(a)}_{i_a}\right) \prod_{a=1}^d \d^d \Bar{w}^{(a)} .
\]
In particular, only the $\d w_i^{(0)}$ components of $\alpha^{(0)} \cdots \alpha^{(d)}$ can contribute to the weight.

For some compactly supported function $\Phi$ we can write the weight as
\[
\begin{array}{ccl}
W (\alpha^{(0)}, \ldots, \alpha^{(d)}) & = & \lim_{\epsilon \to 0} \displaystyle \int_{\CC^{d(d+1)}} \left(\prod_{a = 0}^{d} \d^d w^{(a)} \d^d \Bar{w}^{(a)}\right) \Phi \\ & \times & \displaystyle \frac{1}{(2 \pi i \epsilon)^d} \int_{t_1,\ldots,t_d = \epsilon}^L \frac{\d t_1 \cdots \d t_d}{(2 \pi i t_1)^d \cdots (2 \pi i t_d)^d} \frac{1}{t_1\cdots t_d} \\ & \times & \displaystyle \sum_{i_1,\ldots, i_d} \epsilon_{i_1\cdots i_d} \Bar{w}_{i_1}^{(1)} \cdots \Bar{w}_{i_d}^{(d)} e^{-\sum_{a,b = 1}^d M_{a b} w^{(a)} \cdot \Bar{w}^{(b)}} 
\end{array}
\]

Applying Wick's lemma in the variables $w^{(1)}, \ldots, w^{(d)}$, together with some elementary analytic bounds, we find that the weight above becomes to the following integral over $\CC^d$
\[
f(L) \int_{w^{(0)} \in \CC^d}  \d^d w^{(0)} \d^d \Bar{w}^{(0)} \sum_{i_1,\ldots, i_d} \epsilon_{i_1\cdots i_d}  
\left(\frac{\partial}{\partial w_{i_1}^{(1)}} \cdots \frac{\partial}{\partial w_{i_d}^{(d)}} \Phi\right)|_{w^{(1)}=\cdots=w^{(d)} = 0} 
\]
where
\[
f(L) = \frac{1}{(2 \pi i)^d} \lim_{\epsilon \to 0} \int_{t_1,\ldots,t_d = \epsilon}^L \frac{\epsilon}{(\epsilon + t_1 + \cdots + t_d)^{d+1}} \d^d t .
\]
In fact, $f(L)$ is independent of $L$ and is equal to $\frac{1}{(d+1)!}$ after direct computation. 
Finally, plugging in the forms $\alpha^{(0)}, \ldots, \alpha^{(d)}$, we observe that the integral over $w^{(0)} \in \CC^d$ simplifies to
\[
\frac{1}{(2 \pi i)^d} \frac{1}{(d+1)!} \int_{\CC^d} \alpha^{(0)} \partial \alpha^{(1)} \cdots\partial \alpha^{(d)}
\]
as desired.
\end{proof}

\bibliographystyle{alpha}
\bibliography{hic}

\end{document}